\documentclass{amsart}

\usepackage{amssymb}
\usepackage{amsmath}
\usepackage{amsfonts}
\usepackage[dvips]{graphicx}
\usepackage{psfrag,euscript}

\usepackage{a4wide}

\newtheorem{theorem}{Theorem}
\theoremstyle{plain}
\newtheorem{maintheorem}{Theorem}

\newtheorem{corollary}{Corollary}

\newtheorem{definition}{Definition}

\newtheorem{lemma}{Lemma}

\newtheorem{proposition}{Proposition}
\newtheorem{remark}{Remark}

\numberwithin{equation}{section}

\newcommand{\RR}{{\mathbb R}}

\newcommand{\ZZ}{{\mathbb Z}}
\newcommand{\DD}{{\mathbb D}}
\newcommand{\real}{{\mathbb R}}

\newcommand{\CC}{{\mathbb C}}

\newcommand{\de}{\delta}

\newcommand{\cal}{\mathcal}

\renewcommand{\epsilon}{\varepsilon}

\newcommand{\dist}{\operatorname{dist}}
\newcommand{\diam}{\operatorname{diam}}
\newcommand{\lip}{\operatorname{Lip}}

\newcommand{\interior}{\operatorname{int}}

\newcommand{\varsq}{\operatorname{var^\square}}
\newcommand{\cl}{\operatorname{Cl}}

\newcommand{\esssup}{\operatorname{ess\,sup}}

\newcommand{\re}{{\mathbb R}}

\newcommand{\cC}{\EuScript{C}}

\newcommand{\cF}{\EuScript{F}}

\newcommand{\F}{\EuScript{F}}

\newcommand{\U}{\EuScript{U}}
\newcommand{\cP}{\EuScript{P}}

\newcommand{\inter}{\operatorname{int}}

\newcommand{\supp}{\operatorname{supp}}

\newcommand{\vfi}{\varphi}

\newcommand{\osc}{\operatorname{osc}}

\def \CC {{\mathbb C}}
\def \DD {{\mathbb D}}

\def \II {{\mathbb I}}

\def \PP {{\mathbb P}}

\def \RR {{\mathbb R}}

\def \ZZ {{\mathbb Z}}

\newcommand{\qand}{\quad\text{and}\quad}

\def \cA {{\mathcal A}}

\def \cC {{\mathcal C}}
\def \cD {{\mathcal D}}

\def \cF {{\mathcal F}}

\def \cP {{\mathcal P}}

\def \cV {{\mathcal V}}

\def \fX {{\mathfrak X}}

\def \fS {{\mathfrak S}}

\begin{document}
\title[Correlations $\&$ logarithm law for singular
hyperbolic flows.]{Decay of correlations for maps with
  uniformly contracting fibers and logarithm law for
  singular hyperbolic attractors.}

\author{Vitor Araujo}
\address[V.A.]{Instituto de Matem\'atica,
Universidade Federal da Bahia\\
Av. Adhemar de Barros, S/N , Ondina,
40170-110 - Salvador-BA-Brazil}
\thanks{V.A. and M.J.P. were partially supported by CNPq,
  PRONEX-Dyn.Syst., FAPERJ, Balzan Research Project of J.Palis }
\email{vitor.d.araujo@ufba.br}
\author{Stefano Galatolo}
\address[S.G.]{Dipartimento di Matematica Applicata, Via Buonarroti 1 Pisa}
\email[S. Galatolo]{galatolo@dm.unipi.it}
\urladdr{http://users.dma.unipi.it/galatolo/}
\author{Maria Jose Pacifico}
\address[M.J.]{Instituto de Matem\'atica,
Universidade Federal do Rio de Janeiro\\
C. P. 68.530, 21.945-970 Rio de Janeiro}
\email{pacifico@im.ufrj.br}

\begin{abstract}
  We consider  maps preserving a foliation
  which is uniformly contracting and a one-dimensional
  associated quotient map having exponential convergence to
  equilibrium (iterates of Lebesgue measure converge exponentially fast
  to physical measure). 
  We prove that these maps have exponential decay of
  correlations over a large class of observables.

  We use this result to deduce exponential decay of
  correlations for suitable Poincar\'{e} maps of a large class of singular
  hyperbolic flows. From this we deduce a logarithm law for
  these flows.

\end{abstract}

\renewcommand{\subjclassname}{\textup{2010} Mathematics
  Subject Classification}

\date{\today}

\maketitle

\tableofcontents

\section{Introduction}

The term \emph{statistical properties} of a dynamical system
$F:X\to X$, where $X$ is a measurable space and $F$ a
measurable map, refers to the behavior of large sets of trajectories of the system. It is well known that
this relates to the properties of the evolution of measures
by the dynamics.

Statistical properties are often a better object to be
studied than pointwise behavior. In fact, the future
behavior of initial data can be unpredictable, but
statistical properties are often regular and their
description simpler.  Suitable results can be established in
many cases, to relate the evolution of measures with that of
large sets of points (ergodic theorems, large deviations,
central limit, logarithm law, etc...).

In this paper we take the point of view of studying the
evolution of measures and its speed of convergence to
equilibrium to understand the statistical properties of a
class of dynamical systems.  We consider fiber-contracting
maps which are a skew-product whose base transformation has
exponential convergence to equilibrium (a measure of how
fast iterates of the Lesbegue measure converge to the
physical measure) and deduce that the whole system has
exponential decay of correlations.  By this it is possible to obtain  
other consequences regarding the statistical behavior of the dynamics, 
as the logarithm law, which will be introduced below.
  The regularity required by the techiniques we
use to prove decay of correlations is quite low, so we can
apply it to a class of systems 
including Poincar\'e return maps of Lorenz like systems.

We apply indeed these results to suitable Poincar\'e maps of an
open dense subset of a large class of attractors of
three-dimensional partially hyperbolic flows, including the
\emph{singular-hyperbolic (or Lorenz-like) attractors},
introduced in \cite{MPP04}.  These last kind of attractors
present equilibria accumulated by regular orbits, and the
Lorenz attractor as well the geometric Lorenz attractor are
the most studied examples of this class of attractors
\cite{Lo63,GW79}.

The main feature of these systems we exploit to estimate the
decay of correlations is the existence of a quite regular
stable foliation.  We can disintegrate a measure along the
foliation, and estimate the convergence of its iterates
separately, along the stable and along the unstable
direction, also by the use of suitable anysotropic norms.

  Let us mention that it was proved in
\cite{MP03} that a generic $C^1$ vector field on a closed
$3$-manifold either has infinitely many sinks or sources, or
else its non-wandering set admits a finite decomposition
into compact invariant sets, each one being either a
uniformly hyperbolic set or a singular-hyperbolic attractor
or a singular-hyperbolic repeller (which is a
singular-hyperbolic attractor for the time reversed
flow). This shows that the class of singular-hyperbolic
attractors is a good representative of the limit dynamics of
many flows on three-dimensional manifolds.  More important,
\emph{the class of singular-hyperbolic attractors contains
  every $C^1$ robustly transitive isolated set} for flows on
compact three-manifolds; see \cite{MPP04,AraPac2010}.

We obtain exponential decay of correlations for the
Poincar\'e return maps of a well-chosen finite family of
cross-sections for these flows.  In  \cite{galapacif09},
exponential decay of correlations and limit laws for the
two-dimensional Poincar\'e first return map of a class of
geometrical Lorenz attractors were proved. We now provide
similar results in a much more general context. We remark that our results
on decay of correlations are finer that the one of
Bunimovich \cite{Bu83} and Afraimovich, Chernov, Sataev
\cite{ACS95}, which can be seen as developments on the work
of Pesin~\cite{Pe92}.

From the decay of correlations we then
deduce a logarithm law for the hitting times to small
targets in these cross-sections. Then, exploiting the fact
that {\em the flow }on these attractors can be seen as a
suspension flow over Poincar\'e return maps, we also obtain
a logarithm law for the hitting time for the singular
hyperbolic flow itself. We observe that quantitative
dynamical properties of this class of systems are still
quite unknown.

A logarithm law is a statement on the time a typical orbit
hits a sequence of decreasing targets which is strictly
related to the so called dynamical Borel Cantelli lemma and
to diophantine approximation (see \cite{galapeter09} and
\cite{GalKim07} for relations with the Borel Cantelli lemma
and other approaches).  Roughly speaking, a system has a
logarithm law for a decreasing sequence of targets $S_i$ if
the time which is needed for a typical orbit to hit the
$i$-target is in some sense inversely proportional to the
measure of the target. Hence it is a quantitative statement
indicating how fast a typical orbit "fills" the space.
Hitting time results of this kind have been proved in many
continuous time dynamical systems of geometrical interest:
geodesic flows, unipotent flows, homogeneous spaces etc.;
see e.g. \cite{AthMar09,HilVel95,KlMar99,sull82,GalNis11,GS}.
Other examples of connections with diophantine approximation
can be found in \cite{galapeter09} and \cite{KiMar08}. We
also remark that, in the symbolic setting, similar results
about the hitting time are also used in information theory
(see e.g. \cite{Shld93,Kontoy98}).

The logarithm law automatically holds for fastly mixing
systems (see e.g, Theorem \ref{maine}), but there are mixing
systems for which the law does not hold
(\cite{galapeter09}).  As it will be mentioned in the next
section, some subclass of singular hyperbolic flows has been
proved to be fastly mixing but in general the speed of
mixing (and several statistical properties) of such flows
are still unknown.

The referee of the current paper has pointed out to us that,
in the case of Poincar\'e maps for geometric Lorenz
attractors, as considered in \cite{galapacif09}, exponential
decay of correlations can be also obtained as an application
of the landmark work of Young~\cite{Yo98}.  Indeed, results
on the statistical behavior of more or less general classes
of Lorenz like flows have also been obtained by Markov tower
like constructions; see e.g. \cite{HoMel}.  The approach we
present here to estimate decay of correlations decouples the
statistical properties of the base map of a skew-product
with contracting fibers from the rest of the estimation
about the speed of mixing.  In the application to singular
hyperbolic flows, we use the functional analytic approach to
estimate the convergence to equilibrium for the base map (a
piecewise expanding map), but in other cases also tower like
constructions could be used in this step.  We also remark
that we need very weak regularity assumptions on the base
map apart from the speed of convergence to equilibrium; see
Section~\ref{sec:primeiroretorno} and
Remark~\ref{rmk:slowconverg}. This will be essential when
studying higher-dimensional partially hyperbolic systems,
like the sectional-hyperbolic attractors introduced by
Bonatti, Pumarino, Viana in \cite{BPV97} and Metzger-Morales
in \cite{MeMor06}, where the quotient map over the
contracting fibers in general fails even to have a
derivative; it should be not more than a locally Holder
homeomorphism.

\subsection{Statement of the results}
 \label{sec:statement-results}

 In Section \ref{2p1} we establish results on the decay of
 correlations and convergence to equilibrium for fiber
 contracting maps with a base transformation which rapidly
 converges to equilibrium. Those results imply the following
 statement (which in turn will be applied to singular
 hyperbolic attractors).

 We denote by $Q=\II\times\II$ the unit square, where
 $\II=[0,1]$. For a function $g:Q\to\RR$ we denote by $L(g)$
 the best Lipschitz constant of $g$, that is,
 $L(g)=\sup_{p,q\in Q}\frac{|g(p)-g(q)|}{|p-q|}$ where
 $|\cdot|$ is the Euclidean distance. We define the
 Lipschitz norm by setting $\Vert g\Vert_{lip}=\Vert g\Vert
 _{\infty }+L(g)$ where, as usual,
 $\|g\|_\infty=\esssup_{p\in Q}|g(p)|$ and set
 $\textrm{Lip}(Q)=\{g:Q\to\RR:\|g\|_{lip}<\infty\}$.

\begin{maintheorem}\label{mthm:expdecay}
  Let us consider a map $F:Q\circlearrowleft$ from the unit
  square into itself such that:
\begin{enumerate}
\item $F$ has the form $F(x,y)=(T(x),G(x,y))$ (is a
  skew-product and preserves the natural vertical foliation
  of the square) ;
\item $F|_{\gamma }$ is $\lambda $-Lipschitz with $\lambda
  <1$ (hence is uniformly contracting) on each leaf $\gamma
  $ of the vertical foliation of the square;
\item $var^{\square }(G)<\infty $ (a kind of variation in
  one direction; see beginning of Section
  \ref{sec:p-bounded-variation} for the definition),
\item $T:\II\circlearrowleft$ is piecewise monotonic, with
  $n+1$, $C^{1}$ increasing branches on the intervals
  $(0,c_{1})$,...,$%
  (c_{i},c_{i+1})$ ,..., $(c_{n},1)$ and $\inf_{x\in\II}
  |T^{\prime }(x)|>1$.
\item $\frac{1}{T^{\prime }}$ has finite universal
  $p-$bounded variation (a generalization of the notion of
  bounded variation; see again Section
  \ref{sec:p-bounded-variation} for the definition);
\item $T$ has only one absolutely continuous
  (w.r.t. Lebesgue on $\II$) invariant probability measure
  (a.c.i.m.) for which it is weakly mixing.
\end{enumerate}
Then {\em the unique physical measure} $\mu _{F}$ of $F$ has
exponential decay of correlation with respect to Lipschitz
observables, that is, there are $C,\Lambda \in
\mathbb{R}^{+},~\Lambda <1$, such that
\begin{align*}
  \left|\int f\cdot(g\circ F^{n})\,d\mu _{F} -
    \int g~d\mu _{F} \int f~d\mu _{F} \right|\leq
  C\Lambda^{n}
  \|g\|_{Lip}\|f\|_{Lip}, \quad f,g\in\textrm{Lip}(Q).
\end{align*}
\end{maintheorem}

The notion of physical measure is central in smooth
Ergodic Theory. We say that a $F$-invariant probability
measure $\mu_F$ is \emph{physical} if the \emph{ergodic
  basin} of $\mu_F$
\begin{align*}
  B(\mu_F)=\left\{p\in Q:
  \lim_{n\to+\infty}\frac1n\sum_{j=0}^{n-1}\vfi(F^j(p))=\int\vfi\,d\mu_F,
\quad\text{for all continuous} \quad \vfi:Q\to\RR\right\}
\end{align*}
has positive area in $Q$ (two-dimensional Lebesgue measure).

\begin{remark}
We remark that items (4) to (6) of the assumptions on the
above theorem can be replaced by (much more general) exponential
convergence to equilibrium on the base map under suitable
observables: see Theorem \ref{resuno} and
Theorem~\ref{thm:summary-exp-decay} in Section~\ref{2p1}. We
also remark that Lipschitz norms can be replaced by suitable
weaker anisotropic norms.
\end{remark}

Next, we establish a logarithm law for the dynamics of the
flow.  This is a relation between hitting time to small
targets and the local dimension of the invariant measure we
consider.  Let us consider a family of target sets
$B_{r}(x_0)$, where $x_0$ is a given point, indexed by a
real parameter $r$, and let us denote the time needed for
the orbit of a point $x$ to enter in $B_{r}(x_0)$ by
\begin{equation*}
\tau _{F}(x,B_{r}(x_0))=\min \{n\in \mathbb{N}^{+}:F^{n}(x)\in B_{r}(x_0)\}.
\end{equation*}
A logarithm law is the statement that as $r\to0 $ the
hitting time scales like $1/\mu(B_r )$.

To express this, let us consider the local dimensions of a measure $\mu$
\begin{equation}\label{eq:localdim}
  \overline{d}_{\mu}(x_0)=
  \underset{r\rightarrow 0}{\lim \sup }\frac{\log \mu (B_{r}(x_0))}{%
    \log (r)}
  \quad\text{and}\quad 
  \underline{d}_{\mu}(x_0)=\underset{r\rightarrow 0}{\lim \inf }\frac{\log
    \mu (B_{r}(x_0))}{\log (r)}
\end{equation}
representing the scaling rate of the measure of small balls
as the radius goes to $0$.  When the above limits coincide
for $\mu $-almost every point, we say that $\mu $ is
\emph{exact dimensional} and set $d_{\mu
}=\underline{d}_{\mu }(x)=\overline{d}_{\mu }(x)$.  From the
main results of \cite{galatolo10,galat07,Stein00} it follows
that in the kind of systems mentioned above (under mild
extra conditions on the second coordinate $G$ of $F$)
$\mu_F$ is exact dimensional, and the logarithm law holds.

\begin{proposition}
  \label{cor:hitting-law}
  If $F(x,y)=(T(x),G(x,y))$ is as in the setting of
  Theorem~\ref{mthm:expdecay} and
  \begin{itemize}
  \item $F$ is injective;
  \item $G$ is $C^1$ on $J\times\II$, where $J$ is any
    monotonicity interval of $T$;
  \item $|\partial G/\partial y|>0$ whenever defined,
    $\sup |\partial G / \partial x| <\infty$ and
    $\sup|\partial G / \partial y|<1$;
  \end{itemize}
  then the physical measure $\mu_F$ is exact dimensional and for $\mu_F$-almost
  every $x$ it holds
\begin{align*}
\lim_{r\rightarrow 0}\frac{\log \tau _{F}(x,B_{r}(x_0))}{-\log r}=d_{\mu_F}(x_0).
\end{align*}
\end{proposition}

We now apply these results to three-dimensional flows having
a singular-hyperbolic attractor.  Let $SH^2(M^3)$ be the
family of all $C^2$ vector fields $X$ on a compact
three-manifold $M^3$ having an open trapping region $U$,
i.e., $\overline{X_t(U)}\subset U$ for all $t>0$, such that
its maximal invariant subset $\Lambda=\cap_{t>0}X_t(U)$ is a
compact transitive singular-hyperbolic set.

This means that, for each $X\in SH^2(M^3)$, there exists a
continuous invariant and dominated decomposition
$E^s_\Lambda\oplus E^c_\Lambda$ of the tangent space at each
$x\in\Lambda$, where $E^s_\Lambda$ is a one-dimensional
sub-bundle uniformly contracted by the derivative $DX_t$,
and $E^c_\Lambda$ is a two-dimensional sub-bundle,
containing the flow direction, and whose area is uniformly
expanded by the action of $DX_t$, that is, there are
constants $C,\lambda>0$ such that $|\det DX_t\mid E^c_x| \ge
C e^{\lambda t}$. In addition, there are finitely many
hyperbolic singularities $\sigma$ in $\Lambda$, which are
\emph{Lorenz-like}, i.e., the eigenvalues
$\lambda_i(\sigma)$, $1\leq i \leq 3$, of $DX(\sigma)$ are
real and satisfy
$\lambda_2(\sigma)<\lambda_3(\sigma)<0<-\lambda_3(\sigma)
<\lambda_1(\sigma)$ (see Section \ref{sec:Lorenzmodel} for a
more precise description).  We consider the $C^2$ topology
of vector fields in $SH^2(M^3)$ in what follows.

Given a flow having a singular hyperbolic attractor, there
exists a finite family $\Xi$ of well-adapted cross-sections
of the flow where we can define a Poincar\'e return map $F$
which satisfies the properties in the statement of
Theorem~\ref{mthm:expdecay} after a suitable choice of
coordinates.  To take advantage of the above cited result
from Steinberger \cite{Stein00} on exact dimensionality and
obtain Proposition \ref{cor:hitting-law}, we need that the
Poincar\'e return map $F$ be injective, which is not evident
in the construction and choice of these cross-sections done
at \cite{APPV}. Because of this, the construction presented
here is slightly different from the ones presented
elsewhere.

Once constructed a suitable section,  taking some  nonresonance conditions 
on the eigenvalues of
the equilibria of $X$ inside $\Lambda$ we can reduce the
dynamics on $\Lambda$ to a map $F$ as in
Theorem~\ref{mthm:expdecay}, and obtain
\begin{corollary}
  \label{cor:decaycorr-sing-hyp}
  There exists an open dense set $\cA$ of vector fields (satisfying a nonresonance condition) in
  $SH^2(M^3)$ such that, for each $X\in\cA$, we can find a
  finite family $\Xi$ of cross-sections to the flow $X_t$ of
  $X$ such that an iterate of the Poincar\'e first return map
  $F:\textrm{dom}(F)\subset\Xi\to\Xi$ has a finite set of SRB measures $\mu
 ^{i} _{F}$, each of them has exponential decay of correlations with
  respect to Lipschitz observables: there are $C,\Lambda \in
  \mathbb{R}^{+},~\Lambda <1$ satisfying for every pair
  $f,g:\Xi\to\RR$ of Lipschitz functions
\begin{equation*}
  \left|\int f\cdot(g\circ F^{n})~d\mu ^{i} _{F} -\int g~d\mu ^{i} _{F} \int
    f~d\mu ^{i} _{F} \right|
  \leq C\Lambda
  ^{n}||g||_{Lip}||f||_{Lip}, \quad n\ge1.
\end{equation*}
\end{corollary}

We remark that, since the work of Ruelle~\cite{ruelle1983},
it is well-known that exponentially mixing for the base
transformation of a suspension flow does not imply fast
mixing for the suspension flow. In fact, the suspension flow
might be mixing but without exponential decay of
correlations as in ~\cite{ruelle1983}; or it may have
arbitrarily slow decay of correlations, as shown by
Pollicott in \cite{pol85}. 
Hence we cannot deduce in general any kind of fast mixing
results for the flow on a singular-hyperbolic attractor from
Corollary~\ref{cor:decaycorr-sing-hyp}.  However, in a
recent work of one of the authors with Varandas
\cite{ArVar}, it has been proved the existence of a $C^2$
open subset of vector fields having a geometrical Lorenz
attractor with exponential decay of correlations for the
flow on $C^1$ observables, from which follows the
exponential decay of correlations for the corresponding
Poincar\'e map. But this $C^2$ open subset was obtained
under very strong conditions which cannot hold in such
generality as in Corollary~\ref{cor:decaycorr-sing-hyp}.

Under mild (open and dense) conditions on the eigenvalues of
the hyperbolic singularities of $\Lambda$, we can show that
the second coordinate $G$ of $F$ is in the setting of
Proposition~\ref{cor:hitting-law}. Hence we can also deduce
the logarithm law for the Poincar\'e return map; see Section
\ref{sec:logarithm-law-syngul}.

%As a corollary, we obtain a logarithm law for hitting times
%%i%n singular hyperbolic attractors.
%*****************************
%*****************************
%\begin{corollary}
  %\label{cor:hitting-sing-hyp}
  %In the conditions of
  %Corollary~\ref{cor:decaycorr-sing-hyp} if, in addition,
  %the eigenvalues of every equilibrium point $\sigma$ of $X$
  %in $\Lambda$ satisfy
  %$\lambda_1(\sigma)+\lambda_2(\sigma)<0$ (which includes the classical Lorenz like systems) then $\mu_F$ is
  %exact and, for $\mu_F$ almost every point $x_0$ at the
  %finite family $\Xi$ of adapted cross-sections $X$
  %associated to $\Lambda$, it holds
  %\begin{align*}
    %\lim_{r\rightarrow 0}\frac{\log \tau
    %  _{F}(x,B_{r}(x_0))}{-\log r}=d_{\mu_F}(x_0).
%\end{align*}
%\end{corollary}

We can then use this to obtain a log-law for the
hitting times for the singular-hyperbolic flow, in a way similar to what was done in \cite{galapacif09}.  Let $x, x_{0}\in
\Lambda$ and
\begin{equation}\label{circlen}
\tau _{r}^{X_t}(x,x_{0})=\inf \{t\geq 0 | X_{t}(x)\in B_r(x_{0})\}
\end{equation}
be the time needed for the $X$-orbit of a point $x$ to enter
for the {\em first time} in a ball $B_{r}(x_{0})$. The
number $\tau_{r}^{X_t}(x,x_{0})$ is the \emph{hitting time}
associated to the flow $X_t$ and target $B_r(x_0)$. Let us
consider $d_{\mu_X}(x_0)$, the local dimension at $x_0$ of
the unique physical measure $\mu_X$ supported on the
singular-hyperbolic attractor (which was constructed in
\cite{APPV}).

\begin{corollary}\label{cor:hitting-flow}
  If $X_t$ is a flow over a singular-hyperbolic attractor in
  the setting of Corollary~\ref{cor:decaycorr-sing-hyp} and
 if, in addition,  the eigenvalues of every equilibrium point $\sigma$ of 
$X$ in $\Lambda$ satisfy
  $\lambda_1(\sigma)+\lambda_2(\sigma)<0$ (which includes the 
  classical Lorenz system of ODEs), then for each
  regular point $x_{0}\in \Lambda$ such that
  $d_{\mu_X}(x_{0})$ exists, we have
\begin{equation*}
\lim_{r\rightarrow 0}\frac{\log \tau_{r}^{X_t}(x,x_{0})}{-\log r}=d_{\mu_X}(x_{0})-1
\end{equation*}
for $\mu_X$-almost each $x\in \Lambda.$
\end{corollary}

\subsection{Organization of the text}
\label{sec:organiz-text}

We begin by proving exponential decay of correlations for
fiber-contracting maps in Section~\ref{2p1}.  We use the
exponential convergence to equilibrium of the base
transformation to prove exponential convergence to
equilibrium of the fibered contracting map. For this we
compare the disintegrations of the push-forward measures
along the fibers using the Wasserstein-Kantorovich distance
between measures.  We choose appropriate norms to be able to
compare the size of the correlation functions and deduce the
exponential decay of correlations from the exponential
convergence to equilibrium.  To help the reader follow the
diverse norms used throughout we have compiled a list of
norms in the last Section~\ref{sec:notati-norms-used}.

In Section~\ref{sec:p-bounded-variation},
we explain the meaning of generalized $p$-bounded variation
and the existence of absolutely continuous invariant
probability measures for $C^{1+\alpha}$ piecewise expanding
maps (not necessarily Markov maps) whose densities are of
generalized $p$-bounded variation. This is based on the work
of Keller in \cite{Ke85}.

We present singular-hyperbolic attractors and the related properties which
we use in the remaining part of the paper in
Section~\ref{sec:Lorenzmodel}, mainly following the works
\cite{APPV} and \cite{AraPac2010}.

In Section~\ref{sec:decay-correl-poincar}, we analyze the
Poincar\'e return map associated to a singular-hyperbolic
attractor and show that this map satisfies the conditions
needed to deduce exponential decay of correlations for
Lipschitz observables. Here we take advantage of the
constructions presented in the previous section.

 In Section~\ref{sec:muexata} we also  show that the physical measure of the flow is
exact dimensional. This is done using the results of \cite{Stein00}.

Following \cite{galat07,galatolo10,galapacif09}, in
Section~\ref{sec:logarithm-law-syngul} we use the results proved in the previous sections to establish the
logarithm law for the hitting times for  the Poincar\'e map
and flow of a singular-hyperbolic attractor with mild
conditions on the eigenvalues of its hyperbolic
singularities.

\bigskip\noindent \textbf{Acknowledgments:} S. G.  wishes to
thank IMPA, PUC, and UFRJ (Rio de Janeiro), where a part of
this work has been done, for their warm hospitality. All
authors whish to thank Carlangelo Liverani for illuminating
discussions about convergence to equilibrium and decay of
correlations.

%%%%%%%%%%%%%%%%%%%%%%%%%%%%%%%%%%%%%%%%%%%%%%%%%%%%%%%%%%%%%%%%

\section{Decay of correlations for fiber contracting
  maps\label{2p1}}

In this section we will prove that a certain class of maps,
with contracting fibers and low regularity have exponential
decay of correlation. This class contains suitable
Poincar\'{e} maps of singular-hyperbolic attractors as we
will se in what follows.

The methods we use are related both to coupling and to
spectral techniques, and will be implemented by adequate
anisotropic norms.

We will quantify the speed of convergence of the iterates of
the Lesbegue measure to the physical invariant measure
through the help of a certain distance on the space of
probability measures.

\subsection{The Wasserstein-Kantorovich
  distance\label{sec:W-Kdistance}}

If $Y$ is a bounded metric space, we denote by $PM(Y)$ the
set of Borel probability measures on $Y$. We denote by
$L(g)$ the best Lipschitz constant of $g:Y\to\RR$, that is,
$L(g)=\sup_{x,y}\frac{|g(x)-g(y)|}{|x-y|}$ and set $\Vert
g\Vert_{lip}=\Vert g\Vert _{\infty }+L(g).$

Let us consider the following notion of distance between measures: given two
probability measures $\mu _{1}$ and $\mu _{2}$ on $Y$

\begin{equation*}
W_{1}(\mu _{1},\mu _{2})=\sup\left\{ \left|\int_{Y}g~d\mu
_{1}-\int_{Y}g~d\mu _{2}\right| : g:Y\to\RR, L(g)=1 \right\}.
\end{equation*}
We remark that adding a constant to the test function $g$ does not change
the above difference of integrals.

We also recall that if $F:X\rightarrow Y$ is a map, then it induces an
associated map $F^{\ast }:PM(X)\rightarrow PM(Y)$ defined as $(F^{\ast }(\mu
))(A)=\mu (F^{-1}(A))$.

The above defined distance has the following basic properties; see e.g \cite[%
Prop 7.1.5]{AmbGigSav08} and \cite{galapacif09}.

\begin{proposition}\label{p2}
\label{pr:propert-WKdist}\label{remark} $W_{1}$ is a distance and if $Y$ is
separable and complete, then $PM(Y) $ with this distance becomes a separable
and complete metric space. A sequence is convergent for the $W_{1} $ metrics
if and only if it is convergent for the weak-star topology. Moreover, let $%
F:\gamma \rightarrow \gamma $ be a $\lambda $-contracting map, let us
consider two probability measures $\mu $, $\nu $ on $\gamma $. Then 
\begin{equation*}
W_{1}(F^{\ast }(\mu ),F^{\ast }(\nu ))\leq \lambda \cdot W_{1}(\mu ,\nu ).
\end{equation*}
\end{proposition}

\begin{remark}\label{r1}
\label{conv}(distance and convex combinations) If $\sum_{1}^{n}a_{i}=1,a_{i}
\geq 0$, $\mu _{i},\nu _{i}\in PM(Y)$ for $i=1,\dots,n$, then 
\begin{equation}
W_{1}(\sum_{1}^{n}a_{i}\mu _{i},\sum_{1}^{n}a_{i}\nu _{i})\leq
\sum_{1}^{n}a_{i} \cdot W_{1}(\mu _{i},\nu _{i}).
\end{equation}
\end{remark}

\begin{remark}
\label{lipW}If $g$ is $\ell $-Lipschitz and $\mu _{1},\mu _{2}$ are
probability measures then 
\begin{equation*}
\left|\int_{Y}g~d\mu _{1}-\int_{Y}g~d\mu _{2}\right| \leq \ell \cdot
W_{1}(\mu _{1},\mu _{2}).
\end{equation*}
\end{remark}

\begin{remark}
We will see that the decay of correlation of the system is related to the
speed of convergence of iterates of a starting measure to the invariant one.
In \cite{galapacif09} the rate of convergence with respect to the $W_{1}$
distance was considered to obtain exponential decay of correlation for the
Poincar\'{e} map of a geometric Lorenz flow. Here, being in a more general
case, having less regularity, we cannot perform the same construction. But
this idea will be part of the construction we are going to implement.
\end{remark}

We will also use the following variation distance on the space of
probability measures 
\begin{equation*}
V(\mu _{1},\mu _{2})=\sup \left\{ \left|\int hd\mu_{1}-\int hd\mu
_{2}\right| : \|h\|_{\infty }\leq 1\right\}.
\end{equation*}
We remark that when the two measures have a density, this is the $L^1$
distance between the densities.

\subsection{Distance and disintegration}

\label{sec:notations}

We denote by $\II=[0,1]$ and $Q=\II\times\II$. We consider the $\sup $
distance on $Q$ so that the diameter is one: $diam(Q )=1$. This choice is
not essential, but will avoid the presence of many multiplicative constants
in the following, making notations cleaner.

The square $Q$ will be foliated by stable, vertical leaves. We will denote
the leaf with $x$ coordinate by $\gamma _{x}$ or, with a small abuse of
notation, when no confusion is possible, we will denote both the leaf and
its coordinate by $\gamma$.

Let $f\mu$ be the measure $\mu _{1}$ such that $d\mu _{1}=fd\mu $. Let $\mu $
be a probability measure on $Q$. Such measures on $Q$ will be often
disintegrated in the following way: for each Borel set $A$ 
\begin{equation}
\mu (A)=\int_{\gamma \in I}\mu _{\gamma }(A\cap \gamma )d\mu_{x}  \label{dis}
\end{equation}
with $\mu_{\gamma }$ being probability measures on the leaves $\gamma $ and $%
\mu_{x}$ the marginal on the $x$ axis, which will be absolutely continuous
with respect to the Lebesgue length measure. We denote by $\phi_{x}$ the
density of $\mu_x$.

% \subsection{Disintegration and distance}
% \label{sec:disint-distance}

% Since we consider maps having an invariant foliation, the
% invariant measure will be disintegrated as in Equation
% (\ref{dis}) into a family of measures $\mu _{\gamma }$ on
% almost each stable leaf $\gamma $ and an absolutely
% continuous measure $\mu _{x}$ on the unstable direction.

Now we consider the integral of a suitable observable $g:Q\to\RR$. Let us
consider the following anisotropic norm, considering Lipschitz regularity
only on the vertical direction. Let $\|\cdot\|_{\updownarrow lip}$ be
defined by 
\begin{equation}  \label{eq:norm-lip}
\|g\|_{\updownarrow lip}=\|g\|_{\infty !}+ \lip_y(g),
\end{equation}
where for simplicity we consider 
\begin{equation}  \label{eq:lip_y}
\|g\|_{\infty !}:=\underset{x,y\in \lbrack 0,1]}{\sup }|g(x,y)| \quad\text{%
and}\quad \lip_y(g):=\sup_{\substack{ x,y_{1},y_{2}\in \lbrack 0,1]  \\ %
y_1\neq y_2}} \frac{|g(x,y_{2})-g(x,y_{1})|}{|y_{2}-y_{1}|}.
\end{equation}
If $\mu ^{1}$ and $\mu ^{2}$ are two disintegrated measures as above, the
integral of an observable $g$ can be estimated as function of the above
norm, some distance between their respective marginals on the $x$ axis and
measures on the leaves.

\begin{proposition}
\label{prod}Let $\mu ^{1}$, $\mu ^{2}$ be measures on $Q$ as above, such
that for each Borel set $A$ 
\begin{equation*}
\mu ^{1}(A)=\int_{\gamma \in I}\mu _{\gamma }^{1}(A\cap \gamma )d\mu _{x}^{1}%
\qand\mu ^{2}(A)=\int_{\gamma \in I}\mu _{\gamma }^{2}(A\cap \gamma )d\mu
_{x}^{2},
\end{equation*}%
where $\mu _{x}^{i}$ is absolutely continuous with respect to the Lebesgue
measure. In addition, let us assume that

\begin{enumerate}
\item $\int_{I}W_{1}(\mu _{\gamma }^{1},\mu _{\gamma }^{2})d\mu _{x}^{1}\leq
\epsilon $

\item $V(\mu _{x}^{1},\mu _{x}^{2})\leq \delta $ .
\end{enumerate}

Then\footnote{We remark that to have the left hand of item (1)  well defined we can assume (without changing $\mu ^{1}$)
that $\mu _{\gamma }^{2}$ is defined in some way, for example $\mu _{\gamma
}^{2}=m$ (the one dimensional Lebesgue measure on the leaf)\ for each leaf
where the density of $\mu _{x}^{2}$ is null.
} $|\int g d\mu ^{1}-\int g d\mu ^{2}|\leq ||g||_{\updownarrow
lip}(\epsilon +\delta ).$
\end{proposition}

\begin{proof}
Disintegrating $\mu ^{1}$ and $\mu ^{2}$ we get 
\begin{equation}
|\int gd\mu ^{1}-\int gd\mu ^{2}|=\left\vert \int_{\gamma \in \II%
}\int_{\gamma }g\,d\mu _{\gamma }^{1}\,d\mu _{x}^{1}-\int_{\gamma \in \II%
}\int_{\gamma }g\,d\mu _{\gamma }^{2}\,d\mu _{x}^{2}\right\vert .
\end{equation}%
Adding and subtracting $\int \int_{\gamma }g\,d\mu _{\gamma }^{2}\,d\mu
_{x}^{1}$ the last expression is equivalent to 
\begin{equation*}
\left\vert \int_{\II}\int_{\gamma }g\,d\mu _{\gamma }^{1}\,d\mu
_{x}^{1}-\int_{\II}\int_{\gamma }g\,d\mu _{\gamma }^{2}\,d\mu _{x}^{1}+\int_{%
\II}\int_{\gamma }g\,d\mu _{\gamma }^{2}\,d\mu _{x}^{1}-\int_{\II%
}\int_{\gamma }g\,d\mu _{\gamma }^{2}\,d\mu _{x}^{2}\right\vert .
\end{equation*}%
 This is bounded by (see Remark \ref{lipW}) 
\begin{align*}
\Big|\int_{\II}\Big(\int_{\gamma }g\,d\mu _{\gamma }^{1}& -g\,d\mu _{\gamma
}^{2}\Big)\,d\mu _{x}^{1}\Big|+\left\vert \int_{\II}\int_{\gamma }g\,d\mu
_{\gamma }^{2}\,d\mu _{x}^{1}-\int_{\II}\int_{\gamma }g\,d\mu _{\gamma
}^{2}\,d\mu _{x}^{2}\right\vert \\
& \leq  \Vert g\Vert _{\updownarrow lip} \epsilon \ +\left\vert \int_{\II}\int_{\gamma
}g\,d\mu _{\gamma }^{2}\,d\mu _{x}^{1}-\int_{\II}\int_{\gamma }g\,d\mu
_{\gamma }^{2}\,d\mu _{x}^{2}\right\vert \\
& \leq \epsilon \Vert g\Vert _{\updownarrow lip}+\left\vert \int_{\II%
}\int_{\gamma }g\,d\mu _{\gamma }^{2}\,d(\mu _{x}^{1}-\mu
_{x}^{2})\right\vert .
\end{align*}%
For almost each $\gamma $ it holds $h(\gamma )=|\int_{\gamma }g\,d\mu
_{\gamma }^{2}|\leq ||g||_{\updownarrow lip}$. Thus, by assumption (2) in
the statement, the proposition is proved.
\end{proof}

\subsection{Exponential convergence to equilibrium and decay of
correlations. \label{sec:primeiroretorno}}

Let us consider a manifold $M$ (possibly with boundary) and the dynamics on $%
M$ generated by the iteration of a function $T:M\rightarrow M$. We will
consider a notion of speed of approach of an absolutely continuous initial
measure, with density $f$, to an invariant measure $\mu $. Of course, many
generalizations are possible, but we will consider here only absolutely
continuous measures as starting measures.

\begin{definition}
\label{def:exp-conv-equil} We say that $T$ has exponential convergence to
equilibrium with respect to norms $\Vert .\Vert _{a}$ and $\Vert .\Vert _{b}$%
, if there are $C,\Lambda \in \mathbb{R}^{+},~\Lambda <1$ such that for $%
f\in L^1(m), g\in L^1(\mu)$ 
\begin{equation*}
\left|\int f\cdot (g\circ T^{n}) \, dm - \int g\,d\mu \int f\,dm\right| \leq
C\Lambda ^{n}\cdot \| g\|_{a}\cdot \|f\|_{b}, \quad n\ge1
\end{equation*}
where $m$ is the Lebesgue measure on $M$.
\end{definition}

We remark that in several situations, exponential convergence to equilibrium
can be obtained as consequence of spectral gap of the transfer operator
restricted to suitable funcion spaces, see Lemma \ref{eigen2}.

Now we consider maps preserving a regular foliation, which contracts the
leaves and whose quotient map (the induced map on the space of leaves) has
exponential convergence to equilibrium. We will give an estimation of the
speed of convergence to equilibrium for this kind of maps, establishing that
it is also exponential.

\begin{definition}
\label{def:pi-f} If $f:Q \rightarrow \mathbb{R}$ is integrable, we denote by $\pi(f):\II%
\rightarrow \mathbb{R}$ the function $\pi (f): x\mapsto\int_{\II}f(x,t)~dt.$
\end{definition}

\begin{theorem}
\label{resuno} Let $F:Q\circlearrowleft$ be a Borel function such that $%
F(x,y)=(T(x),G(x,y))$. Let $\mu $ be an $F$-invariant measure with marginal $%
\mu_{x}$ on the $x$-axis which, moreover, is $T$-invariant. Let us suppose
that

\begin{enumerate}
\item $(T,\mu _{x})$ has exponential convergence to equilibrium with respect
to the norm $\Vert \cdot \Vert _{\infty }$ (the $L^{\infty }$ norm) and to a
norm which we denote by $\Vert \cdot \Vert _{\_}$ .

\item $T$ is nonsingular with respect to the Lesbegue measure, piecewise continuous and monotonic: there is a collection of
intervals $\{I_{i}\}_{i=1,...,m}$, $\cup I_{i}=I$ such that on each $I_{i}$, 
$T$ is an homeomorphism onto its image.

\item $F$ is a contraction on each vertical leaf: $G$ is $\lambda $
-Lipschitz in $y$ with $\lambda <1$.
\end{enumerate}

Then $(F,\mu )$ has exponential convergence to equilibrium in the following
sense. There are $C,\Lambda \in \mathbb{R}^{+},~\Lambda <1$ such that 
\begin{equation*}
\left| \int f\cdot(g\circ F^{n}) \, dm- \int g \,d\mu \int f\, dm\right|
\leq C\Lambda ^{n} \cdot\|g\|_{\updownarrow lip}\cdot (||\pi
(f)||_{\_}+||f||_{1})
\end{equation*}
for each $f\geq 0$.
\end{theorem}

\begin{proof}
Let us take an integrable non-negative $f:Q\rightarrow \RR$ and an
essentially bounded $g:Q\rightarrow \RR$. We can divide and multiply by $%
\Vert g\Vert _{\updownarrow lip}\int f\,dm$ obtaining 
\begin{equation*}
\left\vert \int g\circ F^{n}\cdot f\,dm-\int g\,d\mu \int f\,dm\right\vert
=\left( \Vert g\Vert _{\updownarrow lip}\int f\,dm\right) \left\vert \int 
\frac{g\circ F^{n}}{\Vert g\Vert _{\updownarrow lip}}\cdot \frac{f}{\int
f\,dm}\,dm-\int \frac{g}{\Vert g\Vert _{\updownarrow lip}}\,d\mu \right\vert
.
\end{equation*}%
We denote $\nu ^{n}=F^{\ast n}(\frac{f}{\int fdm}m)$ and disintegrate the
measures. Let us suppose having iterated the system $n_{0}$ times, and then
continue to iterate again $n$ more times. By item (1) after the first $n_{0}$
iterations, the marginals become exponentially close:%
\begin{equation} \label{sppo}
\sup_{||h||_{\infty }\leq 1}|\int hd((\nu ^{n_0})_{x})-\int hd(\mu _{x})|\leq
C\Lambda ^{n_{0}}\cdot \Vert \pi (\frac{f}{\int fdm})\Vert _{\_}. 
\end{equation}

Now iterating $n$ more times we estimate the distance on the leaves, and
then the speed of convergence by Theorem \ref{prod}. Let us consider $
\{{\overline{I}}_{i}\}_{i=1,...,m^{\prime }}$, the intervals where the continuous
branches of $T^{n}$ are defined. Let us consider $\varphi
_{i}=1_{{\overline{I}}_{i}\times I}$ and $\nu _{i}=\varphi _{i}\nu ^{n_{0}}$, $\mu
_{i}=\varphi _{i}\mu $ so that $\mu =\sum \mu _{i}$, ${\nu ^{n_0} }=\sum \nu _{i}$. 

Remark that by item (1) after the first $n_{0}$ iterations, also the 
marginals of $\nu _{i}$ and $\mu _{i}$ are exponentially close, and since $%
\varphi _{i}$ have disjoint support:%
\begin{equation}
\sum_{i\leq m^{\prime }}\sup_{||h||_{\infty }\leq 1}|\int hd(( \nu_{i})_{x} )-\int hd(( \mu _{i} )_{x} )|\leq C\Lambda ^{n_{0}}\cdot \Vert \pi (\frac{f}{%
\int fdm})\Vert _{\_}.  
\end{equation}

By triangle inequality%
\begin{equation*}
\left\vert \int \frac{g}{||g||_{\updownarrow lip}}~d({F^{* (n+n_{0})}}(\nu
))-\int \frac{g}{\Vert g\Vert _{\updownarrow lip}}~d(\mu )\right\vert \leq
\sum_{i=1,..,m}\left\vert \int \frac{g}{||g||_{\updownarrow lip}}%
~d({F^{*n}}(\nu _{i}))-\int \frac{g}{\Vert g\Vert _{\updownarrow lip}}%
~d({F^{*n}}(\mu _{i}))\right\vert .
\end{equation*}

Let us denote $T_{i}=T^{n}|_{\overline{I}_{i}}$ , remark that this is injective. Then
by Proposition \ref{prod}%
\begin{align*}
&\left\vert \int \frac{g}{||g||_{\updownarrow lip}}~d({F^{*n}}(\nu
_{i}))-\int \frac{g}{\Vert g\Vert _{\updownarrow lip}}~d({F^{*n}}(\mu
_{i}))\right\vert  
\\
&\leq \int_{I}W_{1}(({F}^{*n}\mu _{i})_{\gamma
},({F}^{*n}\nu _{i})_{\gamma })~d{T^{*n}}((\nu _{i})_{x})+||{T}^{*n}((\mu
_{i})_{x})-{T}^{*n}((\nu _{i})_{x})||_{1} 
\\
&\leq \int_{I}W_{1}({F}^{*n}((\mu _{i})_{T_{i}^{-1}(\gamma
)}),{F}^{*n}((\nu _{i})_{T_{i}^{-1}(\gamma )}))~d{T^{*n}}((\nu
_{i})_{x})+||{T}^{*n}((\mu _{i})_{x})-{T}^{*n}((\nu _{i})_{x})||_{1} 
\end{align*}
and by uniform contraction along the leaves we can bound
\begin{align*}
&\leq \lambda ^{n}\int_{I}W_{1}((\mu _{i})_{T_{i}^{-1}(\gamma )},(\nu
_{i})_{T_{i}^{-1(\gamma )}})~d{T^{*n}}((\nu _{i})_{x})+||{T}^{*n}((\mu
_{i})_{x})-{T}^{*n}((\nu _{i})_{x})||_{1} 
\\
&=\lambda ^{n}\int_{I_{i}}W_{1}((\mu _{i})_{\gamma },(\nu _{i})_{\gamma
})~d((\nu _{i})_{x})+||{T}^{*n}((\mu _{i})_{x})-{T}^{*n}((\nu
_{i})_{x})||_{1}.
\end{align*}

Summarizing, we obtain
\begin{align*}
&\left\vert \int \frac{g}{||g||_{\updownarrow lip}}~d({F^{*n}}(\nu
_{i}))-\int \frac{g}{\Vert g\Vert _{\updownarrow lip}}~d({F^{*n}}(\mu
_{i}))\right\vert 
\\
& \leq \lambda ^{n}\sum_{i}\int_{I_{i}}W_{1}((\mu
_{i})_{\gamma },(\nu _{i})_{\gamma })~d((\nu
_{i})_{x})+\sum_{i}||{T}^{*n}((\mu _{i})_{x})-{T}^{*n}((\nu
_{i})_{x})||_{1} 
\\
&=\lambda ^{n}\int_{I}W_{1}(\mu _{\gamma },\nu _{\gamma }^{n_{0}})~d(\nu
_{x})+\sum_{i}||{T}^{*n}((\mu _{i})_{x})-{T}^{*n}((\nu _{i})_{x})||_{1}.
\end{align*}
Since ${T}^{*n}$ is a $L^{1}$ contraction, from \eqref{sppo}
this is less or equal than
\begin{align*}
\lambda ^{n}\int_{I}W_{1}(\mu _{\gamma },\nu _{\gamma }^{n_{0}})~d(\nu
_{x})
&+
\sum_{i}\|(\mu _{i})_{x}-(\nu _{i})_{x}\|_{1}
\\
&\leq
\lambda ^{n}\int_{I}W_{1}(\mu_{\gamma },\nu_{\gamma
}^{n_{0}})~d(\nu_{x})+C\Lambda ^{n_{0}}\cdot \Vert \pi (\frac{f}{\int fdm}%
)\Vert_{\_}.
\end{align*}
Finally we can deduce
\begin{align*}
&\left( \Vert g\Vert _{\updownarrow lip}\int f\,dm\right) \left\vert \int 
\frac{g}{||g||_{\updownarrow lip}}~d(F^{\ast (n+n_0 )}(\frac{f}{\int fdm}m))-\int 
\frac{g}{\Vert g\Vert _{\updownarrow lip}}~d\mu \right\vert 
\\
&\leq 
\left( \Vert g\Vert _{\updownarrow lip}\int f\,dm\right) \big[\lambda
^{n}+C\Lambda ^{n_{0}}\cdot \Vert \pi (\frac{f}{\int fdm})\Vert _{\_}\big]
\end{align*}
which implies the statement and concludes the proof of
Theorem~\ref{resuno}.
\end{proof}

\begin{remark}
\label{rem1} Since $\Vert f\Vert _{\infty }\geq \Vert f\Vert _{1}$ we can
also state: under the assumptions of Theorem \ref{resuno}, if $f\geq 0$ 
\begin{equation*}
\left\vert \int f\cdot (g\circ F^{n})\,dm-\int g\,d\mu \int f\,dm\right\vert
\leq C\Lambda ^{n}\cdot \Vert g\Vert _{\updownarrow lip}\cdot (||\pi
(f)||_{\_}+||f||_{\infty }).
\end{equation*}
\end{remark}

\begin{remark}\label{rmk:slowconverg} 
  We remark that if, instead of exponential convergence to
  equilibrium for the base map ($C\Lambda ^{n_{0}}\cdot
  \Vert \pi (\frac{f}{\int fdm})\Vert _{\_}\big]$), we
  had a slower rate of convergence ($C_{n_{0}}\cdot \Vert \pi
  (\frac{f}{\int fdm})\Vert _{\_}$ with $C_n \to 0$
  slower than exponential), then because $\lambda ^{n}$ will
  converge to $0$ much faster than $C_n$ (see the last
  equation in the above proof), the same arguments
  lead to an estimation for the convergence to
  equilibrium for the skew product with the same kind
  of asymptotical behavior as $C_n$.
\end{remark}

Now let us relate convergence to equilibrium to decay
of correlations. We consider the general case of a
measure preserving map.

\begin{lemma}[convergence to equilibrium and decay of correlations]
\label{lemmone}
Let us consider a measurable map $F:\Omega\rightarrow
\Omega$, two probability measures $m,\mu $ on $\Omega$,
such that $\mu $ is invariant. If we have a convergence
to equilibrium with speed $C_{n}$ and norms
$\|\cdot\|_{1}$ and $\|\cdot\|_{2}$, that is
\begin{equation*}
\left|\int g(F^{n}(x))f(x)dm-\int f(x)dm\int g(x)d\mu \right| \leq
C_{n}\cdot \Vert f\Vert _{1}\cdot ||g||_{2},
\end{equation*}%
then, for each $k$, $|\int f\cdot(g\circ F^{n})~d\mu -\int g~d\mu \int
f~d\mu |$ is bounded by 
\begin{gather*}
C_{k}||1||_{1}||g\circ F^{n}~f||_{2}+ C_{n}||f\circ F^{k}||_{1}||g\circ
F^{k}||_{2}+C_{k}\left|\int g~d\mu\right| ~||f||_{2}||1||_{1},
\end{gather*}
where $1$ is the constant function with value $1.$
\end{lemma}

\begin{proof}
Adding and subtracting we rewrite $|\int f\cdot(g\circ F^{n})~d\mu -\int
g~d\mu \int f~d\mu |$ as 
\begin{align*}
\Big|&\int f\cdot(g\circ F^{n})~d\mu - \int 1\cdot (g\circ F^{n+k})(f\circ
F^{k})~dm \\
&+\int ( g\circ F^{n+k})(f\circ F^{k})~dm-\int f\circ F^{k}~dm\int g~d\mu \\
&+\int (f\circ F^{k})\cdot 1~dm\int g~d\mu -\int g~d\mu \int f~d\mu \Big|.
\end{align*}

Applying the assumption to each line we obtain the three summands in the
statement.
\end{proof}

Now we use the above results to deduce exponential decay of correlations
from exponential speed of convergence to equilibrium.

\begin{theorem}
\label{resdue} Let $F$ be a map satisfying the assumptions of Theorem \ref%
{resuno} and let us suppose that there are $C_1,K\in \mathbb{R}$ and a
seminorm $\|\cdot\|_{\square }$ such that for each $n\ge1$ 
\begin{equation}
\|\pi (f\circ F^{n})\|_{\_}+\|f\circ F^{n}\|_{\square }\leq C_{1}K^{n}(\|\pi
(f)\|_{\_}+\|f\|_{\updownarrow lip}+\|f\|_{\square }).  \label{square}
\end{equation}%
Then $F$ has exponential decay of correlations: there are $C_2,\Lambda \in 
\mathbb{R}^{+},~\Lambda <1$ such that for $n\ge1$ 
\begin{equation*}
\left|\int f\cdot(g\circ F^{n})~d\mu -\int g~d\mu \int f~d\mu \right| \leq
C_{2}\Lambda^{n}\|g\|_{\updownarrow lip}(\|f\|_{\updownarrow lip}+\|\pi
(f)\|_{\_}+\|f\|_{\square }).
\end{equation*}
\end{theorem}

\begin{proof}
Let us consider bounded observables $f,g$. Since adding a constant to $f$
the correlation integral does not change, we can suppose $f\geq 0$ and we
can apply Theorem \ref{resuno}.

By Theorem \ref{resuno}, Remark \ref{rem1} and Lemma \ref{lemmone}, there
are $C,\Lambda \in \mathbb{R}^{+},~\Lambda <1,$ s.t. for each $k$%
\begin{gather*}
\left| \int f\cdot(g\circ F^{n})~d\mu -\int g~d\mu \int f~d\mu \right| \leq
C\Lambda ^{k}\big(||\pi(1)||_{\_}+||1||_{\infty }\big)\cdot ||f\cdot(g\circ
F^{n})||_{\updownarrow lip}+ \\
C\Lambda ^{n}\big(||\pi (f\circ F^{k})||_{\_}+||f\circ F^{k}||_{\infty }\big)%
\cdot||g\circ F^{k}||_{\updownarrow lip}+C\Lambda ^{k}\left|\int
g~d\mu\right| ~||f||_{\updownarrow lip}\big(||\pi (1)||_{\_}+||1||_{\infty }%
\big).
\end{gather*}

By assumption (\ref{square}), we take $k=\left\lfloor \alpha n\right\rfloor $
(the integer part) with $\alpha $ so small that 
\begin{equation*}
C\Lambda ^{n}(||\pi (f\circ F^{\left\lfloor \alpha n\right\rfloor
})||_{\_}+||f\circ F^{\left\lfloor \alpha n\right\rfloor }||_{\square })
\end{equation*}
goes to zero exponentially; that is, there are $C_{3},\Lambda _{3}\in 
\mathbb{R}^{+},~\Lambda _{3}<1$ such that 
\begin{equation}
C\Lambda ^{n}(||\pi (f\circ F^{\left\lfloor \alpha n\right\rfloor
})||_{\_}+||f\circ F^{\left\lfloor \alpha n\right\rfloor }||_{\square })\leq
C_{3}\Lambda _{3}^{n}(||\pi (f)||_{\_}+||f||_{\square }+||f||_{\updownarrow
lip}).  \label{2square}
\end{equation}
Let us evaluate each term of the above sum. We recall the definition of $\lip%
_y(g)$ from~(\ref{eq:lip_y}). We observe that we can bound the first summand
in the above inequality in the following way 
\begin{equation*}
C\Lambda ^{k}(||\pi (1)||_{\_}+||1||_{\infty })||f \cdot (g\circ
F^{n})||_{\updownarrow lip} \leq C_{2}\Lambda ^{\left\lfloor \alpha
n\right\rfloor }\Big(||f\cdot(g\circ F^{n})||_{\infty !} +Lip_{y}\big(%
f\cdot(g\circ F^{n})\big)\Big).
\end{equation*}%
Since $F$ is contracting on the vertical direction, when $n$ grows we get 
\begin{align*}
Lip_{y}(g\circ F^{n})\rightarrow 0, \quad Lip_{y}(g\circ F^{n})\leq
Lip_{y}(g) \quad \text{and} \quad ||g\circ F^{n}||_{\updownarrow lip}\leq
||g||_{\updownarrow lip}.
\end{align*}
Then, for all big enough $n$, $C_{2}\Lambda ^{\left\lfloor \alpha
n\right\rfloor }\Big(||f\cdot(g\circ F^{n})||_{\infty !} +Lip_{y}\big(%
f\cdot(g\circ F^{n})\big)\Big)$ is bounded from above by 
\begin{align*}
C_{2}\Lambda^{\left\lfloor \alpha n\right\rfloor } \big(||g||_{\infty
!}||f||_{\infty !}+||g||_{\infty !}Lip_{y}(f)+||f||_{\infty !}Lip_{y}(g)%
\big) \leq C_{2}\Lambda ^{\left\lfloor \alpha n\right\rfloor}
||g||_{\updownarrow lip}||f||_{\updownarrow lip}.
\end{align*}
The second summand can be estimated as%
\begin{align*}
C\Lambda ^{n}&\big(||\pi (f\circ F^{k})||_{\_}+||f\circ F^{k}||_{\infty }%
\big)\cdot||g\circ F^{k}||_{\updownarrow lip} \\
&\leq C\Lambda ^{n}(||\pi (f\circ F^{\left\lfloor \alpha n\right\rfloor
})||_{\_}+||f\circ F^{\left\lfloor \alpha n\right\rfloor }||_{\infty }
+||f\circ F^{\left\lfloor \alpha n\right\rfloor }||_{\square })||g\circ
F^{\left\lfloor \alpha n\right\rfloor }||_{\updownarrow lip} \\
&\leq (C_{3}\Lambda _{3}^{n}(||\pi (f)||_{\_}+||f||_{\updownarrow
lip}+||f||_{\square }) +C\Lambda ^{n}||f||_{\infty })||g||_{\updownarrow
lip},
\end{align*}
where the last inequality is obtained using inequality \eqref{2square}.
Finally, the last term is%
\begin{equation*}
C\Lambda ^{\left\lfloor \alpha n\right\rfloor }\left|\int g~d\mu\right|
~||f||_{\updownarrow lip}(||\pi (1)||_{\_}+||1||_{\infty }),
\end{equation*}%
Altogether, we can bound $|\int g\circ F^{n}f~d\mu -\int g~d\mu \int f~d\mu
| $ by $(I)+(II)+(III)+(IV)$, where 
\begin{align*}
(I) &=C_{2}\Lambda ^{\left\lfloor \alpha n\right\rfloor }||g||_{\updownarrow
lip}||f||_{\updownarrow lip} \\
(II)&=C_{3}\Lambda _{3}^{n}(||\pi (f)||_{\_}+||f||_{\updownarrow
lip}+||f||_{\square })||g||_{\updownarrow lip} \\
(III)&=C\Lambda ^{n}||f||_{\infty }||g||_{\updownarrow lip} \\
(IV)&=C\Lambda ^{\left\lfloor \alpha n\right\rfloor }\left|\int
g~d\mu\right| ~||f||_{\updownarrow lip}(||\pi (1)||_{\_}+||1||_{\infty })
\end{align*}
which finally gives%
\begin{equation*}
\left| \int g\circ F^{n}f~d\mu -\int g~d\mu \int f~d\mu \right| \leq
C_{4}\Lambda _{4}^{n}||g||_{\updownarrow lip}(||f||_{\updownarrow lip}+||\pi
(f)||_{\_}+||f||_{\square }).
\end{equation*}
\end{proof}

For ease of reference, the previous results can be summarized as follows.

\begin{theorem}
\label{thm:summary-exp-decay} Let $F:Q\circlearrowleft$ be a Borel function
such that $F(x,y)=(T(x),G(x,y))$, $\mu $ an $F$-invariant probability
measure with $T$-invariant marginal $\mu_{x}$ on the $x$-axis and satisfying

\begin{enumerate}
\item $(T,\mu _{x})$ has exponential convergence to equilibrium with respect
to the norm $\Vert \cdot \Vert _{1}$ and to a norm $\Vert \cdot \Vert _{\_}$;

\item $T$ is nonsingular with respect to the Lesbegue measure, piecewise continuous and monotonic: there is a collection of
intervals $\{I_{i}\}_{i=1,...,m}$, $\cup I_{i}=I$ such that on each $I_{i}$, 
$T$ is an homeomorphism onto its image.

\item $F$ is a uniform contraction on each vertical leaf.
\end{enumerate}

Moreover, let us assume that that there are $C,K\in \mathbb{R}$ and a
seminorm $\|\cdot\|_{\square }$ such that 
\begin{align*}
\|\pi (f\circ F^{n})\|_{\_}+\|f\circ F^{n}\|_{\square }\leq C_{1}K^{n}(\|\pi
(f)\|_{\_}+\|f\|_{\updownarrow lip}+\|f\|_{\square }), \quad n\ge1.
\end{align*}
Then $F$ has exponential decay of correlations: there are $C,\Lambda \in 
\mathbb{R}^{+},~\Lambda <1$ 
\begin{equation*}
\left|\int f\cdot(g\circ F^{n})~d\mu -\int g~d\mu \int f~d\mu \right| \leq
C_{2}\Lambda^{n}\|g\|_{\updownarrow lip}(\|f\|_{\updownarrow lip}+\|\pi
(f)\|_{\_}+\|f\|_{\square })
\end{equation*}
for all $f,g:Q\to\RR$ and $n\ge0$.
\end{theorem}

To use this result in concrete examples we must find suitable norms $%
\|\cdot\|_{\_}$ and $\|\cdot\|_{\square}$ satisfying the above assumptions
in each particular application. For the case of Poincar\'e maps of singular
hyperbolic systems such norms will be introduced in the next section.

%%%%%%%%%%%%%%%%%%%%%%%%%%%%%%%%%%%%%%%%%%%%%%%%%%%%%%%%%%%%

\section{$p-$bounded variation and $C^{1+\alpha}$ piecewise
  expanding maps}
\label{sec:p-bounded-variation}
We recall the main definitions and results about $p-$bounded
variation and iteration of piecewise expanding maps $T$ such
that $\frac{1}{T^{^{\prime }}} $ has $p-$bounded
variation. Almost everything in this section is taken from
\cite{Ke85}. We will however put the results of \cite{Ke85}
in a form which can be used for our purposes.

Given a function $g:[0,1]\rightarrow \mathbb{R}$ we define
its \emph{%
  universal} $p-$variation as the following adaptation of
the usual notion of bounded variation:
\begin{eqnarray*}
  var_{p}(g,x_{1},...,x_{n}) &=&\left(\sum_{i\leq n}|g(x_{i})-g(x_{i+1})|^{p}\right)^{\frac{1}{p}} \\
  var_{p}(g) &=&\sup_{(x_{i})\in
    \text{Finite~subdivisions~of~}[0,1]} var_p(g,x_{1},...,x_{n}).
\end{eqnarray*}
Let us also define $UBV_{p}=\{g  :  var_{p}(g)<\infty \}$.

We will need another definition of variation for maps with
two variables that we present here for convenience.
Similarly to the one dimensional case, if $f:Q\rightarrow \mathbb{R}$
and $x_{i}\leq x_{2}\leq ...\leq x_{n}$, let us define
\begin{equation*}
var^{\square }(f,x_{1},...,x_{n},y_{1},...,y_{n})=\sum_{1\le
  i\leq
n}|f(x_{i},y_{i})-f(x_{i+1},y_{i})|.
\end{equation*}
We then consider the supremum $var^{\square
}(f,x_{1},...,x_{n},y_{1},...,y_{n})$ over all subdivisions $x_{i}$ and all
choices of the $y_{i}$
\begin{equation*}
  var^{\square }(f)=\sup_{n}\left( \sup_{(x_{i}\leq x_{2}\leq ...\leq
      x_{n})\in\II,(y_{i})\in\II}var^{\square
    }(f,x_{1},...,x_{n},y_{1},...,y_{n})\right).
\end{equation*}
Let $m$ be the Lebesgue measure on the unit interval,
$\epsilon >0$ and $h:\II\rightarrow \mathbb{C}$. We define
\begin{align*}
  \osc(h,\epsilon ,x)=\esssup \{|h(y_{1})-h(y_{2})|:
  y_{1},y_{2}\in B_{\epsilon }(x)\cap \II\},
\end{align*}
where $B_{\epsilon }(x)$ is the ball centered in $x$ with
radius $\epsilon $, and the essential supremum is taken with
respect to the product measure $m^{2}$ on $Q$. Now let us
define
\begin{align*}
  \osc_{p}(h,\epsilon )=\|\osc(h,\epsilon ,x)\|_{p}, \quad
  1\le p \le \infty,
\end{align*}
where the $p$-norm is taken with respect to $m$.
\begin{remark}\label{rmk:osc_p}
  $\osc_{p}(h,\ast ):(0,A]\rightarrow \lbrack 0,\infty ]$ is
  a non decreasing function and 
  $\osc_{p}(h,\epsilon )\geq \osc_{1}(h,\epsilon )$.
\end{remark}
Fixed $0\leq r\leq 1$, set
$
R_{p,r,n}=\{h|\forall \epsilon \in (0,A], \osc_{p}(h,\epsilon )\leq n\epsilon
^{r}\}
$
and
$
S_{p,r}=\cup _{n\in \mathbb{N}}R_{p,r,n}.
$

We can now define:
\begin{enumerate}
\item $BV_{p,r}$ as the space of $m-$equivalence classes of
  functions in $S_{p,r}$
\item
$var_{p,r}(h)=\sup_{0<\epsilon \leq A} (\epsilon^{-r}osc_{p}(h,\epsilon))$

(we remark that this definition depends on a fixed constant $A$ and that
$var_{p,r}(h)\geq var_{1,r}(h)$).

\item for $h\in BV_{p,r}$ we define $\|h\|_{p,r}:=var_{p,r}(h)+\|h\|_{p}$.
\end{enumerate}
It turns out that $BV_{p,r}$ with the norm $||h||_{p,r}$ is
a Banach space; see \cite[Thm. 1.13]{Ke85} and \cite[Lemma
2.7]{Ke85}.
\begin{proposition}
\label{cmp}$UBV_{p}\subseteq BV_{p,\frac{1}{p}}\subseteq BV_{1,\frac{1}{p}}$
for all $1\leq p<\infty $. Moreover
\begin{equation}
var_{1,\frac{1}{p}}(h)\leq var_{p,\frac{1}{p}}(h)\leq 2^{\frac{1}{p}%
}var_{p}(h).  \label{4.6}
\end{equation}
\end{proposition}
In what follows we need to compare the $\|\cdot\|_{p,r}$ norm
with the $L^{\infty }(m)$ norm.
\begin{lemma}
\label{lemaa} If $f\in BV_{1,r}$ ($r\leq 1$), then $f\in L^{\infty }(m)$ and
$
\|f\|_{\infty }\leq A^{r-1}\cdot\|f\|_{1,r},
$
where $A$ is the constant in the definition of
$\|\cdot\|_{1,r}$ (see item 2 above) and $\|\cdot\|_{\infty
}$ is the norm of $L^{\infty }(m)$.
\end{lemma}

\begin{proof} From the definition of $var_{p,r}(f)$ and
  considering $\epsilon =A$, then $var_{1,r}(f)\geq
  A^{-r}\cdot\osc_{1}(f,A)$ and
  $\int_{[0,1]}|\osc(f,A,x)|\,dm(x)\leq A^{r}\cdot var_{1,r}(f)$.

  Moreover $B_{\epsilon }(x)\supseteq B_{\epsilon/2}(y)$
  implies $\osc(f,\epsilon ,x)\geq \osc(f,\epsilon/2,y)$,
  and then
\begin{equation*}
  \int_{\lbrack 0,A]}|\osc(f,A,x)|\,dm(x)\geq 
  A\cdot \osc(f,A/2,A/2),
\end{equation*}
and by induction we can prove that (where $[x]$ is the
biggest integer no larger than $x$)
\begin{align*}
  \int_{\lbrack 0,1]}|\osc(f,A,x)|~dm(x)\geq
  A\sum_{i=0}^{[1/A]+1}\osc(f,\frac{A}{2},\frac{A}{2}+iA)
\end{align*}
and $\sum_{i=0}^{[1/A]+1}\osc(h,\frac{A}{2}+iA,\frac{A}{2}%
)\geq \osc(f,\frac{1}{2},\frac{1}{2})$ by the triangle
inequality. Thus
\begin{align*}
  var_{1,r}(f)\geq A^{1-r}\osc(h,\frac{1}{2},\frac{1}{2}).
\end{align*}
Finally $\|f\|_{\infty }\leq
\osc(h,\frac{1}{2},\frac{1}{2})+\|f\|_{1}$ and then
$\|h\|_{1,r}=var_{1,r}(f)+\|f\|_{1}\geq
A^{1-r}\osc(h,\frac{1}{2},\frac{1}{2})+\|f\|_{1}\geq
A^{1-r}\|f\|_{\infty }.$
\end{proof}
Let us consider the iteration of piecewise expanding maps
whose inverse of the derivative has $p-$bounded variation.
Let $\{I_{1},...,I_{n}\}$ a finite interval partition of
$\II$ and $ T:\II\circlearrowleft$ be a transformation which
is monotone and continuous on each interval $I_{i}$. We
assume that
\begin{enumerate}
\item $T$ is nonsingular with respect to the measure $m$;
\item $T^{\prime }$ exists inside the intervals and
  $\frac{1}{T^{\prime }}$ is bounded almost everywhere;
\end{enumerate}
Under these assumptions let us consider a measurable and
bounded function $f$ and the Perron-Frobenius operator
related to $T$
\begin{equation}
Pf(x)=\sum_{i=1}^{N}\frac{f}{T^{\prime }}\circ T_{i}^{-1}\cdot
1_{T(I_{i})}~,where~T_{i}=T|_{I_{i}}.
\end{equation}
This operator represent the action of the transfer operator on densities of absolutely continuous measures, and extends to a linear contraction on
$L^{1}(m)$. Under these assumptions, if the map $T$ is
piecewise expanding, and $1/|T'|$ has universal
$p-$bounded variation, a kind of Lasota Yorke inequality can
be proved (see \cite[Thm 3.2 and 3.5]{Ke85}).
\begin{theorem}[Lasota Yorke inequality for $P$]\label{th-keller}
Let $T,P,m$ be as described above.
\begin{enumerate}
\item If $\frac{1}{T^{\prime }}\in UBV_{p}$ (where
$1\leq p<\infty $) and
\item there is an $n\in \mathbb{N}$ with
  $\|\frac{1}{(T^{n})^{\prime }}\|_{\infty }<1$,
\end{enumerate}
then there are $0<\beta <1$ and $C>0$
such that for each $f\in BV_{1,\frac{1}{p}}$
\begin{equation}
  \|Pf\|_{1,\frac{1}{p}}\leq \beta \|f\|_{1,\frac{1}{p}}+C\|f\|_{1}.
\end{equation}
\end{theorem}
Since $P$ is a $L^{1}(m)$-contraction, the theorem of
Ionescu-Tulcea-Marinescu \cite{IM50} provides a description
of the spectral properties of $P$; see
\cite[Thm. 3.3]{Ke85}.
\begin{enumerate}
\item $P:L^{1}(m)\rightarrow L^{1}(m)$ has a finite number of eingenvalues $%
\lambda _{1},...,\lambda _{r}$ of modulus $1$;
\item Let $E_{i}=\{f\in L^{1}(m)|Pf=\lambda _{i}f\}$, then
  $E_{i}\subset BV_{1,\frac{1}{p}}$ and $\dim (E_{i})<\infty
  $ ($i=1,...,r$);
\item $P=\sum_{i=1}^{r}\lambda _{i}\Psi _{i}+Q$, where $\Psi
  _{i}$ are projections onto the eigenspaces $E_{i}$ with
  $||\Psi _{i}||_{1}\leq 1$ and $%
  Q$ is a linear operator on $L^{1}(m)$ with
  $Q(BV_{1,\frac{1}{p}})\subseteq BV_{1,\frac{1}{p}}$,
  $\sup_{n}||Q^{n}||_{1}<\infty $ and $||Q^{n}||_{1,\frac{%
      1}{p}}=O(\Lambda ^{n})$ for some $0<\Lambda
  <1$. Furthermore $\Psi _{i}\Psi _{j}=0$ ($i\neq j$) and
  $\Psi _{i}Q=Q\Psi _{i}=0$ (for all $i$);
\item $1$ is an eigenvalue of $P$, and assuming
  $\lambda_{1}=1$ and $h=\Psi _{1}(1)$, $\mu =h~m$ is the
  greatest $T$-invariant probability on $X$ absolutely
  continuous with respect to $m$, i.e., if $\tilde{\mu}$ is
  $T$-invariant and $\tilde{\mu}<<m$, then $\tilde{\mu}<<\mu $;
\item there is a finite partition $\{C_{l,k}\}_{l=1,...,r;
    k=1,...,L_{l}}$ of $[0,1]$ such that
  $T(C_{l,k})=C_{l,k+1}$ for $k=1,\dots, L_{1}-1$ and
  $T(C_{l,L_{1}})=C_{l,1}$; moreover, $T^{L_{l}}|_{C_{l,k}}$
  is weakly mixing for each $k,l$.
\end{enumerate}

\begin{remark}\label{eigen1}
  The above result (item (5)) tells that some iterate $P^n$
  of the transfer operator has a finite set of a.c.i.m. with
  no eigenvalues other than $1$ on the unit circle.  An
  iterated of the system can be hence decomposed into a
  finite union of invariant sets and each invarant subsystem
  has a unique a.c.i.m and no other eigenvalues on the unit
  circle.
\end{remark}

This implies exponential {\em convergence to equilibrium} for some iterate  of the map,
as needed in Theorem \ref{resdue}.

\begin{proposition}\label{eigen2}
  Under the same assumptions as above, if $g\in
  L^{1}(m),f\in BV_{1,\frac{1}{p}}$, $\mu$ is an a.c.i.m.,
  the associated transfer operator has $1$ as a simple
  eigenvalue and there are no more eigenvalues on the unit
  circle, then there is $C,\Lambda >0,\Lambda \leq 1$ s.t.
\begin{equation}
  \left|\int g(T^{n}(x))f(x)\, dm-\int g(x)\,d\mu \int
    f(x) \, dm\right|
  \leq 
  C\Lambda ^{n}\cdot \Vert g\Vert _{_{1}}\cdot \Vert f\Vert _{1,r}  \label{decay}.
\end{equation}

\end{proposition}

\begin{proof}
  Let us set $f_{0}=\Psi _{1}(1)$.  Since $1$ is a single
  eigenvalue and there are no more eigenvalues on the unit
  circle, then $BV_{1,\frac{1}{p}}=\mathbb{R}f_{0}+X_{0}$
  (this decomposition is invariant by $P$), and
  $spec(P|X_{0})$ is contained in a disc with radius
  $\Lambda <1$. Without loss of generality we can suppose
  that $\int f(x)dm=1$.  By item (4) above, $f_{0}$ is the
  density of the invariant measure $\mu $, therefore
\begin{eqnarray*}
\int g(T^{n}(x))f(x)\,dm - \int g(x)\,d\mu \int f(x)\,dm 
&=& 
\\
\int g\cdot P^{n}f\,dm
-
\int g\cdot f_{0}\,dm 
&=&
\int g\cdot (P^{n}f-f_{0})\,dm 
\\
\int g\cdot (P^{n}(f-f_{0}))\,dm 
&=&
\int g\cdot (P^{n}(\pi _{0}f))\,dm
\end{eqnarray*}
where $\pi_{0}(g)=g-f_{0}\int g\,dm$ is the spectral
projection onto $X_{0}=\{f:\int fdm=0\}.$ Then, by
Lemma~\ref{lemaa} and the spectral radius theorem, we get the needed estimation
\begin{align*}
  \left|\int ~g\cdot (P^{n}(\pi _{0}f))\,dm\right|
    &\leq
  \|g\|_{1}\|P^{n}(\pi _{0}f)\|_{\infty }
  \leq
  \|g\|_{1}A^{r-1}\|P^{n}(\pi _{0}f)\|_{1,r}
 \\ &\leq
  C^{\prime}\Lambda ^{n}\|g\|_{1}\|\pi _{0}f\|_{1,r}
  \leq
  C^{\prime\prime }\Lambda ^{n}\|g\|_{1}\|f\|_{1,r}.
\end{align*}

\end{proof}

%%%%%%%%%%%%%%%%%%%%%%%%%%%%%%%%%%%%%%%%%%%%%%%%%%%%%%%%%%%%%%%%%

\section{Singular-hyperbolic attractors}
\label{sec:Lorenzmodel}

In this section we will introduce what nowadays is called a
singular-hyperbolic attractor for a $3$-dimensional vector
field $X$ (or flow $X_t$).

% An attractor for us is a compact invariant subset $\Lambda$
% for the flow $X_t$ of a vector field $X$, admitting a
% neighborhood $U$ such that

% \noindent $(i)$
% $\Lambda=\Lambda_X(U):=\textrm{closure}(\cap_{t\in\RR}X_t(U))$ is
% the maximal invariant subset of $X$ in $U$; 

% \noindent $(ii)$
% $\textrm{closure}(X_t(U))\subset U$ for all $t>0$ ($U$ is a
% \emph{trapping region}); and 

% \noindent $(iii)$ $\Lambda$ contains a
% non-trivial (different from a fixed point or a periodic
% orbit) transitive orbit of $X$, i.e., there exists
% $x\in\Lambda$ such that $\textrm{closure}(\{X_t(x):t>0\})=\Lambda$.

The singular-hyperbolic class of attractors we consider
contains the partially hyperbolic attractors $\Lambda$ of a
$C^2$ vector field $X$, with finitely many equilibria
accumulated by regular orbits of $X$ in $\Lambda$, and
admitting a continuous and $DX_t$-invariant splitting of the
tangent bundle over $\Lambda$ into a pair $T_\Lambda
M=E^s_\Lambda\oplus E^{cu}_\Lambda$ of vector sub-bundles.
Here $E^s_\Lambda$ has one-dimensional fibers and is
uniformly contracted by $DX_t$, and $E^{cu}_\Lambda$ has
two-dimensional fibers whose area is uniformly expanded by
$DX_t$; see below for precise definitions. As shown in
\cite{MPP04}, this class is an extension of the class of
uniformly hyperbolic attractors, since every
singular-hyperbolic attractor containing no equilibria of
$X$ is uniformly hyperbolic, that is, the sub-bundle
$E^{cu}_\Lambda$ admits a further continuous and
$DX_t$-invariant splitting $E^{cu}_\Lambda=E^X_\Lambda\oplus
E^u_\Lambda$ into the line bundle generated by the flow
direction over $\Lambda$, and the bundle $E^u_\Lambda$ with
one-dimensional fibers uniformly expanded by $DX_t$. In
particular, Anosov (or globally hyperbolic) flows on
three-dimensional manifolds belong to this class. Moreover,
and more important, \emph{the class of singular-hyperbolic
  attractors contains every $C^1$ robustly transitive
  isolated set} for flows on compact three-manifolds $M$.
That is, as proved in \cite{MPP04}, if for a given open
subset $U$ of $M$ there exists a $C^1$ open subset $\U$ of
$\fX^1(M)$ (the family of $C^1$ vector fields of $M$) such
that the maximal invariant subset $\Lambda_Y(U)$ of $U$
contains a non-trivial transitive orbit of $Y$ for every
$Y\in\U$, then $\Lambda_Y(U)$ is a singular-hyperbolic
attractor and contains at most finitely many hyperbolic
saddle-type equilibria, either for $Y$ of for $-Y$, for each
$Y\in\U$.

%In addition, it was proved in \cite{MP03} that a generic
%$C^1$ vector field on a closed $3$-manifold either has
%infinitely many sinks or sources, or else its non-wandering
%set admits a finite decomposition into compact invariant
%sets, each one being either a uniformly hyperbolic set or a
%singular-hyperbolic attractor or a singular-hyperbolic
%repeller (which is a singular-hyperbolic attractor for the
%time reversed flow). Altogether this shows that the
%class of singular-hyperbolic attractors is a good
%representative of the limit dynamics of many flows on
%three-dimensional manifolds.

Next we describe precisely what we mean by a
singular-hyperbolic attractor.  We start by recalling some
definitions and notations and then we shall list some facts
about this kind of attractors proved elsewhere but that will
be useful for us. We also adapt some known results to our
case and add some new ones, to meet the requirements needed
to prove the decay of correlations and the logarithm law;
the main results proved are listed in
Theorem~\ref{thm:propert-singhyp-attractor} below.

Let $M$ be a $3$-dimensional compact riemanian manifold and
${\cal X}^r(M), \, r \geq 1,\,$ be set of $C^r$ vector
fields (or flows) defined on $M$.  Given a compact invariant
set $\Lambda$ of $X\in {\cal X}^r(M)$, we say that $\Lambda$
is \emph{isolated} if there exists an open set $U\supset
\Lambda$ such that
$$
\Lambda =\bigcap_{t\in \re}X_t(U).
$$
If $U$ above can be chosen such that $X_t(U)\subset U$ for
$t>0$, we say that $\Lambda$ is an \emph{attracting set}.

Given $X \in {\cal X}^r(M)$, a point $x \in M$ is
{\em{regular}} if $X(x) \neq 0$.  In this case we refer to
its orbit as a {\em{regular $X$-orbit or regular orbit for
    short}}.

Given $p \in M$, we define $\omega_X(p)$ as the set of
accumulation points of the positive orbit $\{X_t(p); t \geq
0\}$ of $p$. We also define $\alpha_X(p)=\omega_{-X}(p)$,
where $- X$ is the time reversed vector field.

A subset $\Lambda \subset M$ is \emph{transitive} if it has
a full dense orbit, that is, there is $p\in M$ such that
$\omega_X(p)=\Lambda=\alpha_X(p)$.
%  We say that an attracting set $\Lambda$ is
%\emph{transitive} if it coincides with the $\omega$-limit
%set of a regular orbit.
\begin{definition}\label{def:attractor}
  An \emph{attractor} is a transitive attracting set, and a
  \emph{repeller} is an attractor for the reversed vector
  field $-X$.
\end{definition}

An attractor, or repeller, is \emph{proper} if it is not the
whole manifold.  An invariant set of $X$ is
\emph{non-trivial} if it is neither a periodic orbit nor an
equilibrium of $X$. Recall that a point $\sigma\in M$ is an
equilibrium of $X$ if $X(\sigma)=0$.

\begin{definition}
\label{d.dominado}
Let $\Lambda$ be a compact invariant set of $X \in {\cal
  X}^r(M)$ , $c>0$, and $0 < \lambda < 1$.  We say that
$\Lambda$ has a $(c,\lambda)$-dominated splitting if the
bundle over $\Lambda$ can be written as a continuous
$DX_t$-invariant sum of sub-bundles
$$
T_\Lambda M=E^1\oplus E^2,
$$
such that for every $t > 0$ and every $x \in \Lambda$, we have
\begin{equation}\label{eq.domination}
\|DX_t \mid E^1_x\| \cdot
\|DX_{-t} \mid E^2_{X_t(x)}\| < c \cdot \lambda^t.
\end{equation}
We say that $\Lambda$ is \emph{partially hyperbolic} if it
has a $(c,\lambda)$-dominated splitting such that the
sub-bundle $E^1$ is uniformly contracting, that is, for some
$c>0$ and every $t > 0$ and each $x \in \Lambda$ we have
\begin{equation}\label{contrai}
\|DX_t \mid E^1_x\| < c \lambda^t.
\end{equation}
In this case we denote the one-dimensional bundle $E^1$ by
$E^s$ and the two-dimensional bundle $E^2$ by $E^{cu}$. We
refer to $E^s$ as the contracting direction and to $E^{cu}$
as the central or center-unstable direction of the
splitting.
\end{definition}

The next proposition is an immediate consequence of
\cite[Theorem 1]{Goum07} and will simplify many of the
arguments.

\begin{proposition} \label{Nikolaz} There exists an adapted
  Riemannian metric $\|\cdot\|_0$, equivalent to the
  original one, and $\lambda\in(0,1)$ such that
  \begin{align} \label{domino} \|DX_t \mid E^1_x\|_0
    \cdot \|DX_{-t} \mid E^2_{X_t(x)}\|_0 &< \lambda^t ,
    \quad\text{and}
    \\
    \|DX_t\mid E^1_{x}\|&< \lambda^t,\label{unifcontr}
  \end{align}
  for all $ t\geq 0$, and all $x\in\Lambda$.
\end{proposition}
Throughout the remaining of the paper we assume that the
Riemannian metric is an adapted metric and the first
sub-bundle $E^1$ is one-dimensional. For simplicity we
denote the adapted metric by $\| \cdot \|$.
% , and the flow direction will be contained in
%the second sub-bundle $E^2$, that we call \emph{central
% direction} and denote by $E^{cu}$.

For $x\in \Lambda$ and $t\in\re$ we let $J_t^{cu}(x)$ be the
absolute value of the determinant of the linear map
$$
DX_t \mid E^{cu}_x:E^{cu}_x\to E^{cu}_{X_t(x)}.
$$
We say that the sub-bundle $E^{cu}_\Lambda$ of the
partially hyperbolic invariant set $\Lambda$ is \emph{volume
  expanding} if
\begin{equation}\label{volume}
J_t^{cu}(x)\geq K \, e^{\theta t} \quad \mbox{for every
$x\in \Lambda$ and $t\geq 0$ and some $\theta > 0$}.
\end{equation}

In this case we say that $E^{cu}_\Lambda$ is {\em
  $(K,\theta)$-volume expanding} to indicate the dependence
on $K,\theta$.  This condition is weaker than the uniform
exponential expansion along the central direction.
% where uniform exponential expansion on the central direction means:
% there are $ K, \lambda > 0$ not depending on $x \in \Lambda$
% and $t > 0$ such that $\|DX_t|E^{cu}_x\| \geq K \cdot e^{\lambda t}.$
The Geometric Lorenz attractor
has a volume expanding central direction that is not uniformly
expanding \cite{AraPac2010}.

The domination condition \eqref{eq.domination}, together
with the volume expanding condition (\ref{volume}) along the
central direction, imply that the direction of the flow is
contained in the central bundle $E^{cu}$ \cite[Lemma 6.1, pg
163]{AraPac2010}.

%Recall that $\sigma$ is an equilibrium or equivalently, a singularity, of $X$ if $X(\sigma)=0$ and
Recall that an equilibrium $\sigma$ of $X$ is hyperbolic
only if the real part of every eigenvalue of $DX(\sigma)$ is
distinct of zero, see \cite[Section 2.1,pp 6]{AraPac2010}.
Recall also the definition of a special type of equilibrium
of a vector field $X$ in a $3$-manifold.
\begin{definition}
 \label{lorenz-like}
 We say that an equilibrium $\sigma$ of a $3$-vector field $X$ is
 Lorenz-like if the eigenvalues $\lambda_i$, $1\leq i \leq
 3$, of $DX(\sigma)$ are real and satisfy
\begin{align}\label{eq:lorenz-like}
  \lambda_2<\lambda_3<0<-\lambda_3<\lambda_1.
\end{align}
\end{definition}
If $\sigma$ is a Lorenz-like equilibrium, letting
$\alpha=-\frac{\lambda_3}{\lambda_1}$ and $\beta =-\frac{\lambda_2}{\lambda_1}$ we obtain
  $0< \alpha < 1 < \beta.$
%  and second that there is $\gamma > 0$ such that
%\begin{equation}\label{condicaofolheacaoclasseC1mais holder}
%\alpha + \gamma < \beta.
%\end{equation}

\begin{definition}[Singular Hyperbolic Attractor]
\label{d.singularset}
Let $\Lambda$ be a compact invariant set of $X \in {\cal
  X}^r(M)$.  We say that $\Lambda$ is a
\emph{singular-hyperbolic set} for $X$ if all the
equilibria of $\Lambda$ are hyperbolic, and $\Lambda$ is
partially hyperbolic with volume expanding central
direction. A singular-hyperbolic set which is also an attractor
will be called a {\em {singular-hyperbolic attractor}}.
\end{definition}

The next result is the content of \cite[Lemma 5.27]{AraPac2010}:

\begin{lemma}\label{singularidadeLorenz-like}
  Let $\Lambda$ be a singular-hyperbolic attractor of a
  $3$-dimensional vector field $X$.  Then, every equilibrium
  $\sigma$ properly accumulated by regular orbits within
  $\Lambda$ is Lorenz-like either for $X$ or for $-X$.
\end{lemma}

%Given $p \in M$, we define $\omega_X(p)$ as the set of accumulation points
%of the positive orbit $\{X_t(p); t \geq 0\}$ of $p$. We also define
%$\alpha_X(p)=\omega_{-X}(p)$.
%
%
%A subset $\Lambda \subset M$ is transitive if it has a full dense orbit, that is, there is $p\in M$
%such that $\omega_X(p)=\Lambda=\alpha_X(p)$.
%\begin{definition}
%  \label{def:robust-transit}
  An isolated set $\Lambda$ of a $C^1$ vector field $X$ is
 robustly transitive if it has an open neighborhood $U$
  such that
$$
\Lambda_Y(U)=\bigcap_{t\in\RR}Y_t(U)
$$
is both transitive and non-trivial (i.e., \emph{it is neither
  an equilibrium point nor a periodic orbit}) for any vector field
$Y$ $C^1$-close to $X$.
%\end{definition}
Roughly speaking, $\Lambda$ is robustly transitive
 if it can not be destroyed by small $C^1$ perturbations.

 Morales, Pacifico, and Pujals proved in~\cite{MPP04} that
 any \emph{transitive robust invariant set} of a
 $3$-dimensional flow containing some equilibrium is a
 singular-hyperbolic {\it{attractor or repeller}}. In the
 absence of equilibria, robustness implies uniform
 hyperbolicity.  The most meaningful examples of
 singular-hyperbolic attractors are the Lorenz attractor
 \cite{Lo63} and the so called Geometric Lorenz attractor
 \cite{ABS77,GW79}.

The main results we explain and derive in this section that
will be used in the rest of the paper are stated in the
following.

\begin{theorem}
  \label{thm:propert-singhyp-attractor}
  For an open and dense subset of $C^2$ vector fields $X$ having
  a singular hyperbolic attractor $\Lambda$ on a
  $3$-manifold, there exists a finite family $\Xi$ of
  cross-sections and a global ($n$-th return) Poincar\'e map
  $R:\Xi_0\to\Xi$, $R(x)=X_{\tau(x)}(x)$ such that
  \begin{enumerate}
  \item the domain $\Xi_0=\Xi\setminus\Gamma$ is the entire
    cross-sections with a family $\Gamma$ of finitely many
    smooth arcs removed and $\tau:\Xi_0\to[\tau_0,+\infty)$
    is a smooth function bounded away from zero by some
    uniform constant $\tau_0>0$.
  \item
    We can choose coordinates on $\Xi$ so that the map $R$
    can be written as $F:\tilde Q\to Q$, $F(x,y)=(T(x),G(x,y))$,
    where $Q=\II\times\II$, $\II=[0,1]$ and
    $\tilde Q=Q\setminus\Gamma_0$ with $\Gamma_0=\cC\times\II$
    and $\cC=\{c_1,\dots,c_n\}\subset\II$ a finite set of
    points.
  \item The map $T:\II\setminus\cC\to\II$ is $C^{1+\alpha}$
    piecewise monotonic with $n+1$ branches defined on the
    connected components of $\II\setminus\cC$ and has a
    finite set of a.c.i.m., $\mu^i_T$. Also $\inf|T'|>1$ where it is
    defined, $1/|T'|$ has universal bounded $p$-variation
    and then $d\mu^i_T/dm$ has bounded $p$-variation.
  \item
    The map $G:\tilde Q\to\II$ preserves and uniformly contracts
    the vertical foliation
    $\cF=\{\{x\}\times\II\}_{x\in\II}$ of $Q$: there
    exists $0<\lambda<1$ such that $
    \dist(G(x,y_1),G(x,y_2)) \leq \lambda \cdot |y_1-y_2|$
    for each $y_1, y_2 \in \II$. In addition, the map $G$
    satisfies $\varsq(G)<\infty$.
  \item
    The map $F$ admits a finite family of  physical probability measures
    $\mu^{i}_{F}$ which are induced by $\mu^i_T$ in a standard
    way. The Poincar\'e time $\tau$ is integrable both with
    respect to each $\mu^{i}_{F}$ and with respect to the
    two-dimensional Lebesgue area measure of $Q$.
  \item
    Moreover if, for all singularities $\sigma\in\Lambda$,
    we have the eigenvalue relation
    $-\lambda_2(\sigma)>\lambda_1(\sigma)$, then the second
    coordinate map $G$ of $F$ has a bounded partial
    derivative with respect to the first coordinate, i.e.,
    there exists $C>0$ such that $|\partial_x G(x,y)|<C$ for
    all $(x,y)\in(\II\setminus\{c_1,\dots, c_n\})\times\II$.
  \end{enumerate}
\end{theorem}

\subsection{Preliminary results}\label{s.reliminary}
Next we give the notions and establish notations needed to
describe the properties of singular-hyperbolic attractors
stated in Theorem~\ref{thm:propert-singhyp-attractor}.

\subsubsection{Stable foliations on cross-sections}\label{folheacaosecao}

Hereafter, $\Lambda$ is a singular-hyperbolic attractor of
$X\in {\cal X}^2(M)$ with invariant splitting $T_\Lambda M =
E^{s}\oplus E^{cu}$ with $\dim E^{cu}=2$.  Let
$\tilde{E}^s\oplus \tilde{E}^{cu}$ be a continuous extension
of this splitting to a small neighborhood $U_0$ of
$\Lambda$.  For convenience, we take $U_0$ to be forward
invariant and such that $\cap_{t\geq 0}X_t(U_0)=\Lambda$.
We will see in Subsection~\ref{sec:extens-laminat-wss}, that
$\tilde{E}^s$ may be chosen invariant under the derivative.
In general, the extension $\tilde{E}^{cu}$ of the
center-unstable direction cannot be assumed to be invariant; see
\cite{HPS77} and also \cite[Appendix B]{BDV2004}. However, we can
always consider a cone field around it on $U_0$
$$
C^{cu}_a(x)=\{v=v^s+v^{cu}: v^s\in \tilde{E}^s_x
\text{ and }
v^{cu}\in\tilde{E}^{cu}_x
\text{ with } \|v^s\|\le a\cdot \|v^{cu}\|\}
$$
which is forward invariant for some $a>0$, that is, there is
a large $T > 0$ depending on $a$, but not depending on $x\in
U_0$, such that
\begin{equation}
\label{eq.cone3}
DX_t(C^{cu}_a(x)) \subset C^{cu}_a (X_t(x))
\quad\text{for all  $t\geq T$.}
\end{equation}
Moreover, we may take $a>0$ arbitrarily small, reducing
$U_0$ if necessary.  For notational simplicity, we write
$E^s$ and $E^{cu}$ for $\tilde E^s$ and $\tilde E^{cu}$ in
all that follows.

Next let us recall a few classical facts about partially
hyperbolic systems, especially existence of stable and
strong-stable foliations, and center-unstable foliations.
The standard reference are \cite[Theorems 4.1, 5.1 and
5.5]{HPS77} and \cite[Theorem IV.1]{Sh87}.

Before stating the appropriate result we need to review some
terms in~\cite{HPS77}; namely, the notion of immediate
relative $\rho$ pseudo-hyperbolic splitting.  To define
this, recall that if $L:\mathbb{R}^d\to \mathbb{R}^d$ is an
injective linear map and $G\subset \mathbb{R}^d$ is a
subspace, then the {\it conorm of L restricted to G} is
$$m(L|G) \colon =\inf_{v\neq 0\, ,v\in
  G}\frac{\|Lv\|}{\|v\|}.$$ The conorm gives the minimal
expansion of $L$ on $G$.

Let $\{X_t \colon M\to M\}_{t\in\RR}$ be a flow and $\Omega$
be an $X$-invariant compact set.  A splitting
$T_{\Omega}(M)=E\oplus F$ is {\it immediate relatively
  $\rho$ pseudo-hyperbolic} relative to $X$ if there exists
a continuous function $\rho \colon \Omega \rightarrow
\mathbb{R}^+$ such that
$$\|DX_1|_{E(x)}\|<\rho(x)<m(DX_{-1}|_{F(x)}) \quad\mbox{for
  all}\quad x\in\Omega.$$

Using Proposition \ref{Nikolaz} we have the following

\begin{lemma}\label{l.new}
  Let $X$ be a $C^2$ flow and $\Lambda$ be a
  % and $U_0$ be a neighborhood of a
  singular-hyperbolic attractor with splitting $E^s \oplus
  E^{cu}$. Then $E^s \oplus E^{cu}$ is an immediate
  relatively $\rho$ pseudo-hyperbolic splitting over
  $\Lambda$ relative to $X$ for some continuous function
  $\rho\colon \Lambda \to \mathbb{R}^+$.
\end{lemma}

\begin{proof}
  The domination assumption ensures that
  $\|DX_1\mid_{E(x)}\| < \lambda m(DX_{-1}\mid_{F(X_1(x))})$
  for all $x\in\Lambda$, so
  \begin{align*}
    m(DX_{-1}\mid_{F(X_1(x))})
    >
    \frac{\|DX_1\mid_{E(x)}\|}{\lambda}:=\rho(x)
    >
    \|DX_1\mid_{E(x)}\|
  \end{align*}
  and the above function $\rho:\Lambda\to\RR^+$ is
  continuous since a dominated splitting is continuous.
\end{proof}

We then have the following restatement of Theorems
\cite[Theorem IV.1]{Sh87} and ~\cite[Theorem ~5.5]{HPS77}:

\begin{theorem}\label{contributions}
  Let $X$ be a $C^2$ flow and $\Lambda$ be a compact
  $X$-invariant singular-hyperbolic attractor $\Lambda$
  having a dominated splitting $E^s\oplus E^{cu}$.  There
  are $\epsilon >0 $ and $\lambda \in (0,1)$ such that for
  every point $x \in\Lambda$ there are two embedded discs
  $W^{ss}_\epsilon (x)$ and $W^{cu}_\epsilon (x)$ tangent at
  $x$ to $E^s(x)$ and $E^{cu}(x)$ respectively, and
  satisfying, for all $t>0$
\begin{enumerate}
\item[(1)] $W^{ss}_\epsilon (x)$ is a $C^2$ embedded disc,
\item[(2)] $W^{ss}_\epsilon (x)= \{ q \in M;
  \dist(X_t(x),X_t(q)) \leq \epsilon
\text{  and  }
\lim_{t\to\infty}\dist(X_t(x),X_t(q))/\lambda^t= 0 \},$
\item[(3)] $X_t(W^{ss}_\epsilon (x))\subset W^{ss}_\epsilon (X_t(x))$,
\item[(4)] The embedding $W^{ss}_\epsilon (x)$ depends
  $C^2$ smoothly on $x$ in the following sense: there is a
  neighborhood $\cV$ of $x$ in $\Lambda$ and a continuous
  map $\gamma: \cV\cap\Lambda \to \mathrm{Emb}^2(\II, M)$
  such that
$\gamma(x)(0)=x$ and $\gamma(x)(\II)=W^{ss}_\epsilon(x),$
where $\mathrm{Emb}^2(\II, M)$ is the collection of all
embeddings $\phi:\II\to M$ endowed with the $C^2$ distance;
\item[(5)] $X_t(W^{cu}_\epsilon(x))\cap B_\epsilon(x)
  \subset W^{cu}_\epsilon(X_t(x))$, where
$
B_\epsilon(x)=\{y \in M; \dist(x,y) < \epsilon \},
$
\item[(6)] The $W^{cu}_\epsilon(x)$ depend $C^2$ smoothly on
  $x$ as in \mbox{(4)}. That is, there exists a continuous
  map $\gamma:\cV\cap\Lambda \to \mathrm{Emb}^2(\DD, M)$
  such that $\gamma(x)(0)=x$ and
  $\gamma(x)(\DD)=W^{cu}_\epsilon(x)$, where
  $\mathrm{Emb}^2(\DD, M)$ is the collection of all
  embeddings $\phi:\DD\to M$ from the unit $2$-disk $\DD$
  into $M$ endowed with the $C^2$ distance.
\end{enumerate}
\end{theorem}

Note that, for $x\in\Lambda$, $W^{ss}_\epsilon (x)$ and
$W^{cu}_\epsilon (x)$ are embedded discs, and so,
sub-manifolds of $M$.  We refer to $W^{ss}_\epsilon (x)$ as
local {\emph{strong-stable}} manifold and to
$W^{cu}_\epsilon (x)$ as local {\emph{center-unstable}}
manifold.  Since $E^s$ is uniformly contracting we have that
$W^{ss}_\epsilon(x)$ is uniquely defined.  But we stress
that the center-unstable $W^{cu}_\epsilon(x)$ manifold {\em
  is not} unique without further assumptions; see
\cite{AbrRobKel67}.

The set
\begin{align}\label{eq:center-stable-manifold}
  W^s(x) = \bigcup_{t\in\real} W^{ss}_\epsilon(X_t(x))
  \subset \bigcup_{t\in\real} X_t(W^{ss}_\epsilon(x))
\end{align}
is called the {\em stable manifold} at $x\in\Lambda$.  The
proof that $W^s(x)$ is a manifold is contained in
\cite[Theorem 5.5]{HPS77}.

Denoting $E^{cs}_x=E^s_x\oplus E^{X}_x$, where $E^X_x$ is
the direction of the flow at $x$, it follows that
\begin{align}\label{eq:Ecs}
  %T_x W^{ss}(x)=E^s_x \quad\text{and}\quad
  T_x W^{s}(x)=E^{cs}_x\,.
\end{align}
%We fix a small $\vep$  once and for all. Then we call
%$W^{ss}_{\vep}(x)$ the local \emph{strong-stable manifold}
%and $W^{cu}_{\vep}(x)$ the local \emph{center-unstable
%manifold} of $x$.

%About the class of differentiability of these manifolds, see \cite[Theorems 4.1 and 5.1]{HPS67}.

%Denoting by ${\cal O}(x)$ the $X$-orbit of $x\in \La$, and
%$E^{cs}_x=E^s_x\oplus E^{X}_x$, where $E^X_x$ is the direction
%of the flow at $x$, it follows that
%$$
%W^{s}_\vep({\cal O }(x))=\cup_{t\in\re}W^{ss}_\vep(X_t(x))
%$$
%is also an invariant $C^1$-manifold tangent to $E^{cs}(x)$. This
%is called the (local) \emph{stable manifold} at the orbit of $x$.

\subsubsection{Extension of the lamination
  $\{W^{ss}_{\epsilon}\}_{x\in\Lambda}$ to a contracting
  invariant foliation in a neighborhood of $\Lambda$}
\label{sec:extens-laminat-wss}

Here we show that the collection of $C^2$ strong-stable leaves
through the points of $\Lambda$, depending continuously on
the base point, which is known as a lamination, can be
extended to an invariant foliation (in the usual sense from
Differential Topology) of a open neighborhood of $\Lambda$,
whose leaves are $C^2$ submanifolds and whose foliated
charts are of class $C^1$. In addition, these leaves are
uniformly contracted by the action of the flow. The argument
we present below follows \cite[Appendix 1]{PT93} closely.

Let us fix a neighborhood $U_0$ of the singular-hyperbolic
attractor $\Lambda$ on the manifold $M$ such that the
closure of $X_t(U_0)$ is contained in $U_0$ for all
$t>0$. We consider the following family of directions
through the points of $U_0$
\begin{align*}
  D=\{(x,\ell): x\in U_0, \ell\in\PP^1(T_xM)\}.
\end{align*}
This set is clearly a smooth manifold of dimension $5$.  The
time-one map of the flow $f=X_1$ induces a map $\psi:D\to
\psi(D)$ given by
\begin{align*}
  \psi(x,\ell)=(f(x),Df(x)\cdot\ell), \quad (x,\ell)\in D.
\end{align*}
Naturally, this map $\psi$ is of class $C^1$ if $f$ is of
class $C^2$.
We note that the subset $\Omega:=\{(x,\ell)\in D:
x\in\Lambda, \ell=E^{ss}_x\}$ is compact and fixed by
$\psi$. We claim that it is also a partially hyperbolic
subset for $\psi$. To prove this, we consider the inverse
map $\vfi(x,\ell)=(f^{-1}(x),Df^{-1}(x)\cdot\ell)$.

Indeed, for every $(x,\ell)\in\Omega$ we can calculate the
derivative $D\vfi$ of $\vfi$ at $(x,\ell)\in\Omega$ using
coordinates provided by the splitting of $T_\Lambda M$, as
follows. Directions on $\PP^1(T_xM)$ near the stable
direction can be seen as graphs of a linear map
$\ell:E^{ss}_x\to E^{cu}_x$ which in fact is given by the
vector $\ell(1)\in E^{cu}_x$, since $E^{ss}_x$ is
one-dimensional. The action of $Df^{-1}(x)$ on the directions of
$\PP^1(T_xM)$ can be represented by the graph transform
$\ell^\prime=Df^{-1}\mid_{E^{cu}_{x}}\circ \ell \circ
(Df^{-1}\mid_{E^{ss}_{x}})^{-1}:E^{ss}_{f^{-1}(x)}\to
E^{cu}_{f^{-1}(x)}$. So in this coordinates we have
$\vfi(x,\ell)=(f^{-1}(x),\ell^\prime)$. Then we have, since a
tangent direction at $\ell\in\PP^1(T_xM)$ is another linear
map $\xi:E^{ss}_x\to E^{cu}_x$ given by the vector $\xi(1)$
\begin{align*}
  D\vfi(x,\ell)\cdot(v,\xi)
  =
  (Df^{-1}(x)\cdot v, Df^{-1}(x)\cdot\xi)
\end{align*}
which we may represent by the following isomorphism of
$\RR^5$
\begin{align*}
  \begin{pmatrix}
    Df^{-1}(x) & 0 \\ 0 & (Df^{-1}\mid_{E^{ss}_{x}})^{-1} \cdot Df^{-1}\mid_{E^{cu}_{x}}
  \end{pmatrix}:
  (E^{ss}_{x} \oplus E^{cu}_{x})\oplus E^{cu}_{x}
  \to
    (E^{ss}_{f^{-1}(x)} \oplus E^{cu}_{f^{-1}(x)})\oplus E^{cu}_{f^{-1}(x)}.
\end{align*}
We have now using the domination assumption from \eqref{domino}
\begin{align*}
  \|(Df^{-1}\mid_{E^{ss}_{x}})^{-1} \cdot
  Df^{-1}\mid_{E^{cu}_{x}}\|
  \le
  \|Df\mid_{E^{ss}_{f^{-1}(x)}}\|\cdot\|(Df\mid_{E^{cu}_{f^{-1}(x)}})^{-1}\|
  <\lambda
\end{align*}
and by condition \eqref{contrai} we get
\begin{align*}
  \|(Df^{-1}\mid_{E^{ss}_{x}})^{-1} \cdot
  Df^{-1}\mid_{E^{cu}_{x}}\|
  <
  \lambda\cdot \|(Df\mid_{E^{cu}_{f^{-1}(x)}})^{-1}\|
  <
  m(Df^{-1}(x))
\end{align*}
since the assumption \eqref{domino} ensures that
$\|(Df\mid_{E^{cu}_{f^{-1}(x)}})^{-1}\|
<
\lambda\cdot\|Df\mid_{E^{ss}_{f^{-1}(x)}}\|^{-1}$
which is the same as
\begin{align*}
  \|Df^{-1}\mid_{E^{cu}_{x}}\|
  <
  \lambda\cdot
  \|Df^{-1}\mid_{E^{ss}_{x}}\|
  <
  \|Df^{-1}\mid_{E^{ss}_{x}}\|.
\end{align*}
This shows the partial hyperbolicity of $\Omega$.  Hence the
map $D\vfi$ is \emph{immediate relative
  $\rho$-pseudo-hyperbolic} for some function $\rho$
strictly smaller than $1$ over $\Omega$. We may now use
\cite[Theorem IV.1]{Sh87} and ~\cite[Theorem ~5.5]{HPS77} to
obtain a $C^1$ center-unstable manifold $W^c(x)$ for $\vfi$
through each point $(x,E^{ss}_x)\in\Omega$, which is a
center-stable manifold for $\psi$, with dimension $3$ and
tangent to $T_x M\times E^{ss}_x$ at $(x,E^{ss}_x)$.

This manifold projects to a neighborhood of $x\in\Lambda$ on
$M$ through the canonical projection of $D$ on the first
coordinate. This means that on a neighborhood $U\subset U_0$
of $\Lambda$ we can define a field of directions
$\{\ell(y)\}_{y\in U}$ such that $(y,\ell(y))\in W^c(x)$ for
some $x\in\Lambda$. Integrating these direction we obtain
$C^2$ one-dimensional submanifolds $W^{ss}(y)$ passing
through $y\in U$ (note that the field of direction is $C^1$
smooth because $W^c$ was $C^1$). Since the lamination $W^c$
is $\psi$-invariant, we can deduce that $f(W^{ss}(y))\subset
W^{ss}(f(y))$ for all $y\in U\cap f^{-1}(U)$.

\emph{We have obtained a one-dimensional foliation of a
neighborhood $U$ of $\Lambda$ which extends the
strong-stable lamination provided by
Theorem~\ref{contributions}. For small enough $U$ the leaves
of this foliation are uniformly contracted by a rate close
to $\lambda$ under the action of $f$.}

This results extend to the action of $X_t$ in a standard
way: we have the same conclusions for the diffeomorphism
$X_t$ for $t>0$ in the place of $f=X_1$.

\begin{corollary}
  \label{cor:extension-stable-foliation}
  For any compact invariant subset $\Lambda$ of a $C^2$ flow
  $X_t$ which is partially hyperbolic, that is, $\Lambda$
  satisfies conditions~\eqref{domino} and \eqref{contrai}
  for a continuous $DX_t$-invariant splitting $T_\Lambda
  M=E^s\oplus E^{cu}$ with one-dimensional stable direction,
  there exists a neighborhood $U$ of $\Lambda$ in the
  ambient manifold $M$ where a extension $\cF^s(x)$ of the
  local stable lamination
  $\{W^{ss}_\epsilon(x)\}_{x\in\Lambda}$ is defined, is a
  locally $X_t$-invariant foliation and its leaves are
  uniformly contracted by $X_t$, for all $t>0$. Each leaf is
  a $C^2$ one-dimensional sub-manifold of $M$.
\end{corollary}

The dependence of $W^{ss}_\epsilon(x)$ on $x$ is similar to
item (4) of Theorem~\ref{contributions} but in the $C^1$
topology. We stress that the smoothness of $x\mapsto
W^{ss}_\epsilon(x)$ is in general \emph{not} related to the
differentiability of the foliation
$\cF^{ss}=\{W^{ss}_\epsilon(x)\}_{x\in U}$ of the
neighborhood $U$ of $\Lambda$; see e.g. \cite{PSW97}.

% to $C^{1+\alpha}$ for some
% $\alpha>0$ along a cross-section to the flow, which is
% crucial in what follows, as explained in the following
% section.

%%%%%%%%%%%%%%%%%%%%%%%%%%%%%%%%%%%%%%%%%%%%%%%%%%%%%%%%%%%%%

\subsection{Cross-sections and Poincar\'e maps}
\label{sec:cross-sections-poinc}

Now let $\Sigma$ be a \emph{cross-section} to the flow, that
is, a $C^2$ embedded compact disk transverse to $X$ at every
point $x \in \Sigma$. We assume from now on that
cross-sections are contained in the open neighborhood of
$\Lambda$ where a contracting foliation $\cF^{ss}$ which
extends the strong-stable lamination through the points of
$\Lambda$ is defined.

For $x\in\Sigma$ we define $W^s(x,\Sigma)$ to be the
connected component of $\cF^{sc}(x)\cap\Sigma$ that contains
$x$, where $\cF^{cs}(x):=\cup_{t\in\RR}X_t(\cF^s(x))$ is the
central-stable leaf obtained from the strong-stable leaf
$\cF^s(x)$ in a manner similar to
(\ref{eq:center-stable-manifold}).  Since the flow
$(X_t)_{t\in\RR}$ is $C^2$, $W^s(x,\Sigma)$ is a $C^2$
co-dimension one embedded curve for every $x\in
\Sigma$. These leaves form a foliation $\cF^{\,s}_\Sigma$ of
$\Sigma$.

From the celebrated work of Anosov~\cite{An67} and more
recent developments in the partially hyperbolic setting by
Pugh-Shub~\cite{PS72} and Brin-Pesin~\cite[Theorem
3.1]{BP74}, it is known that the holonomies (projection
along leaves) between pairs of transverse surfaces to
$\cF^{ss}$ admit a H\"older Jacobian with respect to
Lebesgue induced measure. This naturally implies a similar
statement for holonomies transverse to $\cF^{cs}$. 

In this setting the holonomy (projection) between pairs of
transverse curves to $\cF^{\,s}_{\Sigma}$ along the lines of
$\cF^{\,s}_{\Sigma}$ can be seen as maps between intervals
of the real line having a H\"older Jacobian with respect to
Lebesgue measure. Hence these holonomies are $C^{1+\alpha}$
maps, for some $0<\alpha<1$ which depends on $X$ only; see
e.g.~\cite[Section 2.7.2]{AraPac2010}.

In this case the leaves $W^s(x,\Sigma)$, for $x\in \Sigma$,
define a foliation $\cF^{\,s}_{\Sigma}$ of $\Sigma$ whose
transversal smoothness is H\"older-$C^1$.% ; see \cite[Section
% 2.7.2]{AraPac2010}, where this property is deduced from the
% fact that the leaves of $W^s(x,\Sigma)$ are a codimension
% one foliation in $\Sigma$. For more on this, see
% e.g. \cite{PSW97} and references therein.

\begin{remark}\label{r.foliated}
  Given a cross-section $\Sigma$ there is no loss of
  generality in assuming that it is the image of the square
  $\II\times\II$ by a $C^{1+\alpha}$ diffeomorphism $h$, for
  some $0<\alpha<1$, which sends vertical lines inside
  leaves of $\F^{s}_{\Sigma}$.  We denote by
  $\interior(\Sigma)$ the image of $(0,1)\times(0,1)$ under
  the above-mentioned diffeomorphism, which we call the
  \emph{interior} of $\Sigma$.

  % We note that the $C^2$ assumption on $h$ implies in
  % particular that the curvature of $\Sigma$ is bounded.

  We also assume that each cross-section $\Sigma$ is
  contained in $U_0$, so that every $x\in\Sigma$ is such
  that $\omega(x)\subset \Lambda$. 
  From now on we always assume that cross-sections are of
  this kind.
\end{remark}

Given any two cross-sections $\Sigma$ and $\Sigma'$ to the
flow, a {\em {Poincar\'e map}} is a map defined by
$$
R: U \subset \Sigma \to \Sigma', \quad x \in U \mapsto
X_{t(x)}(x) \in \Sigma'
$$ (for a suitable hitting time $t$ which will be precised later).
We note that, in general, $R$ needs not correspond to the
first time the orbits of $\Sigma$ encounter $\Sigma'$, nor
it is defined everywhere in $\Sigma$.  If $R$ is defined at
$x\in \Sigma$, a time $t(x)>0$ so that $X_{t(x)}(x) \in
\Sigma'$ is called a {\em {Poincar\'e time}} of $x$.

The continuity of the flow implies that, if $R$ is defined
at $x\in \Sigma$, then it is defined in an open neighborhood
$U_x$ of $x$ in $\Sigma$ and it is a $C^1$ local
diffeomorphism, see \cite[Proposition 1.2, pp 94]{PM82}.

\subsubsection{Hyperbolicity of Poincar\'e maps}
\label{s.22}

Let $\Sigma$ be a cross-section to $X$ and
$R:\Sigma\to\Sigma'$ be a Poincar\'e map $R(y)=X_{t(y)}(y)$
to another cross-section $\Sigma'$ (possibly
$\Sigma=\Sigma'$), defined as above.

The splitting $E^s\oplus E^{cu}$ over $U_0$ induces a continuous
splitting $E_\Sigma^s\oplus E_\Sigma^{cu}$ of the tangent bundle
$T\Sigma$ to $\Sigma$ (and analogously for $\Sigma'$),
defined by (recall~\eqref{eq:Ecs} for the use of $E^{cs}$)
\begin{equation}\label{eq.splitting}
E_\Sigma^s(y)=E^{cs}_y\cap T_y{\Sigma}
\quad\mbox{and}\quad
E_\Sigma^{cu}(y)=E^{cu}_y\cap T_y{\Sigma}.
\end{equation}

The next result establishes that if the Poincar\'e time
$t(x)$ is sufficiently large then (\ref{eq.splitting})
defines a hyperbolic splitting for the transformation $R$ on
the cross-sections. Given a pair $\Sigma,\Sigma'$ of
cross-sections in $\Xi$, we write $\Sigma(\Sigma')$ for the
subset of points of $\Sigma$ whose Poincar\'e map is defined
and hits $\Sigma'$.

\begin{proposition}\label{p.secaohiperbolica}\cite[Proposition
  6.15, pp 172]{AraPac2010} Let $R:\Sigma(\Sigma')\to\Sigma'$
  be a Poincar\'e map as before with Poincar\'e time
  $t(\cdot)$.  Then $DR_x(E_\Sigma^s(x)) = E_\Sigma^s(R(x))$
  at every $x\in\Sigma(\Sigma')$ and $DR_x(E_\Sigma^{cu}(x)) =
  E_\Sigma^{cu}(R(x))$ at every $x\in\Lambda\cap\Sigma(\Sigma')$.

  Moreover, for every given $0<\lambda<1$ there exists
  $T_1=T_1(\Sigma,\Sigma',\lambda)>0$ such that if
  $t(\cdot)>T_1$ at every point, then
$$
\|DR \mid E^s_\Sigma(x)\| < \lambda \quad\text{and}\quad
\|DR \mid E^{cu}_\Sigma(x)\| > 1/\lambda \quad\text{at every
  $x\in\Sigma(\Sigma')$.}
$$
\end{proposition}

Given a cross-section $\Sigma$, a positive number $\rho$,
and a point $x\in \Sigma$, we define the unstable cone of
width $\rho$ at $x$ by
\begin{equation}
\label{cone}
C_\rho^u(x)=\{v=v^s+v^u : v^s\in E^s_\Sigma(x),\, v^u\in E^{cu}_\Sigma(x)
\mbox{ and } \|v^s\| \le \rho \, \|v^u\| \}
\end{equation}
(we omit the dependence on the cross-section in our
notations). We note that $C^u_\rho(x)=C^{cu}_\rho(x)\cap
T_x\Sigma$.

Let $0<\rho< 1$ be a small constant. In the following
consequence of Proposition~\ref{p.secaohiperbolica} (which
is itself a consequence of partial hyperbolicity for the
splitting $E^s\oplus E^{cu}$) we assume that the
neighborhood $U_0$ has been chosen sufficiently small,
depending on $\rho$ and on a bound on the angles between the
flow and the cross-sections.

\begin{corollary}\cite[Corollary 6.17, pp 173]{AraPac2010}
\label{ccone}
For $R:\Sigma\to\Sigma'$ as in
Proposition~\ref{p.secaohiperbolica}, with $t(\cdot)>T_1$\,,
and any $x \in\Sigma(\Sigma')$, we have
$$
DR(x) (C^u_{\rho}(x)) \subset C_{\rho/2}^u(R(x))
\quad\mbox{and}\quad
\| DR(x)(v)\| \ge \frac{\lambda^{-1}}2 \cdot \|v\|
\quad\mbox{for all}\quad v\in C^u_{\rho}(x).
$$
\end{corollary}
The proof of this corollary is based on the observation
that, for small $\rho>0$, the vectors in $C^u_\rho(x)$ can
be written as the direct sum of a vector in $E^{cu}_x$,
which is expanded at a rate $\lambda^{-1}$, with a vector in
$E^{cs}_x$, which is contracted at a rate $\lambda$. Hence,
for small $\rho$, the center-unstable component dominates
the stable component and the length of the vector is
increased at a rate close to $\lambda^{-1}$.

In this way we can always achieve an arbitrarily large
expansion rate along the directions of the unstable cone as
long as we take $\lambda$ sufficiently close to zero and,
consequently, a big enough threshold time $T_1$.

Let us introduce the following notion: a \emph{cu-curve} in
$\Sigma$ is a \emph{curve contained in a cross-section
  $\Sigma\in \Xi$ whose tangent direction $T_z\gamma$ is
  contained in a center-unstable cone} $C^u_\rho(z)\subset
T_z\Sigma$ for all $z\in\gamma$; see \eqref{cone} below.

\begin{remark}\label{rmk:boundary-cu-curve}
  The cone $C_a^{cu}(x)$ is defined at every $x\in\Sigma$
  and we can choose $\Sigma$ so that the diffeomorphism $h$
  sends horizontal lines into $cu$-curves, i.e., curves whose
  tangent directions are contained in the $cu$-cone at every
  point.
\end{remark}

\subsubsection{Adapted cross-sections}
\label{s.23}

The next step is to exhibit stable manifolds for Poincar\'e
transformations $R:\Sigma\to\Sigma'$. The natural candidates
are the intersections $W^s(x,\Sigma)=W^s(x)\cap\Sigma$ we
introduced previously. By construction, these leaves are
contracted by the action of the flow and so they are
contracted by the transformation $R$.  Moreover, as already
commented before, these intersections define a
$C^{1+\alpha}$ stable foliation $\cF^{\,s}_{\Sigma}$ of
$\Sigma$ with a H\"older-$C^1$ holonomy.  For our purposes
it is also important that this foliation be invariant:
\begin{equation}\label{eq.stableMarkov}
R(W^s(x,\Sigma)) \subset W^s(R(x),\Sigma')
\qquad \text{for every } x\in\Lambda\cap\Sigma(\Sigma').
\end{equation}
In order to have this we restrict our class of
cross-sections so that the center-unstable boundary is
disjoint from $\Lambda$.  We recall (see
Remark~\ref{r.foliated}) that we are considering
cross-sections $\Sigma$ that are diffeomorphic to the square
$\II\times\II$, with the vertical lines
$\{\eta\}\times\II$ being mapped to stable sets
$W^s(y,\Sigma)$.  The \emph{stable boundary}
$\partial^{s}\Sigma$ is the image of $\{0,1\}\times\II$.
The \emph{center-unstable (or $cu$-)boundary}
$\partial^{cu}\Sigma$ is the image of
$\II\times\{0,1\}$. The cross-section is
\emph{$\delta$-adapted} if
$$
d(\Lambda \cap \Sigma,\partial^{cu}\Sigma)> \delta,
$$
where $d$ is the intrinsic distance in $\Sigma$. We also
recall that, from Remark~\ref{rmk:boundary-cu-curve}, we
choose the cross-sections so that the $cu$-boundary is in
fact formed by $cu$-curves.

We call \emph{vertical strip} of $\Sigma$ the image
$h(J\times\II)$ for any compact subinterval $J$, where
$h:\II^2\to \Sigma$ is the coordinate system on
$\Sigma$ as in Remark~\ref{r.foliated}.  Notice that every
vertical strip is an $\delta$-adapted cross-section.

\begin{figure}[htpb]
  \centering
    \includegraphics[height=4cm]{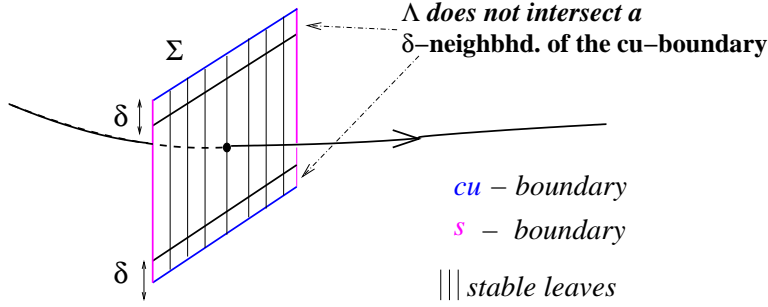}
  \caption{An adapted cross-section for $\Lambda$.}
  \label{fig:an-adapted-cross}
\end{figure}

\begin{lemma}\cite[Lemma 6.22, pp 177]{AraPac2010}\label{l.existeadaptada}
  Let $x \in \Lambda$ be a regular point, that is, such that
  $X(x)\neq 0$. Then there exists $\delta>0$ for which we
  can find a $\delta$-adapted cross-section $\Sigma$ at $x$.
\end{lemma}

Given cross-sections $\Sigma$ and $\Sigma'$ we set
$\Sigma(\Sigma')=\{ x\in\Sigma: R(x)\in\Sigma'\}$ the domain
of the return map from $\Sigma$ to $\Sigma'$.  The next
lemma establishes that if the cross-sections are adapted,
then we have the invariance property
\eqref{eq.stableMarkov}.

\begin{lemma}\cite[Lemma 6.23, pp 178]{AraPac2010}
\label{stablereturnmap}
Given $\delta>0$ and $\delta$-adapted cross-sections
$\Sigma$ and $\Sigma'$, there exists
$T_2=T_2(\Sigma,\Sigma')>T_1> 0$, where $T_1$ is as in
Proposition \ref{p.secaohiperbolica}, such that if
$R:\Sigma(\Sigma')\to\Sigma'$ defined by $R(z)=R_{t(z)}(z)$
is a Poincar\'e map with time $t(\cdot)>T_2$, then
\begin{enumerate}
\item $R\big(W^s(x,\Sigma)\big)\subset W^s(R(x),\Sigma')$
      for every $x\in\Sigma(\Sigma')$, and also
\item $d(R(y),R(z))\le \frac12 \, d(y,z)$ for every $y$,
      $z\in W^s(x,\Sigma)$ and $x\in\Sigma(\Sigma')$.
\end{enumerate}
\end{lemma}

This lemma provides a sufficient condition for having
partial hyperbolicity for the Poincar\'e return map. Indeed,
if $t>T_2>T_1$, then the stable leaves are sent strictly
inside stable leaves and uniformly contracted by the ratio
$1/2$; and the unstable cones on cross-sections are
preserved.

\subsubsection{Poincar\'e maps near Lorenz-like equilibria}
\label{sec:poincare-maps-near}

Here we consider the Poincar\'e maps of the flow near the
singularities.

We recall that, since the equilibria $\sigma=\sigma_k$ in
our setting are all Lorenz-like, the unstable manifold
$W^u(\sigma)$ is one-dimensional, and there is a
one-dimensional strong-stable manifold $W^{ss}(\sigma)$
contained in the two-dimensional stable manifold
$W^s(\sigma)$. By the smooth linearization results
provided by Hartman~\cite{Hartman02} in the absence of resonances,
orbits of the flow in a small neighborhood $U_\sigma$ of the
given equilibrium $\sigma$ are solutions of the following
linear system, modulo a smooth change of coordinates:
\begin{align}\label{eq:LinearLorenz}
(\dot x, \dot y, \dot z)
=
(\lambda_1 x,\lambda_2 y, \lambda_3 z)
\quad
\text{thus}\quad
X_t(x_0,y_0,z_0)=
(x_0e^{\lambda_1t}, y_0e^{\lambda_2t}, z_0e^{\lambda_3t}),
\end{align}
with $\lambda_2<\lambda_3<0<-\lambda_3<\lambda_1$.

More precisely, let us consider and  use the following smooth
linearization result.

\begin{theorem}
  \label{thm:smooth-linear}
  Let $n\in\ZZ^+$ be given. Then there exists an integer
  $N=N(n)\ge2$ such that: if $\Gamma$ is a real non-singular
  $d\times d$ matrix with eigenvalues
  $\gamma_1,\dots,\gamma_d$ satisfying
  \begin{align}\label{eq:non-resonance}
    \sum_{i=1}^d m_i \gamma_i \neq \gamma_k
    \quad
    \text{for all}
    \quad
    k=1,\dots, d
    \qand%\text{and}\quad
    2\le\sum_{j=1}^d m_j\le N
  \end{align}
  and if $\dot\xi=\Gamma\xi+\Xi(\xi)$ and
  $\dot\zeta=\Gamma\zeta$, where $\xi,\zeta\in\RR^d$ and
  $\Xi$ is of class $C^N$ for small $\|\xi\|$ with
  $\Xi(0)=0, \partial_\xi\Xi(0)=0$; then there exists a
  $C^n$ diffeomorphism $R$ from a neighborhood of $\xi=0$ to
  a neighborhood of $\zeta=0$ such that $R\xi_t
  R^{-1}=\zeta_t$ for all $t\in\RR$ and initial conditions
  for which the flows $\zeta_t$ and $\xi_t$ are defined in
  the corresponding neighborhood of the origin.
\end{theorem}

\begin{proof}
See \cite[Theorem 12.1, p. 257]{Hartman02}.
\end{proof}

We recall that, in general, hyperbolic singularities are
only linearizable by an at most H\"older homeomorphism
according to the standard Hartman-Grobman Theorem~\cite{PM82,robinson1999}.

By Theorem \ref{thm:smooth-linear} , hence it is enough for us to choose the eigenvalues
$(\lambda_1,\lambda_2, \lambda_3)\in\RR^3$ of $\sigma$ satisfying a
\emph{finite set of non-resonance relations}
\eqref{eq:non-resonance} for a certain $N=N(2)$ and for each
singularity $\sigma_k$ in $\Lambda$. For this condition
defines an open and dense set in $\RR^3$ and so all small
$C^1$ perturbations $Y$ of the vector field $X$ will have a
singularity whose eigenvalues
$(\lambda_1(Y),\lambda_2(Y),\lambda_3(Y))$ are still in the
$C^2$ linearizing region.

We note that in (\ref{eq:LinearLorenz}) $x_1$ corresponds to
the strong-stable direction at $\sigma$, $x_2$ to the
expanding direction and $x_3$ to the weak-stable direction.

Then for some $\delta>0$ we may choose
cross-sections contained in $U_\sigma$
\begin{itemize}
\item $\Sigma^{o,\pm}_\sigma$ at points $y^{\pm}$ in
  different components of
  $W^u_{loc}(\sigma)\setminus\{\sigma\}$
\item $\Sigma^{i,\pm}_\sigma$ at points $x^{\pm}$ in
  different components of $W^s_{loc}(\sigma)\setminus
  W^{ss}_{loc}(\sigma)$
\end{itemize}
and Poincar\'e first hitting time maps
$R^\pm:\Sigma^{i,\pm}_\sigma\setminus\ell^\pm\to
\Sigma^{o,-}_\sigma\cup\Sigma^{o,+}_\sigma$, where $
\ell^\pm=\Sigma^{i,\pm}_\sigma\cap W^s_{loc}(\sigma), $
satisfying (see Figure~\ref{fig:singularbox0})
\begin{enumerate}
\item every orbit in the attractor passing through a small
  neighborhood of the equilibrium $\sigma$ intersects some
  of the incoming cross-sections $\Sigma^{i,\pm}_\sigma$;
\item $R^\pm$ maps each connected component of
  $\Sigma^{i,\pm}_\sigma\setminus\ell^\pm$ diffeomorphically
  inside a different outgoing cross-section
  $\Sigma^{o,\pm}_\sigma$, preserving the corresponding
  stable foliations.
\end{enumerate}
Here we write $W^*_{loc}(\sigma), *=s,ss,u$ for the local
invariant stable, strong-stable and unstable manifolds of
the hyperbolic saddle-type singularity $\sigma$ (see
e.g. \cite{PM82}), so that these invariant manifold extend
up to the cross-sections $\Sigma^{i,\pm}$ and
$\Sigma^{o,\pm}$.

We note that at each flow-box near a singularity there are
four cross-sections: two ``ingoing''
$\Sigma_{\sigma}^{i,\pm}$ and two ``outgoing''
$\Sigma_{\sigma}^{o,\pm}$.

\begin{figure}[ht]
\psfrag{S1}{$\Sigma^{i,+}$}
\psfrag{S2}{$\Sigma^{i,-}$}
\psfrag{S3}{$\Sigma^{o,+}$}
\psfrag{S4}{$\Sigma^{o,-}$}
\psfrag{s}{$\sigma$}
\psfrag{L1}{$\ell^+$}
\psfrag{L2}{$\ell^-$}
\psfrag{z}{{\footnotesize $z$}}
\psfrag{R}{{\footnotesize $R(z)$}}
\psfrag{X1}{{\footnotesize $x_1$}}
\psfrag{X2}{{\footnotesize $x_2$}}
\psfrag{X3}{{\footnotesize $x_3$}}
\includegraphics[height=5cm]{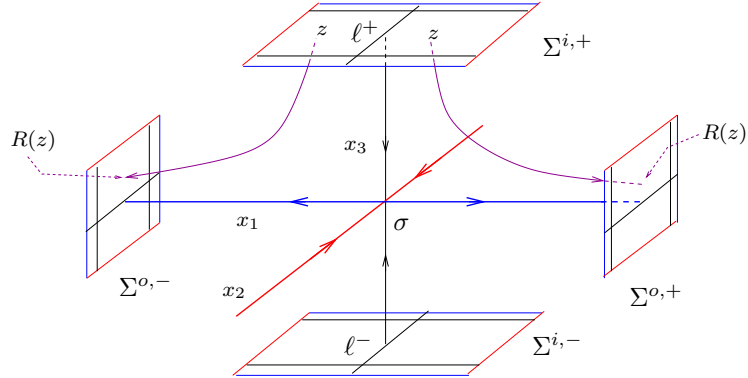}
\caption{\label{fig:singularbox0} Cross-sections near a
  Lorenz-like equilibrium.}
\end{figure}

Using $C^2$ linearizing coordinates in a flow-box near a
singularity, with the appropriate rescaling, we can assume
without loss of generality that, for a small $\delta>0$,  see
Figure~\ref{fig:singularbox0}
\begin{align*}
  \Sigma^{i,\pm}&=\{ (x_1,x_2,\pm1): |x_1|\le\delta ,
  |x_2|\le\delta\} \quad\text{and}
  \\
  \Sigma^{o,\pm}&=\{ (\pm1,x_2,x_3): |x_2|\le\delta ,
  |x_3|\le\delta\}.
\end{align*}
Then from \eqref{eq:LinearLorenz} we can determine the
expression of the Poincar\'e maps between ingoing and
outgoing cross-sections easily
\begin{equation}
  \label{eq:nonflatsing}
\Sigma^{i,+}\cap\{x_1>0\}\to \Sigma^{0,+},
\quad
(x_1,x_2,1)\mapsto
\big(1,x_2\cdot x_1^{-\lambda_2/\lambda_1},
x_1^{-\lambda_3/\lambda_1}\big).
\end{equation}
The cases corresponding to the other ingoing/outgoing pairs
and signs of $x_1,x_2$ are similar.

This shows that the map obtained by identifying points with
the same $x_2$ coordinate, i.e., points in the same stable
leaf, is simply $x\mapsto f(x)= x^{\alpha}$ where
$\alpha=-\lambda_3/\lambda_1\in(0,1)$. Analogously, the
coordinate transverse to the stable leaves transforms
according to the map $g(x,y)=y x^\beta$ where
$\beta=-\lambda_2/\lambda_1>0$.

\begin{remark}\label{rmk:g-bdd-derivative}
  Here $\partial_x g(x,y)=\beta y x^{\beta-1}$ is bounded
  if, and only if, $\beta\ge1$ or, equivalently
  $-\lambda_2>\lambda_1$.
\end{remark}

In these coordinates it is easy to see that for points
$z=(x_1,x_2,\pm1)\in\Sigma^{i,\pm}$ the time $\tau^\pm$
taken by the flow starting at $z$ to reach one of
$\Sigma^{o,\pm}$ depends on $x_1$ only and is given by
\begin{align}\label{eq:tau_sing_int}
  \tau^\pm(x_1)=-\frac{\log|x_1|}{\lambda_1}
  \quad\text{and consequently}\quad
\int_{-\delta}^\delta |\tau^\pm(x_1)|\, dx_1<\infty.
\end{align}
This in particular shows that the return time on a ingoing
cross-section near a singularity is constant on stable leaves.

\subsection{Global Poincar\'e map}
\label{sec:global-poincare-retu}

In this section we exibit a global Poincar\'e map for the
flow near the singular-hyperbolic attractor $\Lambda$.  The
construction we perform here is slightly different from the
one presented at \cite{APPV}: we need injetiveness of the
Poincar\'e return map to prove exact dimensionality of the
physical measure. For this we need to cover the attractor by
flow boxes through pairwise disjoint cross-sections and then
consider a fixed iterate of the Poincar\'e first return map
between these cross-sections.  This is the main difference
with respect to the usual construction presented elsewhere.

We observe first that by Lemma~\ref{l.existeadaptada} we can
take a $\de$-adapted cross-section at each non-singular
point $x\in\Lambda$. We know also that near each singularity
$\sigma_k$ of $\Lambda$ there is a flow-box $U_{\sigma_k}$
containing $\sigma_k$ in its interior. Let $S(\Lambda)$
denote the finite set of equilibria contained in $\Lambda$,
all of which are Lorenz-like.

\begin{description}
\item[Step 1] Choose a flow-box $U_\sigma$ near each
  singularity $\sigma\in S(\Lambda)$ as explained in
  Section~\ref{sec:poincare-maps-near} with the extra
  conditions
  \begin{enumerate}
  \item for any pair of distinct $\sigma_1,\sigma_2\in
    S(\Lambda)$ the flow-boxes
    \begin{align*}
      \Sigma^{i,\pm}_\sigma(T)&:=\{X_s(x):
      x\in\inter(\Sigma_\sigma^{*,\pm}), |s|<T_1\}, \quad
      \sigma=\sigma_1,\sigma_2, \quad *=i,o
    \end{align*}
    are pairwise disjoint, and
  \item the smallest time needed for the positive orbit of a
    point in $\Sigma^{i,\pm}$ to reach $\Sigma^{o,\pm}$ is
    bigger than $T_1$.
  \end{enumerate}
  We note that since we may take $\Sigma_\sigma^{*,\pm}$
  arbitrarily close to $\sigma$, these conditions can always
  be achieved. We denote by $\fS$ the family of all such
  cross-sections near the singularities of $\Lambda$.
\item[Step 2] Consider the open set $V_S=\cup_{\sigma\in
    S(\Lambda)} \cup_{*=i,o} \Sigma_\sigma^{*,\pm}(T_1)$ and
  the compact subset $\Lambda_1:=\Lambda\setminus V_S$ of
  $\Lambda$. For any $x\in \Lambda_1$ we know that $x$ is a
  regular point. Hence we have a $\delta$-adapted
  cross-section $\Sigma_x$ through $x$. We consider the
  $\epsilon_0$-flow-box
  \begin{align*}
    \Sigma_x(\epsilon_0):=\{X_s(x): x\in\inter(\Sigma_x),
    |s|<\epsilon_0\}
  \end{align*}
  for a given fixed $\epsilon_0>0$ small and
  $\epsilon_0<T_1$.  We note $x\in\Lambda_1$ ensures that
  $\Sigma_x(\epsilon_0)$ does not contain any singularity
  and, in fact, does not intersect any of the cross-sections
  fixed at Step 1.

  The collection $\cC:=\{\Sigma_x(\epsilon_0):
  x\in\Lambda_1\}$ is an open cover of the compact set
  $\Lambda_1$. We fix a finite subcover $\cC_0=\{
  \Sigma_{x_1}(\epsilon_0),\dots,\Sigma_{x_k}(\epsilon_0)\}$
  in what follows and also consider the corresponding finite
  family of cross-sections
  $\Xi_0=\{\Sigma_{x_1},\dots,\Sigma_{x_k}\}$.
\item[Step 3] Now we adjust the construction so that the
  Poincar\'e first return time between elements of $\Xi_0$
  is bigger than some uniform positive constant.

  For any given pair $\Sigma,\Sigma^\prime\in\Xi_0$, if we
  have
  $\interior(\Sigma)\cap\interior(\Sigma^\prime)\neq\emptyset$,
  then we may assume without loss of generality that the
  intersection is transversal. For otherwise, if
  $h:\II^2\to\Sigma$ is the coordinate system of $\Sigma$
  given according to Remark~\ref{r.foliated}, we may find a
  $C^{1+\alpha}$ embedding $\tilde h:\II^2\to M$ close
  enough to $h$ so that $\widetilde\Sigma:=\tilde h(\II^2)$
  is a $\delta$-adapted cross-section and there exists
  $\phi:\II^2\to(-\epsilon_0,\epsilon_0)$ such that $\tilde
  h(s,t)=X_{\phi(s,t)}(h(s,t))$ and both pairs
  $\interior(\widetilde \Sigma),\interior(\Sigma^\prime)$
  and $\partial\widetilde \Sigma, \partial\Sigma^\prime$ are
  transversal. In particular, $\partial\widetilde \Sigma$
  and $ \partial\Sigma^\prime$ must be disjoint because the
  ambient space is three-dimensional.
  
  Hence we may assume that the (transversal) intersection
  $\Sigma\cap\Sigma^\prime$ is formed by finitely many
  smooth closed curves. We consider the sub-strip of
  $\Sigma$ given by
  \begin{align*}
    \Sigma_0:=\cup\{ W^s(y,\Sigma) : y\in \Sigma\cap\Sigma^\prime\}
  \end{align*}
  and also the sub-strip of $\Sigma^\prime$ given by
  \begin{align*}
    \Sigma_1:=\cup\{ W^s(z,\Sigma^\prime) : z\in
    \Sigma\cap\Sigma^\prime\};
  \end{align*}
  see Figure~\ref{fig:2cross}.
  According to the definition of
  $W^s(z,\Sigma),W^s(z,\Sigma^\prime)$ there are
  $\phi_z:W^s(z.\Sigma)\to(-\epsilon_0,\epsilon_0)$ and
  $\tilde\phi_z:W^s(z,\Sigma^\prime)\to(-\epsilon_0,\epsilon_0)$
  such that for each
  $z\in\Sigma\cap\Sigma^\prime$
  \begin{align*}
    W^{ss}_0(z):=\{X_{\phi(x)}(x): x\in W^s(z,\Sigma)\} \cup 
    \{X_{\tilde\phi(y)}(y): y\in
    W^s(z,\Sigma^\prime)\}\subset W^{ss}_{\epsilon}(z)
  \end{align*}
  and $\phi(z)=0=\tilde\phi(z)$.  Since, by
  Theorem~\ref{contributions}, the stable manifolds in the
  neighborhood $U_0$ of $\Lambda$ depend
  $C^{2}$-smoothly on the base point
  \begin{align*}
    \Sigma_2:=\{W^{ss}_0(z):z\in\Sigma\cap\Sigma^\prime\}
  \end{align*}
  is a $\delta$-adapted cross-section whose flow-box with
  time $2\epsilon_0$ covers the $\epsilon_0$-flow-box of
  $\Sigma_0$ and $\Sigma_1$.

  We replace $\Sigma$ and $\Sigma ^\prime$ in $\Xi_0$ by the
  following strips: the closure of the connected components
  of $\Sigma\setminus\Sigma_0$; together with the closure of
  the connected components of
  $\Sigma^\prime\setminus\Sigma_1$; and the closure of
  $\Sigma_2$; see Figure~\ref{fig:2cross}. The number of
  such components is finite and, moreover, their flow boxes
  with time $2\epsilon_0$ cover at least the same portion of
  $\Lambda$ as the flow-boxes of $\Sigma$ and
  $\Sigma^\prime$.
\end{description}
This procedure ensures that, given any pair
$\Sigma,\tilde\Sigma$ in $\Xi_0$, their interiors do not
intersect, and the minimum Poincar\'e first return time
between these sections is strictly positive.  At this point
we redefine $\cC_0=\{\Sigma(2\epsilon_0): \Sigma\in\Xi_0\}$.
\begin{figure}[ht]
\psfrag{I}{$\Sigma\cap\Sigma^\prime$}
\psfrag{U}{$\Sigma_1$}\psfrag{Z}{$\Sigma_0$}\psfrag{D}{$\Sigma_2$}
\psfrag{S}{$\Sigma$}\psfrag{L}{$\Sigma^\prime$}
\includegraphics[height=5cm]{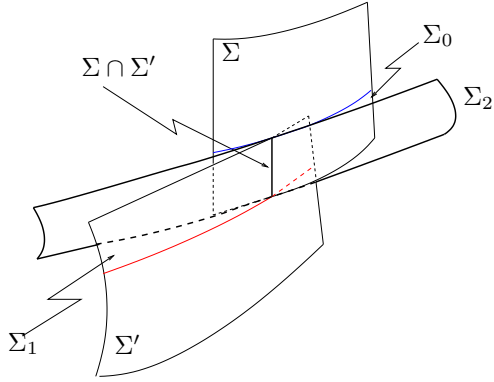}
\caption{\label{fig:2cross} Cross-sections which intersect
  and their adaptation.}
\end{figure}

We define $\Xi:=\Xi_0\cup\fS$ the family of all
cross-sections chosen in the above steps.
\begin{remark}\label{rmk:Xi-family}
  After this construction we note that
  
  \begin{enumerate}
  \item we can ensure that each of the open sets
    $\Sigma(2\epsilon_0), \Sigma\in\Xi$ and $U_\sigma,
    \sigma\in S(\Lambda)$ is contained in the trapping
    region $U_0$, which we also assume is a neighborhood of
    $\Lambda$ where the extension of the strong-stable
    foliation is defined;
  \item given $\Sigma\in\Xi$ and
    $x\in\interior(\Sigma)$, the Poincar\'e first return
    time for the positive orbit of $x$ to reach some
    cross-section in $\Xi$ is strictly positive; since the
    number of cross-sections is finite and each
    cross-section is compact, there exists $\epsilon_1>0$
    such that
    \begin{align*}
      \inf\{t>0: X_t(x)\in\Xi\}\ge \epsilon_1.
    \end{align*}
  \item since $\Lambda$ is an attractor, the
    omega-limit set $\omega(z)$ of any $z\in
    \cup_{\Sigma\in\cC}\Sigma(\epsilon_0)\cup\cup_{\sigma\in
      S(\Lambda)} U_\sigma$ is contained in $\Lambda$. Let
    us assume that $z$ is a regular point.  Thus, with the
    exception of the local stable manifolds of $\sigma$ in
    $U_{\sigma}$, a point $w\in\omega(z)$ has regular orbit
    under the flow which cross some cross-section in $\Xi$
    in some future time. Therefore, the orbit of $z$, which
    accumulates in $\omega(z)$, must cross some
    cross-section of $\Xi$.
  \end{enumerate}
\end{remark}

\begin{definition}\label{def:global-poincare1-map}
  [Global Poincar\'e first return map] For any point $z$ in
  the interior of the cross-sections in $\Xi$, we consider
  the first hit at a cross-section from $\Xi$. This gives
  the \emph{global Poincar\'e first return map}
  \begin{align}\label{eq:global_Poincare_map}
    R_0(z):=X_{\tau_0(z)}(z)
  \end{align}
 and we say that
  \begin{align}\label{eq:global_Poincare_ret_time}
    \tau_0(z):=\inf\{t>0: X_t(z)\in\Xi\}
  \end{align}
  is the \emph{Poincar\'e time} of $z$.  If the point $z$
  never returns to one of the cross-sections, then the map
  $R_0$ is not defined at $z$.
\end{definition}

\begin{remark}
  \label{rmk:injective0}
  This construction ensures that $R_0$ is injective, since
  it is a first return map between cross-sections of a flow.
\end{remark}

In this way we cannot yet ensure that the Poincar\'e time is
big enough to guarantee hyperbolicity of the return map. To
obtain such big enough Poincar\'e time we consider an
iterate of $R_0$, as follows.
Using Proposition~\ref{p.secaohiperbolica} and
Lemma~\ref{stablereturnmap}, for the collection $\Xi$ of
$\delta$-adapted cross-sections, we consider the following
threshold time
\begin{align*}
  T:=\max\{T_1,
  T_2(\Sigma,\Sigma^\prime):\Sigma,\Sigma^\prime\in\Xi,\Sigma\neq\Sigma^\prime\}.
\end{align*}
Hence, if we choose a big enough iterate $R_0^N$ of the
Poincar\'e first return map so that the return time for
$R=R_0^N$ is bigger than $T$, then the tangent map to
$R:\Sigma(\Sigma^\prime)\to\Sigma^\prime$ is hyperbolic
between any pair of cross-sections.
\begin{definition}
  \label{def:global-poincare-map}
  [Global Poincar\'e map] We choose $N\in\ZZ^+$ such that
  $N\epsilon_1>T$ and set $R:=R_0^N$.
\end{definition}
We note that $R$ in the definition above is guaranteed to
have hyperbolic derivative at every point where it is
defined, since we can write
\begin{align}\label{eq:globalPtime}
  R(z)=X_{\tau(z)}(z) \quad\text{with}\quad
  \tau(z)=\sum_{i=0}^{N-1} \tau_0(R_o^i(z))>T
\end{align}
if $R(z)$ is defined, $z\in\Xi$. The function $\tau$ is the
\emph{global Poincar\'e time} and~(\ref{eq:globalPtime})
shows that both Proposition~\ref{p.secaohiperbolica} and
Lemma~\ref{stablereturnmap} simultaneously hold.

In addition, by Lemma~\ref{stablereturnmap}, if $R$ is
defined for $x\in\Sigma$ on some $\Sigma\in\Xi$, then $R$ is
defined for every point in $W^s(x,\Sigma)$. Hence \emph{the
  domain of $R\mid\Sigma$ consists of strips of
  $\Sigma$}. The smoothness of $(t,x)\mapsto X_t(x)$ ensures
that the strips
\begin{equation}
  \label{eq:strip}
  \Sigma(\Sigma') = \{ x\in\Sigma : R(x)\in\Sigma'\}
\end{equation}
have non-empty interior in $\Sigma$ for every
$\Sigma,\Sigma'\in\Xi$.

\begin{remark}
  \label{rmk:R-injective}
  Since $R$ is a fixed iterate of the injective map $R_0$,
  we see that $R$ is also injective. Moreover, by item (3)
  of Remark~\ref{rmk:Xi-family}, the family of all
  $\Sigma(\Sigma^\prime)$ for $\Sigma^\prime\in\Xi$ covers
  $\Sigma$ except the points where $R$ is not defined.
\end{remark}

\subsubsection{Finite number of strips in the domain of the
  global Poincar\'e return map}
\label{sec:finite-number-strips}

The next result shows that, fixing a cross-section
$\Sigma\in\Xi$, the points where $R$ is not defined are
contained in finitely many stable leaves. Thus, after
Remark~\ref{rmk:R-injective}, the family of all possible
strips, defined as in~\eqref{eq:strip} by the set of points
$\Sigma(\Sigma^\prime)$ which move from $\Sigma$ to some
strip $\Sigma^\prime\in\Xi$, covers $\Sigma$ except
for finitely many stable leaves $W^s(x_i,\Sigma), i=1,\dots,
m=m(\Sigma)$.
% Moreover, each strip given by~\eqref{eq:strip} has finitely
% many connected components.
Thus the number of strips in
each cross-section is finite.

We note that $R$ is locally smooth for all points
$x\in\inter(\Sigma)$ such that $R(x)\in\inter(\Xi)$ by the
Tubular Flow Theorem, \cite[Theorem 1.1, pp 40]{PM82}, and
the smoothness of the flow, where $\inter(\Xi)$ is the union
of the interiors of each cross-section of $\Xi$.  Let
$\partial^s\Xi$ denote the union of all the leaves forming
the stable boundary of every cross-section in $\Xi$.

\begin{lemma}\cite[Lemma 6.29, pp 182]{AraPac2010}
  \label{le:descont}
  The subset $\cD$ of points for which $R$ is not defined in
  $\Xi\setminus\partial^s\Xi$ is contained in the set of
  points $x\in\Xi\setminus\partial^s\Xi$ such that:
\begin{enumerate}
\item[(a)] either $R(x)$ is defined and belongs to $\partial^s\Xi$;
\item[(b)] or there is some time $0<t \le t_2$, $t_2$ given at
Lemma \ref{stablereturnmap},  such that $X_t(x)\in
  W^s_{\epsilon}(\sigma)\cap \Sigma_j$ for some singularity $\sigma$ of
  $\Lambda$ and $\Sigma_j \in \Xi$.
\end{enumerate}
Moreover this set is contained in a finite number of stable
leaves of the cross-sections $\Sigma\in\Xi$.
\end{lemma}
The proof of this lemma depends on the fact that the
Poincar\'e time is finite for points where $R(x)$ is
defined, so the other points must be inside the attractor
but never cross other cross-sections of $\Xi$, so they
must be converging to an equilibrium point of $X$ in
$\Lambda$; or they will hit the stable boundary of some
cross-section of $\Xi$.

Let $\Gamma$ be the finite set of stable leaves of $\Xi$
provided by Lemma~\ref{le:descont} together with
$\partial^s\Xi$. Then the complement $\Xi\setminus\Gamma$ of
this set is formed by finitely many open strips $\Sigma\in
\Xi$, where $R$ {\em{is smooth, i.e., of class $C^2$}}.
Each of these strips is then a connected component of the
sets $\Sigma(\Sigma')$ for $\Sigma,\Sigma'\in\Xi$, where $R$
is a $C^2$ diffeomorphism.

\subsubsection{Integrability of the global Poincar\'e return time $\tau$}
\label{sec:integr-return-time}

We may now obtain a crucial property for the construction of
the physical measure for singular hyperbolic attractors and
to study its properties in what follows.
\begin{lemma}
  \label{le:global_poincare_time_int}
  The global Poincar\'e time $\tau$ is integrable with
  respect to the Lebesgue area measure in $\Xi$, induced by
  the Riemannian Lebesgue volume form on the manifold.
\end{lemma}

\begin{proof}
  Indeed, given $z\in\Xi$, the point $R(z)=X_{\tau(z)}(z)$
  is also given by $R_0^N(z)=X_{S_N\tau_0(z)}(z)$ where
  $\tau(z)=S_N\tau_0(z)=\sum_{k=0}^{N-1}\tau_0(R_0^k(z))$. Hence
  $\tau$ is bounded by the sum of at most $N$ exit-time
  functions of flow-boxes of $S(\Lambda)$ (all of them
  integrable with respect to Lebesgue measure) plus the sum
  of at most $N$ bounded Poincar\'e first return time
  functions between cross-sections in $\cC_0$, away from
  singularities.
  Thus the Poincar\'e time of $R$ on $\Xi$ is Lebesgue
  integrable.
\end{proof}

\subsection{The two-dimensional map $F$}
\label{sec:two-dimens-piecew}

>From now on, $\Xi$ is the collection of all strips
$\Sigma(\Sigma^\prime)$ where the Poincar\'e return map is
smooth.  We still denote the strips by the letter $\Sigma$
in what follows.  We choose a $C^2$ $cu$-curve
$\gamma_{\,\Sigma}$ transverse to $\cF^{\,s}_{\Sigma}$ in
each $\Sigma\in\Xi$.  Then the projection $p_\Sigma$ along
leaves of $\cF^{\,s}_{\Sigma}$ onto $\gamma_{\,\Sigma}$ is a
$C^{1+\alpha}$ map, since the stable leaves $W^s(x,\Sigma)$
are defined through every point of $\Sigma\in\Xi$ and
holonomies depend $C^{1+\alpha}$ smoothly on the base point.

Given a set $A$, $\cl(A)$ means the closure of $A$.  We
define
\[
I=\bigcup_{\Sigma,\Sigma'\in\Xi} \cl\big( \Sigma( \Sigma' )
\big)\cap\gamma_{\,\Sigma}\quad\quad \mbox{and}\quad\quad
S=\bigcup_{\Sigma,\Sigma'\in\Xi} \cl\big(\Sigma( \Sigma' )
\big).
\]
As the number of strips is finite, by the properties of
$\Sigma( \Sigma' )$ obtained earlier, the set $I$ is
$C^2$-diffeomorphic to a closed interval $\II=[0,1]$ with
finitely many points $\cC=\{ c_1,\dots, c_n\}$ removed, and
$p_\Sigma|p_\Sigma^{-1}(I)$ becomes a $C^1$ submersion.  The
set $S$ is $C^{1+\alpha}$-diffeomorphic to a
{\em{non-degenerate closed rectangle}} $Q \subset \RR^2,\,\,
Q=[0,1]\times [0,1]$ with finitely many vertical lines
$\cC\times\II=\{c_1,\dots,c_n\}\times\II$ removed. We denote
by $H$ the $C^{1+\alpha}$-diffeomorphism $H: S\to Q$ which
sends stable leaves to vertical lines and consider the
composition map
\[F=H\circ R \circ H^{-1}: Q \to Q .
\]

According to Lemma~\ref{stablereturnmap},
Proposition~\ref{p.secaohiperbolica} and
Corollary~\ref{ccone}, the Poincar\'e map
$R:\Xi\setminus\Gamma\to\Xi$ takes stable leaves of
$\cF^{\,s}_{\Sigma}$ inside stable leaves of the same
foliation and is $C^1$ piecewise hyperbolic.  In addition,
by Corollary \ref{ccone}, a $cu$-curve $\gamma\subset\Sigma$
is taken by $R$ into a $cu$-curve $R(\gamma)$ in the image
cross-section.

We can define unstable cones on $Q$ using the smoothness of
$H$ as $\CC^u_{\rho}(H(x)):=DH(x)\cdot
C^u_\rho(x)$. This ensures, in particular, that $cu$-curves
are taken by $F=H \circ R \circ H^{-1}$ into $cu$-curves.
Hence, the map $F=H \circ R \circ H^{-1}: Q \to Q$ can be
written as
$$
F(x,y)=(T(x),G(x,y)),
$$ where\,\,
\[
T:\II\setminus\cC\to \II, \,\,
%\,\mbox{given by}\,
(\II\setminus\cC)\ni z \mapsto H\Big(p_{\Sigma'}\Big( R
\big( W^s(H^{-1}(z),\Sigma) \cap \Sigma(\Sigma') \big) \Big)
\Big).
\]
Moreover, by construction, we have that the following hold:
\begin{enumerate}
\item[(a)] $T:\II\setminus\cC \to \II$ is not defined at a
  finite number of points $c_1, \cdots, c_n$, and it is
  $C^1$ at $\II\setminus\cC=\cup_{0\leq j \leq n} I_j$.  The
  points $c_1, \cdots, c_n$ correspond either to the
  projection of a line $\ell= \Sigma_i\cap
  W^s_\epsilon(\sigma)$ of points which fall in the stable
  manifold of an equilibrium $\sigma$,
  % $\Sigma^{i,+}_k$ a ingoing cross-section at
  % $W^s(\sigma_k)$,
  or to the projection of the boundary of a strip
  $\Sigma\in\Xi$.
%In the first case, the expression of the restriction of $f$ to $[c_i,c_{i+1}]$ is like in Property (P2) and in
%the second case the expression of $f$ is like in Property (P1) described Section \ref{sujeira}.
\item[(b)] $G : Q \to \II$ is not defined at a finite
  number of vertical lines $\ell_{c_i}=\{c_i\}\times\II$ in
  $Q$, corresponding to $p_\Sigma^{-1}(c_i)$, where $c_i$
  are as in (a) and $G$ is $C^1$ restricted to $Q\setminus
  (\cup_{1\leq i\leq n} \ell_{c_i})$.
  % According to possibilities (1) and (2) to $c_i$
  % described in (a), the expression of the restriction of
  % $g$ to a strip $\Sigma_i=[c_i,c_{i+1}] \times [0,1]$ is
  % like in Property (P2) or Property (P1) described in
  % Section \ref{sujeira}.
\item[(c)] the choice of $R$ as an iterate of the first
  return map $R_0$ between the finite family $\Xi$ of
  cross-sections ensures that $F$ is injective in the
  following sense: if $i\neq j$ then $F((c_i,c_{i+1})\times
  [0,1])\cap F((c_j,c_{j+1})\times [0,1])=\emptyset$.
\end{enumerate}
Finally, with the convention $c_0=0 $ and $c_{n}=1$, the
restriction of $F$ to each strip $\Sigma_j =(c_{j},c_{j+1})
\times [0,1] \subset Q$, $0\leq j\leq n -1$, is given by
$$F_{\Sigma_j}(x,y)=\big(T_{\Sigma_j}(x),G_{\Sigma_j}(x,y)\big).$$
%where $\Sigma_j=[c_{j-1},c_j] \times [0,1]$.
%, and $f_{\Sigma_j}$ and $g_{\Sigma_j}$ satisfy either
%Property (P1) or Property (P2), according to the type of the discontinuity ${c_j}$,
%described in Lemma \ref{le:descont} and Section \ref{sujeira}.

% When no confusion is possible, we refer to
% $F_{\Sigma_j}(\cdot ,\cdot )$ as $ F(\cdot ,\cdot
% )=(f(\cdot),g(\cdot,\cdot)).$

We note that Lemma \ref{stablereturnmap}(b) together with
the fact that $H$ is a $C^{1+\alpha}$-diffeomorphism imply
\begin{lemma}\label{gLipschitz}
  The map $F:Q \to Q$ preserves the vertical foliation
  $\mathcal{F}^s$ of $Q$, and $F|\gamma$ is
  $\lambda$-Lipschitz with $\lambda < 1$ on each leaf
  $\gamma \in \mathcal{F}^s$.
\end{lemma}

\begin{remark}\label{rmk:horiz-cu-curve}
  By taking the cross-sections $\Sigma$ small enough we can
  ensure that the unstable cone has very small variation
  along the strip. We can then ensure that all curves which
  are at a constant distance from $\gamma_\Sigma$ along
  $\cF^s_\Sigma$ are also $cu$-curves. Since $\cF^s_\Sigma$
  is sent to vertical lines in $Q$ through $H$, this in
  turn ensures that all horizontal lines in $Q$ which do
  not intersect the vertical lines $\ell_{c_i}$ are
  $cu$-curves, because these horizontal lines correspond
  through $H$ to curves which are essentially at a constant
  distance from $\gamma_\Sigma$ along the stable leaves.
\end{remark}

\subsubsection{Additional properties of the one-dimensional
  map $T$}\label{sec:f}

As already mentioned, since the flow $(X_t)_{t\in\RR}$ is $C^2$,
%it is well known
%\cite{Man87,PT93} that the stable leaf $W^s(x,\Sigma)$  is a $C^2$
% embedded disk for every $x\in \Sigma\in\Xi$ and these
the leaves $W^s(x,\Sigma)$, $x\in \Sigma$, define a $C^{1}$
foliation $\cF^{\,s}_{\Sigma}$ of each $\Sigma\in\Xi$ with a
H\"older-$C^1$ holonomy (since the leaves are
one-dimensional).

These properties taken together with the expansion provided
by Corollary~\ref{ccone} imply (see the proof in
\cite[Sec. 7.3.2, pp 222]{AraPac2010}).

\begin{lemma}\label{fholder}
  The one-dimensional map $T$ obtained above is in fact a
  $C^{1+\alpha}$ piecewise expanding map such that $1/|DT|$
  is $\alpha$-H\"older, for some $0< \alpha < 1$, restricted
  to each $I_j$.
  \end{lemma}

  The uniform expansion is a consequence of the existence of
  a uniform bound for the angles between $\cF^s_\Sigma$ and
  the curves $\gamma_\Sigma$, once we have fixed the set
  $\Xi$ of cross-sections, and our ability to obtain an
  arbitrarily large expansion rate along the unstable cones
  if we choose the threshold $T_1>0$ large enough.

  As seen in Section~\ref{sec:p-bounded-variation} (see
  Theorem \ref{th-keller} and consequences), if $T$ is
  piecewise expanding and $h=1/|DT|$ has finite universal
  $p$-bounded variation then there is an absolutely
  continuous invariant measure with $p$-bounded variation
  density.  It is easy to see that if $1/|DT|$ is piecewise
  $\alpha$-H\"older for some $\alpha\in(0,1)$, then it is of
  universal $p$-bounded variation.  Moreover, by taking an
  iterate $T^k$ of $T$ if necessary, we can assume that each
  ergodic absolutely continuous invariant probability
  measure for $T$ is decomposed into a finite family of of
  probability measures invariant for $T^k$ and having
  exponential speed of convergence to equilibrium.

  Thus, Lemma~\ref{fholder} together with the results from
  Section \ref{sec:p-bounded-variation} imply the following
  result.

\begin{lemma}\label{ftemacim}
  The one-dimensional map $T$ obtained above has finitely
  many ergodic physical measures $\mu_T^1,\dots,\mu_T^l$,
  whose density is a function of $p$-bounded variation, and
  whose ergodic basins cover Lebesgue almost all points of
  $\II$.
\end{lemma}

We recall that, given a $\vfi$-invariant Borel probability
measure $\mu$ with respect to a map $\vfi:X\circlearrowleft$
on a metric space $X$, we denote by
\begin{align*}
  B(\mu)=\{x\in X:
  \lim_{n\to+\infty}\frac1n\sum_{j=0}^{n-1}\psi(\vfi^j(x))
  =\int\phi\,d\mu,\quad\forall \psi\in C^0(X,\RR)\}
\end{align*}
the \emph{ergodic basin} of $\mu$ and say that $\mu$ is
\emph{physical} if the volume of $B(\mu)$ (or some other
natural measure) is positive.

According to standard constructions described in \cite{APPV}
and \cite[Section 7.3, pp. 225-235]{AraPac2010}, each
absolutely continuous ergodic probability measure $\mu_T^i$
for $T$ can be lifted to a unique physical ergodic
probability measure $\mu_F^i$ for the map $F$, in such a way
that $\pi_*\mu^i_F=\mu^i_T$, where $\pi:Q\to\II$ is the
projection on the first coordinate and we have $\pi\circ
F=T\circ\pi$. Hence the ergodic basin $B(\mu_F^i)$ of
$\mu_F^i$ is given by $\pi^{-1}(\supp\mu^i_T)$ and is a
finite collection of strips with non-empty interior and the
interior of the supports of distinct $\mu^j_F$ and $\mu^i_F$
are disjoint.

Moreover, each probability
measure $\mu_F^i$ can be lifted to a physical ergodic
probability measure $\nu^i_\Lambda$ for the flow of $X$
supported in the attractor $\Lambda$; more on this in
Subsection~\ref{sec:integr-global-poinca}.  Since a
singular-hyperbolic attractor is transitive, that is, it has
a dense orbit, it follows that there can be only one such
physical measure for the flow in the basin of attraction of
$\Lambda$; see~\cite[Section 7.3.8,
pp. 234-235]{AraPac2010}.

\begin{lemma}\label{funicamedida}
  For each absolutely continuous ergodic probability measure
  $\mu_T^i$ for $T$ there exists a unique physical ergodic
  probability measure $\mu_F^i$ for the map $F$ such that
  $\pi_*\mu^i_F=\mu^i_T$.  

  For each physical measure $\mu^i_F$ for $F$ there exists a
  unique ergodic physical measure $\nu_X$ for
  $X$.
\end{lemma}

We note that since $F$ is given by a power of the first
return map $R_0$, the uniqueness statement above does not
imply uniqueness of the absolutely continuous invariant
probability measure $\mu_T$ for $T$.

\subsubsection{Positive entropy for the two-dimensional map}
\label{sec:positive-entropy-two}

We will also need the fact that the entropy $h_{\mu^i_F}(F)$
of the map $F$ with respect to each physical measure
$\mu^i_F$ is positive. Since we know that $\pi\circ
F=T\circ\pi$ and 
$h_{\mu^i_T}(T)=\int\log|DT|\,d\mu_T>0$, where $\mu^i_T$ is
one ergodic absolutely continuous $T$-invariant probability
measure, we see that $h_{\mu^i_F}(F)>0$, for each
$F$-invariant ergodic and physical probability measure $\mu^i_F$.

%%%%%%%%%%%%%%%%%%%%%%%%%%%%%%%%%%%%%%%%%%%%%%%%%%%%%%%%%%%%

\subsubsection{Additional properties of the map $G$}\label{Sec:propg}

Since $ F(x,y)=(T(x),G(x,y)), $ Lemma \ref{gLipschitz} means
that there is $0< \lambda < 1$ such that for all fixed $x$,
\begin{equation}\label{gcontracaofibra}
 \dist(G(x,y_1),G(x,y_2))
 \leq
 \lambda \cdot |y_1-y_2|\,, \,\, \forall \,\, y_1, y_2 \in \II.
\end{equation}

The form of the singularities ensures the following result.

\begin{lemma}
  \label{le:domination-derivative}
  If for all singularities $\sigma$ we have that the
  eigenvalues
  $\lambda_2(\sigma)<\lambda_3(\sigma)<0<\lambda_1(\sigma)$
  satisfy the non-resonance conditions expressed in
  Theorem~\ref{thm:smooth-linear},
  % \begin{align*}
  %   \lambda_1(\sigma)+\lambda_3(\sigma)>0
  %   \quad\text{and}\quad
  %   \lambda_2(\sigma)+\lambda_1(\sigma)<0,
  % \end{align*}
  % then there exists $\kappa>0$ such that, if $x_1$ and $x_2$
  % satisfy $c_i \notin [x_1,x_2]$ for $i=1,\dots, n$, then
  % for every $y \in [0,1]$ one has that
% \footnote{The curve
%     $t\mapsto (t,y)$ with $t\in[x_1,x_2]$ is a $cu$-curve
%     and its image $F(t,y)$ is also a $cu$-curve (this is a
%     consequence of domination). Hence by the Generalized
%     Mean Value Theorem there exists $\xi\in(x_1,x_2)$ such
%     that
%     \begin{align*}
%       \frac{|g(x_2,y)-g(x_1,y)|}{|f(x_2)-f(x_1)|}
%       =
%       \frac{|\partial_x g(\xi,y)|}{|f^\prime(\xi)|}
%       \le
%       \rho_0
%     \end{align*}
%     where $\rho_0>0$ is such that $\CC_{\rho}^u(x,y)\subset
%     \{(u,v)\in\RR^2: |v|<\rho_0|u|\}$ for all $(x,y)\in\QQ$.
%     In fact, for $\xi$ close to a singular line the quotient
%     can be made arbitrarily close to zero. Isn't this enough?
%   }
% \[
% | g(x_1,y)-g(x_2,y)| < \kappa \cdot |x_1 - x_2|.
% \]
 then the map $G:Q\to\II$
  % \begin{align*}
  %   g_y: [x_1,x_2]\to\RR,\quad t\mapsto g(t,y)
  % \end{align*}
satisfies $\varsq(G)<\infty$.
\end{lemma}

\begin{proof}
%   We consider the map $t\mapsto g(t,y)$ with $t\in[x_1,x_2]$
%   and $x_1, x_2$ as in the statement.  The Mean Value
%   Theorem ensures that
% $$
% |g(x_1,y) - g(x_2,y)| \leq |\partial_x g(x,y)| \cdot |x_2 -
% x_1|
% $$
% for some $x\in(x_1,x_2)$. We just have to show that the
% above partial derivative is bounded by a uniform constant
% $\kappa>0$.
  From the choices made in
  Section~\ref{sec:poincare-maps-near} we see that the
  expression of the Poincar\'e maps between ingoing and
  outgoing cross-sections implies the following.

  For a maximal interval $(c_i,c_{i+1})$, where the first
  coordinate map $T$ is monotonous, corresponding to an
  ingoing cross-section near a equilibrium point of the flow $\sigma_i$ with
  the eigenvalue ratios $\alpha_i>0$ and $\beta_i\in(0,1)$,
  we have that $ F\mid(c_i,c_{i+1})\times\II$ is given by
\begin{align}\label{eq:expr-g-near-sing}
 (x,y)\mapsto((x-c_i)^\beta, y(x-c_i)^{\alpha_i})
 \quad\text{or}\quad
 (x,y)\mapsto(|x-c_{i+1}|^\beta, y|x-c_{i+1}|^{\alpha_i})
\end{align}
For the remaining cases, $F\mid(c_i,c_{i+1})\times\II$ is
just a Poincar\'e map between tubular neighborhoods of regular
points for the flow, whose derivatives are bounded:
$T^\prime$, $\partial_xG$ and $\partial_yG$ are bounded
functions. Since there are finitely many such tubular
neighborhoods, we let $K>0$ be an upper bound for these
derivatives.

Let us now consider the estimation of the total
variation. Let $0=x_1<x_2<\dots,<x_n=1$ be a partition of
$\II$ and $y_1,y_2,\dots,y_n$ arbitrary points in $\II$. For
a sequence $x_k<x_{k+1}<\dots<x_l$ in $(c_i,c_{i+1})$ we
consider two cases.
\begin{description}
\item[Case A]
  $(c_i,c_{i+1})\times\II$ is the domain corresponding to an
  ingoing cross-section near an equilibrium point of the flow.  We can write,
  since $(x_j,y_j)\in Q$
  \begin{align*}
    \sum_{k\le j\le l}\big| G(x_j,y_j) - G(x_{j+1},y_j)\big|
    &=
    \sum_{k\le j\le l}
    y_j\big((x_{j+1}-c_i)^{\alpha_i}-(x_j-c_i)^{\alpha_i}\big)
    \\
    &\le
    \sum_{k\le j\le l}
    \big((x_{j+1}-c_i)^{\alpha_i}-(x_j-c_i)^{\alpha_i}\big)
    \\
    &=(x_{l}-c_i)^{\alpha_i}-(x_k-c_i)^{\alpha_i} \le1
  \end{align*}
  where we have assumed that
  $G(x,y)=y(x-c_i)^{\alpha_i}$. The other case, with
  $c_{i+1}$ in the place of $c_i$, is similar.
\item[Case B]   $(c_i,c_{i+1})\times\II$ is the domain
  corresponding to a tubular neighborhood away from a
  singular flow-box. We can now write by the Mean Value Theorem
  \begin{align*}
    \sum_{k\le j\le l}\big| G(x_j,y_j) - G(x_{j+1},y_j)\big|
    =
    \sum_{k\le j\le l}\big| \partial_x G(x_j^*,y_j) \big|
    (x_{j+1}-x_j)
    \le
    K(x_l-x_k)
    \le K
  \end{align*}
  for some $x_j^*\in(x_j,x_{j+1})$.
\end{description}
Finally, the case $x_k<c_i<x_{k+1}$ can also be bounded
\begin{align*}
  \big| G(x_k,y_k) - G(x_{k+1},y_k)\big|
  \le
  2\sup_{x\in\II}|g(x,y_k)|\le2.
\end{align*}
We note that the bounds we have obtained do not depend on
the choice of the $y_j$.
Since there are finitely many such tubular neighborhoods and
flow-boxes to consider, we see that
\begin{align*}
  \sup_{n\in\ZZ^+}
  \sup_{\substack{y_1,\dots,y_n\in\II\\0=x_1<\dots<x_n=1}}
  \sum_{j=1}^n \big| G(x_j,y_j) - G(x_{j+1},y_j)\big|
\end{align*}
is bounded above by a constant times the number of
smooth branches of the first coordinate function $f$.
\end{proof}

Now we state a straightforward consequence of
Remark~\ref{rmk:g-bdd-derivative}.

\begin{lemma}
  \label{le:F-g-bdd-derivative}
  If for all singularities $\sigma\in\Lambda$ we have the
  eigenvalue relation
  $-\lambda_2(\sigma)>\lambda_1(\sigma)$, then the second
  coordinate map $G$ of $F$ has a bounded partial derivative
  with respect to the first coordinate, i.e., there exists
  $C>0$ such that $|\partial_x G(x,y)|<C$ for all
  $(x,y)\in(\II\setminus\{c_1,\dots, c_n\})\times\II$.
\end{lemma}

\begin{proof}
  We just have to observe that in the domains corresponding
  to ingoing cross-sections near singularities, the
  expression of the map $g$ is given
  by~(\ref{eq:expr-g-near-sing}).  Thus we can apply
  Remark~\ref{rmk:g-bdd-derivative} to conclude that
  $\partial_x g$ is bounded if, and only if, the stated
  eigenvalue relation holds for the eigenvalues of $DX$ at
  the equilibrium point. For the other domains, since $F$ is just
  a Poincar\'e map between tubular neighborhoods of regular
  points for the flow, its partial derivatives are bounded.
\end{proof}

\subsection{Integrability of
  $\tau$, $\log|T^\prime|$ and $\log|\partial_yG|$ with
  respect to the physical measures}
\label{sec:integr-global-poinca}

Now we show that the global Poincar\'e time $\tau$ is
integrable with respect to the $F$-invariant physical
measures $\mu^i_F$ on $Q$, which lifts to the physical
measure $\nu$ for the flow on the singular-hyperbolic
attractor and itself is a lift of the $T$-invariant
absolutely continuous probability measure $\mu_T$ on $\II$.

\begin{proposition}
  \label{pr:tau_muF-int}
  The global Poincar\'e time $\tau$ is integrable with
  respect to each $F$-invariant physical probability measure
  $\mu^i_F$.
\end{proposition}

\begin{proof}
  We recall the main steps of the construction of $\mu_F$
  from $\mu_T$.

  We denote by $\cF=\{\{x\}\times\II\}_{x\in\II}$ the
  vertical foliation on the square $Q$ and by
  $\Gamma=\{\{c_i\}\times\II\}_{i=1,\dots,n}$ the leaves
  corresponding to discontinuities of the map $T$. We have
  already shown that $\cF$ is
  \begin{itemize}
  \item{{\em invariant:\/}} the image of any $\xi\in\cF$
    distinct from $\Gamma$ is contained in some element
    $\eta$ of $\cF$;
  \item{{\em contracting:\/}} the diameter of $F^n(\xi)$
    goes to zero when $n\to\infty$, uniformly over all the
    $\xi\in\cF$ for which $F^n(\xi)$ is defined.
  \end{itemize}
  Let $\mu^i_T$ be an absolutely continuous probability
  measure on $\II=Q/\cF$ invariant under the transformation
  $T$ obtained in Subsection~\ref{sec:f}.  For each bounded
  function $\psi:Q\to\RR$, let $\psi_{-}:\cF\to\RR$ and
  $\psi_{+}: \cF\to\RR$ be defined by
$$
\psi_{-}(\xi)=\inf_{x\in\xi}\psi(x) \qquad\mbox{and}\qquad
\psi_{+}(\xi)=\sup_{x\in\xi}\psi(x).
$$
\begin{lemma}
  \label{l.fconvergence} {\cite[Lemma 7.21]{AraPac2010}}
  Given any continuous function $\psi:Q\to\RR$, both
  limits
  \begin{equation}
    \label{eq:bothlim}
    \lim_n \int (\psi\circ F^n)_{-}\,d\mu^i_T
    \quad\text{and}\quad
    \lim_n \int (\psi\circ F^n)_{+}\,d\mu^i_T
  \end{equation}
  exist, and they coincide.
\end{lemma}

% \begin{remark}\label{rmk:int_inbetween}
%   We note that $\int\psi\,d\mu_F = \lim_n \int (\psi\circ
%   F^n)_{\#}\,d\mu_f$ for every function $\psi:\Xi\to\RR$
%   which is % $\mu_f$-almost everywhere
%   continuous and any choice of a sequence $(\psi\circ
%   F^n)_{\#}:\cF\to\RR$, with
% $$
% \inf (\psi\mid F^n(\xi)) \le (\psi\circ F^n)_{\#}(\xi) \le
% \sup (\psi\mid F^n(\xi)).
% $$
% Moreover we can define $\int\psi\,d\mu_F$ for any measurable
% $\psi:\Xi\to\RR$ such that
% \[
% \lim_{n\to+\infty} \big( \sup (\psi\mid F^n(\xi)) - \inf
% (\psi\mid F^n(\xi))\big)=0
% \]
% uniformly in $n\in\NN$ and in $\xi\in\cF$. This will be
% useful in what follows.
% \end{remark}

>From this it is straighforward to prove the following.

\begin{corollary}\cite[Corollaries 7.22, 7.25 \& Subsection
  7.3.5]{AraPac2010}
  \label{c.poincareinvariant}
  There exists a unique probability measure $\mu^i_F$ on
  $Q$ such that
$$
\int \psi\,d\mu^i_F =\lim \int (\psi\circ F^n)_{-}\,d\mu^i_T
=\lim \int (\psi\circ F^n)_{+}\,d\mu^i_T
$$
for every continuous function $\psi:\Xi\to\RR$.  Besides,
$\mu^i_F$ is invariant under $F$ and is $F$-ergodic. In
addition, the basin $B(\mu^i_F)$ of $\mu_F$ equals $Q$
except for a zero Lebesgue (area) measure subset.
\end{corollary}

Now, for the integrability of $\tau$ with respect to
$\mu^i_F$, we use the construction in
Lemma~\ref{l.fconvergence}.  We observe first that for each
$(x,y)\in Q\setminus\Gamma$ we have a constant $c>0$ such
that
\begin{align*}
  \big(\tau\circ F^n(x,y)\big)_+-\big(\tau\circ
  (F^n(x,y)\big)_- < c\cdot\diam(F^n(\{x\}\times\II))
\end{align*}
since, as explained in
Subsection~\ref{sec:global-poincare-retu}, the Poincar\'e
time on points of a ingoing cross-section near a equilibrium
point is constant on the stable leaves; and on other
cross-sections is just a uniformly bounded smooth function
of $y$. Moreover, because the density $d\mu^i_T/dm$ is bounded
from above (since it is a function of generalized bounded
variation on the interval; see
Section~\ref{sec:p-bounded-variation}), by a constant $L>0$
say, we have
\begin{align*}
  \int\tau_+ \, d\mu^i_T \le L\int \tau_+\,dm <\infty
\end{align*}
since, as explained in~\eqref{eq:tau_sing_int}, the integral
with respect to Lebesgue (length) measure over an ingoing
cross-section is finite, and the return time function is
globally bounded otherwise.

Hence we may define $\tau_\II(x):=\tau(x,y)_+$ and write
\begin{align*}
  \tau_\II(T^n(x))= (\tau(F^n(x,y)))_+ \ge
  \tau(F^n(x,y))-c\cdot\diam(F^n(\{x\}\times\II)) \ge
  \tau(F^n(x,y))-c\lambda^n
\end{align*}
for all $y\in\II$, $n\ge1$ and $\mu_T$-almost every
$x\in\II$.  Thus we obtain
\begin{align*}
  \int (\tau\circ F^n)_- \,d\mu^i_T &\le \int (\tau_\II\circ
  T^n) \,d\mu^i_T + c\lambda^n = \int \tau_\II\,d\mu^i_T +
  c\lambda^n
\end{align*}
and from Lemma~\ref{l.fconvergence} we see that
$\int\tau\,d\mu^i_F\le \int \tau_\II\,d\mu^i_T<\infty$. This
completes the proof of the proposition.
\end{proof}

Some other integrability or regularity properties will be needed in the
sequel. We obtain them here.

\begin{proposition}
  \label{pr:logTprimeDyG-int}
  We have the following properties:
  \begin{enumerate}
  \item $0<\int \log|T^\prime|\,d\mu_F <\infty$;
  \item $\int -\log|\partial_yG(x,y)|\,d\mu_F <\infty$;
  \item the  maps $y\mapsto \partial_y G(x,y)$ are
    uniformly equicontinuous for
    $x\in\II\setminus\{c_1,\dots,c_n\}$, i.e., outside the
    singularities of the map $T$.
  \end{enumerate}
\end{proposition}

\begin{proof}
  For the integrability, we repeat the arguments in the
  proof of Proposition~\ref{pr:tau_muF-int}.

  Indeed, at a maximal interval $(c_i,c_{i+1})$, where the
  first coordinate map $T$ is monotonous, corresponding to
  an ingoing cross-section near a equilibrium point $\sigma_i$
  with the eigenvalue ratios $\alpha_i>0$ and
  $\beta_i\in(0,1)$, we have the
  expression~\eqref{eq:expr-g-near-sing}, thus we obtain
  $\log|T^\prime(x)|=\log\beta_i+(\beta_i-1)\log|x-c_l|$ and
  $-\log|\partial_y G(x,y)|=(1-\alpha_i)\log|x-c_l|$ with
  $l=i$ or $l=i+1$. 

  For the remaining cases, $F\mid(c_i,c_{i+1})\times\II$ is
  just a Poincar\'e map between tubular neighborhoods of
  regular points for the flow bounded away from equilibria,
  whose derivatives are bounded: $T^\prime$ and
  $\partial_yG$ are bounded functions. In addition, because
  $\det DF=T^\prime(x)\cdot\partial_yG(x,y)\neq0$ we have
  that $|\partial_y G(x,y)|$ is also bounded away from zero.
  Since there are finitely many such tubular neighborhoods,
  we let $K>0$ be an upper bound for these derivatives.

  We can now argue exactly as in the proof of
  Proposition~\ref{pr:tau_muF-int} to conclude that the
  integrals in items 1 and 2 of the statement are finite,
  since both $\log|T^\prime(x)|$ and $-\log|\partial_y
  G(x,y)|$ are comparable to the logarithm of the distance
  to the singular set $\{c_1,\dots,c_n\}$, and this function
  in integrable, as explained in
  \eqref{eq:tau_sing_int}. Moreover, since $|T^\prime|>1$ in
  $\II\setminus\{c_1,\dots,c_n\}$ we also obtain that the
  integral in item 1 is positive.

  Finally, on the one hand, we have seen that $|\partial_y
  G(x,y)|$ does not depend on $y$ at monotonicity intervals
  associated to ingoing cross-sections near
  singularities. On the other hand, on other monotonicity
  intervals, the function $|\partial_y G(x,y)|$ is bounded
  above and also away from zero, and it is moreover a $C^1$
  function with bounded derivatives. This is enough to
  conclude item 3 of the statement.  The proof is complete.
\end{proof}

%%%%%%%%%%%%%%%%%%%%%%%%%%%%%%%%%%%%%%%%%%%%%%%%%%%%%%%%%%%%%%%%%%%

%%%%%%%%%%%%%%%%%%%%%%%%%%%%%%%%%%%%%%%%%%%%%%%%%%%%%%%%%%%%

\section{Decay of correlation for the Poincar\'{e} maps.}
\label{sec:decay-correl-poincar}

Now we are ready to consider the Poincar\'{e} maps of
singular hyperbolic attractors, and use the results of the
previous section to deduce exponential decay of correlations
for these maps for each one of the physical measures. To apply Theorem \ref{resdue} we have to
consider a suitable seminorm.  We will use $var^{\square}$
as defined at beginning of Section
\ref{sec:p-bounded-variation} (recall also Theorem
\ref{thm:propert-singhyp-attractor}, item 4). The following
Lemma will establish an estimation (see Equation
\ref{square} ) which will allow to apply Theorem
\ref{resdue}.

\begin{lemma}
\label{so_hard}If $F$ has the following properties

\begin{itemize}
\item  $F:Q \rightarrow Q $ is of the form $%
F(x,y)=(T(x),G(x,y))$ and

\item the map uniformly contracts the vertical leaves

\item $var^{\square }(G)<\infty $

\item $T:I\rightarrow I$ has $m+1$, increasing branches on the
intervals $[-\frac{1}{2}=c_{0},c_{1})$,...,$(c_{i},c_{i+1})$ ,..., $(c_{m},%
\frac{1}{2}=c_{m+1}]$,
\end{itemize}

If moreover $p\geq 1$, then there are $C,K\in \mathbb{R}$ such that
\begin{equation*}
||\pi (f\circ F^{n})||_{1,\frac{1}{p}}+var^{\square }(f\circ F^{n})\leq
CK^{n}(||\pi (f)||_{1,\frac{1}{p}}+||f||_{\updownarrow lip}+var^{\square
}(f)).
\end{equation*}
for each $n\geq 0$.
\end{lemma}

Before proving the above Lemma, we need the following
\begin{lemma}\label{4}
$||\pi (f)||_{1,1}\leq 2var^{\square }(f)$
\end{lemma}

\begin{proof}
Let us fix $y$ and set $y_{1}=y,...,y_{n}=y$, then
\begin{equation*}
var^{\square }(f,x_{1},...,x_{n},y_{1},...,y_{n})=\sum_{i\leq
n}|f(x_{i},y)-f(x_{i+1},y)|\leq var^{\square }(f)
\end{equation*}
and
\begin{equation*}
\int \Big(\sum_{i\leq n}|f(x_{i},y)-f(x_{i+1},y)|\Big)\,dy\leq var^{\square }(f).
\end{equation*}
Since
\begin{equation*}
\sum_{i\leq n}\left|\int f(x_{i},y)dy-\int
  f(x_{i+1},y)\,dy\right|
\leq 
\int\Big(\sum_{i\leq n}|f(x_{i},y)-f(x_{i+1},y)|\Big)\,dy
\end{equation*}
then
\begin{equation*}
\sum_{i\leq n}|\pi (f)(x_{i})-\pi (f)(x_{i+1})|\leq var^{\square }(f).
\end{equation*}

For any $x_{1},...x_{n}$. By proposition \ref{cmp} Equation \ref{4.6} we
then have the required result.
\end{proof}

\begin{proof}[Proof of Lemma \ref{so_hard}] We first remark
  that
\begin{equation*}
||\pi (f\circ F^{n})||_{1,\frac{1}{p}}\leq ||\pi (f\circ F^{n})||_{1,1}\leq
2var^{\square }(f\circ F^{n})
\end{equation*}%
so it is sufficient to prove that
\begin{equation*}
3var^{\square }(f\circ F^{n})\leq CK^{n}(||f||_{\updownarrow
lip}+var^{\square }(f)).
\end{equation*}

Let us fix $x_{1}\leq ...\leq x_{k}$,%
\begin{eqnarray*}
var^{\square }(f\circ F,x_{1},...,x_{k},y_{1},...,y_{k}) &=&\sum_{i\leq
k}|f(F(x_{i},y_{i}))-f(F(x_{i+1},y_{i}))| \\
&=&\sum_{i\leq k}|f(T(x_{i}),G(x_{i},y_{i}))-f(T(x_{i+1}),G(x_{i+1},y_{i}))|.
\end{eqnarray*}

Suppose that there are no $\{c_{i}\}$ between $x_{i}$ and
$x_{i+1}$. Since $f$ is Lipchitz along the $y$ direction,
by the third item in the assumptions
\begin{gather*}
|f(T(x_{i}),G(x_{i},y_{i}))-f(T(x_{i+1}),G(x_{i+1},y_{i}))\leq \\
\leq
|f(T(x_{i}),G(x_{i},y_{i})))-f(T(x_{i+1}),G(x_{i},y_{i}))|+||f||_{%
\updownarrow lip}|G(x_{i},y_{i})-G(x_{i+1},y_{i})|.
\end{gather*}
We note that
$\sum_{i}|G(x_{i},y_{i})-G(x_{i+1},y_{i})|\leq var^{\square
}(G). $

Now suppose that $x_{k_{j}},...,x_{k_{j+1}-1}\in
(c_{j},c_{j+1})$. Since $T|_{(c_{i},c_{i+1})}$ is increasing, then
\begin{equation*}
\sum_{k_{j}\leq i\leq
k_{j+1}-2}|f(T(x_{i}),G(x_{i},y_{i}))-f(T(x_{i+1}),G(x_{i},y_{i}))|\leq
var^{\square }(f)
\end{equation*}
and 
\begin{align*}
  |f(T(x_{k_{j+1}-1}),G(x_{k_{j+1}-1},y_{k_{j+1}-1}))-
  f(T(x_{k_{j+1}}),G(x_{k_{j+1}-1},y_{k_{j+1}-1}))| \leq
  2||f||_{\infty !}.
\end{align*}
 Hence, putting all togheter one obtains
\begin{equation*}
\sum_{i\leq
k}|f(T(x_{i}),G(x_{i},y_{i}))-f(T(x_{i+1}),G(x_{i+1},y_{i}))|\leq
m\cdot var^{\square }(f)+||f||_{\updownarrow lip}var^{\square }(G)+2m||f||_{\infty
!}.
\end{equation*}%
This gives
\begin{eqnarray*}
var^{\square }(f\circ F) &\leq &m\cdot var^{\square }(f)+var^{\square
}(G)||f||_{\updownarrow lip}+2m||f||_{\infty !} 
\\
var^{\square }(f\circ F^{2}) &\leq &m\cdot var^{\square }(f\circ F)+var^{\square
}(G)||f\circ F||_{\updownarrow lip}+2m||f\circ F||_{\infty
  !} 
\\
&=&m(m\cdot var^{\square }(f)+var^{\square }(G)||f||_{\updownarrow
lip}+2m||f||_{\infty })+var^{\square }(G)||f||_{\updownarrow
lip}+2m||f||_{\infty !} \\
var^{\square }(f\circ F^{n}) &\leq &m^{n}var^{\square
}(f)+(m^{n-1}+...+m)(var^{\square }(G)||f||_{\updownarrow
lip}+2m||f||_{\infty !})
\end{eqnarray*}
and the statement is proved.
\end{proof}
By Theorem~\ref{thm:summary-exp-decay} and equation
\eqref{decay}, we obtain exponential decay of correlations
with suitable norms for a class of maps containing the ones found as Poincar\'e sections of singular hyperbolic flows.
\begin{theorem}
  \label{restre}Let $F:Q \rightarrow Q $ a Borel
  function such that $%
  F(x,y)=(T(x),G(x,y))$. Let $\mu$ be an invariant measure
  for $F$ with marginal $\mu _{x}$ on the $x$-axis (which is absolutely continuous and
  invariant for $T:\II\circlearrowleft$). Let us suppose
  that
\begin{itemize}
\item $F$ is a contraction on each vertical leaf: $G$ is $\lambda $%
-Lipschitz in $y$ with $\lambda <1$ for each $x$;
\item $var^{\square }(G)<\infty$;
\item $T:I\rightarrow I$ is piecewise monotonic, with $n+1$ $C^{1}$,
increasing branches on the intervals $(-\frac{1}{2},c_{1})$,...,$
(c_{i},c_{i+1})$ ,..., $(c_{n},\frac{1}{2})$ and $inf_x (T^{\prime }(x))>1$;
\item $\frac{1}{T^{\prime }}$ has finite universal $p-$bounded variation;
\item $\mu _{x}$  is weakly mixing.
\end{itemize}
Then, $F$ has exponential decay of correlations with respect
to suitable norms: there are $C,\Lambda \in
\mathbb{R}^{+},\Lambda <1$
\begin{equation*}
\left|\int g\circ F^{n} \cdot f \,d\mu - \int g\,d\mu \int
  f\,d\mu \right|
\leq C\Lambda^{n}\|g\|_{\updownarrow lip}
(\|f\|_{\updownarrow lip}+\|\pi(f)\|_{1,\frac{1}{p}}+var^{\square }(f)).
\end{equation*}
\end{theorem}
Since in next section we need to use decay of correlation with respect to
lipshitz observables, we show the following estimation.
\begin{lemma}
Let $f:Q\rightarrow \mathbb{R}$, then
\begin{equation*}
\|\pi (f)\|_{1,\frac{1}{p}}\leq 2 var^{\square }(f)\leq 2\|f\|_{lip}.
\end{equation*}
\end{lemma}

\begin{proof}
  Let $f:Q\to\RR$ be Lipschitz. By Lemma \ref{4} we have
  only to prove $var^{\square }(f)\leq \|f\|_{lip}$, but
  \begin{align*}
    var^{\square }(f,x_{1},...,x_{n},y_{1},...,y_{n})
    &=\sum_{1\le i\leq
      n}|f(x_{i},y_{i})-f(x_{i+1},y_{i})|\leq
    \|f\|_{lip}\sum_i |x_{i}-x_{i+1}| =\|f\|_{lip}.
  \end{align*}
\end{proof}
Altogether we obtain exponential decay of correlations with
respect to Lipschitz observables, as stated.
\begin{proposition}\label{dec-lip}
If $F$ satisfies the assumptions of Theorem \ref{restre} then it has
exponential decay of correlations with respect to Lipshitz observables:
there are $C,\Lambda \in \mathbb{R}^{+},~\Lambda <1$ such that
\begin{align*}
  \left|\int g\circ F^{n} \cdot f \,d\mu - \int g\,d\mu \int
    f\,d\mu \right|\leq C\Lambda^{n}\|g\|_{Lip}\|f\|_{Lip}.
\end{align*}
\end{proposition}

\begin{remark}\label{dec-sec}
  We have seen that if we consider a suitable iterate of  the first return  map on a section of a flow having a singular
  hyperbolic attractor, then by Theorem \ref{thm:propert-singhyp-attractor} and Remark \ref{eigen1}  all the needed assumptions are satisfied. Hence we conclude that this iterate has exponential decay of correlations over lipschitz observables for each of its physical invariant measures.
\end{remark}

%%%%%%%%%%%%%%%%%%%%%%%%%%%%%%%%%%%%%%%%%%%%%%%%%%%%

\section{Exact dimensionality}
\label{sec:muexata}

To establish the logarithm law for a singular hyperbolic
system, we need to establish that the local dimension exists
at a section of the system (see Proposition~\ref{maine}) .  We
recall and use a result of Steinberger \cite{Stein00} and
prove that for a singular hyperbolic system, under certain
general assumptions the local dimension is defined at almost
every point.

Let us consider a map $F:Q\to Q $,
$F(x,y)=(T(x),G(x,y))$ where
\begin{enumerate}
\item[(1)] $T:[0,1]\to [0,1]$ is piecewise monotonic: there
  are $c_i\in [0,1]$ for $0\leq i\leq N$ with $0=c_0< \cdots
  < c_N=1$ such that $T|(c_i,c_{I+1})$ is continuous and
  monotone for $0\leq i < N$. Furthermore, for $0\leq i< N$,
  $T|(c_i,c_{i+1})$ is $C^1$ and that $\inf_{x\in
    \cP}|T'(x)|> 0$
  holds where $\cP=[0,1]\setminus \cup_{0\leq i < N}c_i$.
\item[(2)] $G:Q\to (0,1)$ is $C^1$ on $\cP\times [0,1]$. Furthermore,
$\sup|\partial G/\partial x| < \infty$, $\sup|\partial G/\partial y| < 1$ and
$|(\partial G/\partial y)(x,y)| > 0$ for $(x,y)\in \cP\times [0,1]$.
\item[(3)] $F((c_i,c_{i+1})\times [0,1])\cap
  F((c_j,c_{j+1})\times [0,1])=\emptyset$ for distinct $i,
  j$ with $0\leq i, j < N$.
\end{enumerate}

% and prove a lemma concerning
% the derivatives of the first return map $F(x,y)=(f_{Lo}(x),g_{Lo}(x,y)$ defined for the
% geometric Lorenz attractor constructed before.
%
% First identify the cross section to the geometric Lorenz flow
% construct in Section \ref{sec:constr-geometr-model} with the square $I^2=[-1/2,1/2]^2$,
% and recall that the first return map associated $F$ can be written as
% $$F(x,y)=(f_{Lo}(x),g_{Lo}(x,y)), \quad (x,y) \in [-1/2,1/2]^2\setminus \{(0,y), 0\leq y \leq 1\}.$$
% The properties of the maps $g_{Lo}$ and $f_{Lo}$ described in Subsections \ref{asegundacoordenada}
% and \ref{sec:propert-one-dimens} guaranty that our map $F$ defined above
% is a {\em 2-dimensional} Lorenz map in the sense defined by Steinberg at
% \cite[pg912]{S00}.

Now consider the projection $\pi_{x}:Q\to I$, set ${\cal
  V}=\{(c_i,c_{i+1}), 1\leq i \leq N \}$ and
$\cV_k=\bigvee_{i=0}^{k}T^{-i}\cV$.  For $x\in E$ let
$J_k(x)$ be the unique element of $\cV_k$ which contains
$x$.  We say that $\cV$ {\em is a generator} if the length
of the intervals $J_k(x)$ tends to zero for $n \to \infty$
for any given $x$. In piecewise expanding maps it is easy to
see that ${\cal V}$ is a generator.  Set
\begin{align}\label{eq:log-psi-vfi}
  \psi(x,y)=\log|T'(x)| \quad \mbox{and} \quad
  \varphi(x,y)=-\log|(\partial G/\partial y)(x,y)|.
\end{align}
The result of Steinberger that we shall use is the following

\begin{theorem}\cite[Theorem 1]{Stein00}
  \label{th:Steinberg} Let $F$ be a two-dimensional map as
  above and $\mu_F$ an ergodic, $F$-invariant probability
  measure on $Q$ with the entropy $h_{\mu}(F)>0$.  Suppose
  $\cV$ is a generator, $\int\varphi\cdot d\mu_F < \infty$
  and $0<\int \psi d\mu_F < \infty$.  If the maps $y\mapsto
  \varphi(x,y)$ are uniformly equicontinuous for $x \in
  I\setminus\{0\}$ and $1/|T'|$ has finite universal $p$- Bounded
  Variation, then
$$
d_{\mu}(x,y)=h_{\mu}(F)\left(\frac{1}{\int\psi\cdot d\mu} +
\frac{1}{\int \varphi\cdot d\mu}\right)
$$
for $\mu$-almost all $(x,y) \in Q$.
\end{theorem}

Now, in the systems we consider, item (3) above is satisfied
because the map is induced by a first return
Poincar\'e map induced by a flow; see
Remark~\ref{rmk:injective0} and
Definition~\ref{def:global-poincare-map} in
Section~\ref{sec:cross-sections-poinc}.  Moreover
$\sup|\partial G/\partial x| < \infty$ in item (2) above is
established at item (6) of Theorem
\ref{thm:propert-singhyp-attractor}, provided that for all
equilibria $\sigma\in\Lambda$ we have the eigenvalue
relation $-\lambda_2(\sigma)>\lambda_1(\sigma)$.  Let us
also observe that, for the first return map $F:Q\setminus
\Gamma\to Q$, associated to the singular-hyperbolic flow,
the entropy is positive $h_\mu(F) > 0$; see
Section~\ref{sec:positive-entropy-two}.

So, all we need to prove that $(\Xi,F,d\mu_F)$ is exact
dimensional is to verify that $F(x,y)$ satisfies the
hypothesis of Theorem \ref{th:Steinberg}. However,
Proposition~\ref{pr:logTprimeDyG-int} provides precisely
that for the functions $\vfi,\psi$ defined above in
(\ref{eq:log-psi-vfi}): we have
\begin{enumerate}
 \item $\int \varphi d\mu^i_F < \infty$;
\item $0< \int \psi d\mu^i_F < \infty$; and
\item the maps $y \mapsto \varphi(x,y)$ are uniformly equicontinuous
for $x\in\II\setminus\{c_1,\dots,c_n\}$;
\end{enumerate}
where $\mu^i_F$ is each one of the invariant ergodic SRB measure described
in Section~\ref{sec:Lorenzmodel}.

  This all together completes what is necessary to use
  Theorem~\ref{th:Steinberg}, establishing that each $\mu^i_F $ is
  exact dimensional.

  The exact dimensionality of the measure on the section
  implies the exact dimensionality of the measure $\mu$ on
  the flow at almost each point, and the dimension satisfies
  $d_{\mu}(x)=d_{\mu_F}(x)+1$ at almost every point $x$.

%%%%%%%%%%%%%%%%%%%%%%%%%%%%%%%%%%%%%%%%%%%%%%%%%%%%%%%%%%%

\section{Logarithm law for singular hyperbolic attractors}
\label{sec:logarithm-law-syngul}

In this section we prove the logarithm law for singular
hyperbolic flows. More precisely we prove

\begin{theorem}\label{thm:exactdim-loglaw}
  Let $X_t:M\circlearrowleft$ be a $C^2$ flow having a
  singular hyperbolic attractor $\Lambda$ satisfying the
  nonresonance condition in Theorem~\ref{thm:smooth-linear}
  and $-\lambda_2(\sigma)>\lambda_1(\sigma)$ at each fixed
  point $\sigma$ of the flow in $\Lambda$. Let us consider
  its physical invariant measure $\mu$.  Let us consider
  $x_0$ and the local dimension at $x_0$
\begin{equation}
  d_{\mu}(x_0)=\lim_{r\rightarrow 0}\frac{\log \mu (B_{r}(x_0))}{\log r},
\end{equation}
(which was above proved to exist almost everywhere), then
for $\mu$ almost every $x$
\begin{equation}
\lim_{r\rightarrow 0}\frac{\log \tau (x,B_{r}(x_0))}{-\log r}=d_{\mu}(x_0) -1
\end{equation}
where $ \tau (x,B_{r}(x_0))$ is the time needed for the
orbit of $x$ to hit the ball $B_{r}(x_0)$ as defined
in~(\ref{circlen}).
\end{theorem}

%Of course, noting that Proposition \ref{maine} holds for
%targets which are sublevels of Lipschitz functions, it is
%possible to give other statements, with the same methods,
%replacing balls with more general shrinking targets.

The remaining part of the section is devoted to the proof of
the Theorem \ref{thm:exactdim-loglaw}.  The strategy is to
first establish the law for the Poincar\'e map associated to
the flow, and then extend it to the flow itself.  The first
step is based on a result about discrete time systems which
we recall: let $(M,F,\mu _F )$ be an ergodic, measure
preserving transformation on a metric space.  In this
setting the following is proved in \cite{galat07} ( see also
\cite{galatolo10} for a generalization to targets different
than balls).
\begin{proposition}
  \label{maine}For each $x_0$ 
\begin{equation}\label{easy}
  \lim \sup_{r\rightarrow 0}\frac{\log \tau _{F}(x,B_r (x_0))}{-\log r}\geq
  \overline{d}_{\mu _F} (x_0)~,~\lim \inf_{r\rightarrow 0}\frac{\log \tau _{F}(x,B_r (x_0))}{
    -\log r}\geq \underline{d}_{\mu _F} (x_0)
\end{equation}
hold for   $\mu _F$-almost every $x$.

Moreover, if the system has super-polynomial decay of
correlations under Lipschitz observables and $d_{\mu} (x_0)$ exists,
then for $\mu$-almost every $x$ it holds
\begin{equation}\label{log-law-disc-time}
\lim_{r\rightarrow 0}\frac{\log \tau _{F}(x,B_r (x_0))}{-\log r}=d_{\mu _F} (x_0).
\end{equation}
\end{proposition}

\begin{remark}\label{log-sec}
  We show how this result can be applied to deduce a
  logarithm law for the Poincar\'e map of singular
  hyperbolic systems as described in Theorem
  \ref{thm:propert-singhyp-attractor}, which is a suitable
  iteration of the first return map and then to the first
  return map.  By Proposition \ref{dec-lip} and Remark
  \ref{dec-sec}, we know that the Poincar\'e map we consider
  has super-polynomial decay of correlations for Lipshitz
  observables with respect to each physical invariant
  measure $\mu ^i_F$. Since in the previous section we
  proved exact dimensionality of those measures, we can
  apply Proposition \ref{maine} to our Poincar\'e map.

  More precisely, to apply this result to the Poincar\'e map
  $F$ and establish the logarithm law, suppose $x_0$ is in
  the basin of $\mu ^i_F$ and consider the system $(Q,\mu
  ^i_F,F)$. Now note that since we are dealing with a ratio
  of logarithms, and \eqref{easy} always hold, if we
  establish the logarithm law \eqref{log-law-disc-time} for
  some iterate $F=F^n_0$, then it will hold also for $F_0$;
  indeed, applying Proposition \ref{maine}, we know that
  there is a set $A$ with $\mu_i (A)=1$ such that
  $\limsup_{r\rightarrow 0}\frac{\log \tau _{F}(x,B_r
    (x_0))}{-\log r}\leq d_{\mu^i_F} (x_0)$ then we also
  have $\limsup_{r\rightarrow 0}\frac{\log \tau _{F_0}(x,B_r
    (x_0))}{-\log r}\leq d_{\mu^i_F} (x_0)$ for each $x\in
  A$. But this is also true for each $x\in F_0^{-i}(A) $ for
  each $i$, which by ergodicity of the first return map
  $F_0$ covers a full measure set for the invariant measure
  $\mu_{F_0} $ of the first return map $F_0$.  Finally we
  remark that, since each $\mu^i_F $ is exact dimensional,
  and each $\mu^i_F $ give rise to the ergodic physical
  measure $\mu$ of the flow by suspension, then the
  a.e. local dimension of each $\mu^i_F $ will be the same
  for each $i$.

\end{remark}

%%%%%%%%%%%%%%%%%%%%%%%%%%%%%%%%%%%%%%%%%%%%%%%%%

\subsection{Logarithm law for the flow}
\label{sec:logarithm-law-flow}

In Remark \ref{log-sec} we showed that, on the section, the
logarithm law holds for the first return map. Let us extend
this to the flow.  We consider a general measure preserving
flow and note that, just like in discrete time systems, one
inequality between hitting time behavior and dimension of
the measure is valid in general (see \cite[Remark
2.4]{GalNis11}).
 
\begin{proposition}\label{easyineqflow}
  Let us consider a $C^{1}$ flow preserving the measure $\mu
  $, then for each $x_{0}$ where $d_{\mu }(x_{0})$ is
  defined
\begin{equation}
\underset{r\rightarrow 0}{\lim \sup }
\frac{\log \tau (x,B_{r}(x_{0}))}{-\log
r} \geq \underset{r\rightarrow 0}{\lim \inf }
\frac{\log \tau (x,B_{r}(x_{0}))}{-\log
r} \geq d_{\mu }(x_{0})-1
\end{equation}
hold for amost each $x$.
\end{proposition}
The other inequality can be established by the behavior of
the system on a section, as considered before.  Let $\Sigma$
be a section of a measure preserving flow $(M,\Phi^{t})$.
We will show that, if the flow is ergodic and the return
time is integrable, then the hitting time scaling behavior
of the flow can be estimated by the one of the system
induced on the section.

Given any $x\in X$, let us denote by $t(x)$ the smallest
strictly positive time such that $\Phi ^{t(x)}(x)\in
\Sigma$. We also consider $t^{\prime }(x)$, the smallest non
negative time such that $\Phi^{t^{\prime }(x)}(x)\in
\Sigma$. We remark that these two times differ on the
section $\Sigma$, where $t^{\prime }=0$, while $t$ is the
return time to the section.  We define $\pi:X\rightarrow
\Sigma$ as $\pi(x)=\Phi ^{t^{\prime }(x)}(x)$, the
projection on $\Sigma $. As before denote by $\mu _{F_0}$
the invariant measure for the first return Poincar\'{e} map
$F_0$ which is induced by the invariant measure $\mu$ of the
flow.

\begin{proposition}[see \cite{GalNis11}]
  \label{4log}
  Let us suppose that the flow $\Phi ^{t}$ is ergodic and
  has a transverse section $\Sigma $ with an induced first
  return map $F_0$ and an ergodic invariant measure $\mu
  _{F_0}$ such that
\begin{equation*}
\int_{\Sigma }t(x)~d\mu _{F_0}<\infty .
\end{equation*}
Let $r\geq 0$ and $B_{r}(x_0)\subseteq \Sigma $ be balls
centered in $x_0$, with $\lim_{r\rightarrow 0}\mu
_{F_0}(B_{r} )=0$.  Then, there is a full measure set
$C\subseteq \Sigma $ not depending on $x_0$, such that if
$x\in C$%
\begin{eqnarray}
\liminf_{r\rightarrow 0}\frac{\log \tau (x,B_{r}(x_0))}{-\log r}
&=&\liminf_{r\rightarrow 0}\frac{\log \tau _{F_0 }(\pi (x),B_{r}(x_0))}{-\log r%
}, \\
\limsup_{r\rightarrow 0}\frac{\log \tau (x,B_{r}(x_0))}{-\log r}
&=&\limsup_{r\rightarrow 0}\frac{\log \tau _{F_0}(\pi (x),B_{r}(x_0))}{-\log r%
}.
\end{eqnarray}
(We recall that $\tau(x,B_{r}(x_0))$ is the time
needed for the flow to take $x$ to $B_{r}(x_0)$ and $\tau
_{F_0 }(\pi (x),B_{r}(x_0))$ is the time needed for 
the induced map $F_0$ to take $x_0$ to the same set.)
\end{proposition}
Once we have the logarithm law for the first return map, by this proposition we can extend it to the
flow.  Indeed,  on the section there will be a set $C$ of points
where the logarithm law is satistied for each $x\in C$.  Let
us consider $A=\{ y\in M: \Phi^t(y)\in C \text{ for some }
t\geq 0 \}$. Then $\mu(A)=1$ (recall that $\mu $ is the
physical measure of the flow) because the flow is ergodic
and, by definition of $A$, for each $x\in A$ there is $x'\in
C$ such that $\tau (x,B_{r}(x_0))\leq \tau
(x',B_{r}(x_0))+\text{const.}$, where the constant represent
the time which is needed for $x$ to arrive in $x'$ by the
flow, and does not depend on $r$.
Hence the same logarithm
law which is satisfied on $C$ is satisfied on $A$, showing
that, for $x$ in the full measure subset of $A$
\begin{equation}
\underset{r\rightarrow 0}{\lim \inf }\frac{\log \tau (x,B_{r}(x_{0}))}{-\log
r} \leq \underset{r\rightarrow 0}{\lim \sup }\frac{\log \tau (x,B_{r}(x_{0}))}{-\log
r} \leq d_{\mu }(x_{0})-1.
\end{equation}
A similar construction can be done if the target point $x_0$
is not on the section, and extend the result to $x_0 \in M$.
The other inequality is provided by Proposition
\ref{easyineqflow}. All together this proves the logarithm
law for the flow, once it was proved on the section. This
was done in Remark \ref{log-sec}. Hence we have established the
logarithm law for the flow stated in Theorem
\ref{thm:exactdim-loglaw}.

%\subsubsection{Logarithm law for singular-hyperbolic attractors}
%\label{sec:logarithm-law-singul}

\section{Notation of norms used throughout the paper}
\label{sec:notati-norms-used}

To help the reader, in this section we give a list of notations related to the various norms which are used in the paper and some explanations to explain their role in the construction.

\subsection{Custom norms used}
\label{sec:custom-norms-used}

\begin{description}
\item[$\|\cdot\|_{\_}$] the "regularity" norm on the base (in a skew product the norm for which convergence to equilibrium is established for the base transformation).

\item[$\|\cdot\|_{\updownarrow lip}$] norm evaluating the Lipschitz constant in
the vertical direction.

\item[$\|\cdot\|_{\square }$] a seminorm on the square having some suitable
properties with respect to the previous two above norms (see Theorem \ref{resdue} ).

\item[$var^{\square }$] a sort of bounded variation seminorm on the square
(but looking only on increments on the horizontal direction) which is an
example of $\|\cdot\|_{\square }$ (for the definition, see beginning of Section \ref{sec:p-bounded-variation} ).

\item [$\|\cdot\|_{\infty !}$] the sup norm on all points
  (not neglecting zero Lebesgue measure sets, this norm will
  be used where we have to consider both Lebesgue and the
  invariant measure, which are singular with respect to each
  other).
\end{description}

\subsection{Other norms we use or mention}
\label{sec:other-norms-we}

\begin{description}
\item [$\|\cdot\|_{\infty }$] the usual $L^{\infty }$ norm with respect to
Lebesgue measure.

\item [$\|\cdot\|_{p}$] the usual $L^{p}$ norm with respect to Lebesgue measure.

\item [$\|\cdot\|_{1}$] the usual $L^{1}$ norm with respect to Lebesgue measure.

\item [$\|\cdot\|_{p,r}$] the generalized bounded variation norm with respect to
Lebesgue measure (see Section \ref{sec:p-bounded-variation}).

\item [$\|\cdot\|_{Lip}$] the usual Lipschitz norm.
\end{description}

%%%%%%%%%%%%%%%%%%%%%%%%%%%%%%%%%%%%%%%%%%%%%%%%%%%%%%%%%%

\def\cprime{$'$}

%\bibliographystyle{abbrv}
%\bibliography{../bibliobase/bibliography}

\end{document}